\documentclass[12pt]{amsart}

\usepackage{amssymb}

\usepackage{enumerate}

\makeatletter
\@namedef{subjclassname@2010}{
  \textup{2010} Mathematics Subject Classification}
\makeatother

\newtheorem{thm}{Theorem}[section]
\newtheorem{theorem}[thm]{Theorem}
\newtheorem{proposition}[thm]{Proposition}
\newtheorem{lemma}[thm]{Lemma}

\theoremstyle{definition}

\numberwithin{equation}{section}

\begin{document}


\baselineskip=17pt



\title[Weak invariance principle for dependent random variables]{On the rate of convergence in the weak invariance principle for dependent random variables with applications to Markov chains}

\author[I.~Grama]{Ion~Grama}
\address[I. Grama]{ Universit\'{e} de Bretagne Sud, LMBA, Campus de
Tohannic, BP 573, 56017 Vannes cedex, France}
\email{ion.grama@univ-ubs.fr}

\author[E.~Le~Page]{Emile~Le~Page}
\address[E. Le Page]{ Universit\'{e} de Bretagne Sud, LMBA, Campus de
Tohannic, BP 573 56017 Vannes cedex, France}
\email{emile.lepage@univ-ubs.fr}

\author[M.~Peign\'{e}]{Marc~Peign\'{e}}
\address[M. Peign\'{e} ]{ Universit\'{e} F. Rabelais Tours cedex, LMPT, Parc de
Grandmont, 37200 Tours, France}
\email{peigne@lmpt.univ-tours.fr}

\maketitle


\begin{abstract}
We prove an invariance principle for non-stationary random processes and
establish a rate of convergence under a new type of mixing condition. The
dependence is exponentially decaying in the gap between the past and the
future and is controlled by an assumption on the characteristic function of
the finite dimensional increments of the process. The distinct feature of
the new mixing condition is that the dependence increases exponentially in
the dimension of the increments. The proposed mixing property is
particularly suited for processes whose behavior can be described in terms
of spectral properties of some related family of operators. Several examples
are discussed. We also work out explicit expressions for the constants
involved in the bounds. When applied to Markov chains our result specifies
the dependence of the constants on the properties of the underlying Banach
space and on the initial state of the chain.
\end{abstract}



\section{\label{Intro}Introduction}

Let $\left( X_{k}\right) _{k\geq 1}$ be a sequence of real valued random
variables (r.v.'s) defined on the probability space $\left( \Omega, \mathcal{%
F}, \mathbb{P}\right) $ and $S_{N}(t)=N^{-1/2}\sum_{k=1}^{\left[ Nt\right]}X_{k}, $ $t\in \lbrack 0, 1].$
The (weak) invariance principle states that the process
$ \frac{1}{\sqrt{N}} \left( S_{N}(t)\right) _{0\leq t\leq 1}$ converges weakly to the Brownian process
$ \left( W(t)\right) _{0\leq t\leq 1}$ and is a powerful tool for various
applications in probability and statistics. It extends the scope of the
central limit theorem to continuous functionals of the stochastic process $%
\left( S_{N}(t)\right) _{0\leq t\leq 1}, $ such as,  for example,  the maxima
or the $L^{2}$-norm of the trajectory of the process,  considered in the
appropriate functional spaces. The rates of convergence in the (weak)
invariance principle,  for independent r.v.'s under the existence of the moments of order $2+2\delta, $ with $\delta >0, $
have been obtained in Prokhorov
\cite{Prokhorov54},  Borovkov \cite{Bor73},  Koml\'{o}s,  Major and Tusn\'{a}dy
\cite{KMT},  Einmahl \cite{Einm1989},  Sakhanenko \cite{Sakh83},  \cite{Sakh85},
Zaitsev \cite{Zai98}, \cite{Zai07} among others. In the
case of martingale-differences, for $\delta \leq \frac{1}{2}, $ the rates are essentially the same as in the
independent case (see,  for instance Hall and Heyde \cite{Hall Heyde book},
Kubilius \cite{Kub},  Grama \cite{Grama func89}).
The almost sure invariance principle is a reinforcement of the weak invariance principle
which states that the trajectories of a process are approximated with the trajectories of the Brownian motion a.s.\ in the sense that within
a particular negligible error $r_N \rightarrow 0$ it holds
$  \sup_{0\leq t\leq 1} \left| \frac{1}{\sqrt{N}} S_{N}(t) - W(t) \right| =O(r_N)$ a.s.
There are many recent
results concerning the rates of convergence in the strong invariance
principle for weakly dependent r.v.'s under various conditions. We refer to
Wu \cite{Wu},  Zhao and Woodroofe \cite{Zhao Woodr},  Liu and Lin \cite{Liu Lin},  Cuny \cite{Cuny2011}, Merlev\`{e}de and Rio \cite{Merl Rio},  Dedecker,
Doukhan and Merlev\`{e}de \cite{DedDoukhMerl2012} and to the references
therein. However,  in contrast to the case of independent r.v.'s where it is found that
the optimal rate is of order $N^{-\frac{\delta }{2+2\delta }} $ for the strong invariance principle and
$N^{-\frac{\delta }{3+2\delta }} $ for the weak invariance principle,
the problem of obtaining the best rate of convergence in both the weak and strong invariance
principles for dependent variables is not yet settled completely.

Gou\"{e}zel \cite{Gouez} has introduced a new type of mixing condition which
is tied to spectral properties of the sequence $\left( X_{k}\right) _{k\geq
1}.$ Consider the vectors $\overline{X}_{1}=$ $\left(
X_{J_{1}}, ..., X_{J_{M_{1}}}\right) $ and $\overline{X}_{2}=$ $\left(
X_{k_{gap}+J_{M_{1}+1}}, ..., X_{k_{gap}+J_{M_{1}+M_{2}}}\right), $ called the
past and the future,  respectively,  where $X_{k+J_{m}}=\sum_{l\in
J_{m}}X_{k+l}, $ $J_{m}=[j_{m-1}, j_{m}), $ $j_{0}\leq...\leq j_{M_{1}+M_{2}}, $
and $k_{gap}$ is a gap between $\overline{X}_{1}$ and $\overline{X}_{2}.$
Roughly speaking,  the condition used in \cite{Gouez} suppose that the
characteristic function of $\left( \overline{X}_{1}, \overline{X}_{2}\right) $
is exponentially closed to the product of the characteristic functions of
the past $\overline{X}_{1}$ and the future $\overline{X}_{2}, $ with an error
term of the form $A\exp \left( -\lambda k_{gap}\right), $ where $\lambda $
is some non negative constant and $A$ is exponential in terms of the size of
the blocks.
This mixing property is particularly suited for systems whose behavior can
be described in terms of spectral properties of some related family of
operators,  as initiated by Nagaev \cite{Nagaev57},  \cite{Nagaev61} and
Guivarc'h \cite{Guiv84}. Examples are Markov chains,  whose perturbed
transition probability operators $\left( \mathbf{P}_{t}\right) _{\left\vert
t\right\vert \leq \varepsilon _{0}}$ exhibit a spectral gap and enough
regularity in $t, $ and dynamical systems,  whose characteristic functions can
be coded by a family of operators $\left( \mathcal{L}_{t}\right)
_{\left\vert t\right\vert \leq \varepsilon _{0}}$ satisfying similar
properties. Gou\"{e}zel proves in \cite{Gouez} an almost sure invariance
principle with a rate of convergence of order $N^{-\frac{\delta }{2+4\delta }}.$

The scope of the present paper is to improve on the results of Gou\"{e}zel by quantifying the rate of convergence in the (weak) invariance principle
for dependent r.v.'s under the mixing condition
introduced above.
Although the strong and weak invariance principles are
closely related,  it seems that the rate of convergence in the (weak)
invariance principle is less studied under weak dependence constraints. We
refer to Doukhan,  Leon and Portal \cite{DoukhLeonPort},  Merlev\`{e}de and
Rio \cite{Merl Rio} and Grama and Neumann \cite{GrNeumann2006}. However,
these results rely on mixing conditions which are not verified in the
present setting. Under the above mentioned mixing and some further mild
conditions including the moment assumption $\sup_{k\geq 1}\mathbb{E}%
\left\vert X_{k}\right\vert ^{2+2\delta }<\infty $ we obtain a bound of the
order $N^{-\frac{1+\alpha }{1+2\alpha }\frac{\alpha }{3+2\alpha }}, $ for any
$\alpha <\delta.$ Moreover,  we give explicit expressions of some constants
involved in the rate of convergence; for instance,  in the case of Markov
walks we are able to figure out the dependence of the rate of convergence on
the properties of the Banach space related to the corresponding family of
perturbed transition operators $\left( \mathbf{P}_{t}\right) _{\left\vert
t\right\vert \leq \varepsilon _{0}}$ and on the initial state $X_{0}=x$ of
the associated Markov chain. When compared with the rate $N^{-\frac{1}{2}%
\frac{\alpha }{1+2\alpha }}$ in the almost sure invariance principle of \cite%
{Gouez} ours appears with a loss in the power of multiple $\frac{2+2\alpha }{%
3+2\alpha }<1.$ This loss in the power is exactly the same as in the case of
independent r.v.'s,  when we compare the almost sure invariance principle
(rate $N^{-\frac{\delta }{2+2\delta }}$) and the (weak) invariance principle
(rate $N^{-\frac{\delta }{3+2\delta }}$).

As in the paper \cite{Gouez} our proof relies on a progressive
blocking technique (see Bernstein \cite{bernst27}) coupled with a triadic
Cantor like scheme and on the Koml\'{o}s,  Major and Tusn\'{a}dy
approximation type results for independent r.v.'s (see \cite{KMT},  \cite%
{Einm1989},  \cite{Zai98}),  which is in contrast to approaches usually based
on martingale methods.

As a potential application of the obtained results we point out the
asymptotic equivalence of statistical experiments as developed in \cite%
{GrNuss1998},  \cite{GrNuss2001},  \cite{GrNeumann2006},  whose scope can be
extended to various models under weak dependence constraints.

Our paper is organized as follows. In Sections \ref{sec Intro Res} and \ref%
{sec Applic} we formulate our main results and give an application to the
case of Markov chains. In Section \ref{sec Notations} we introduce the
notations to be used in the proofs of the main results. Proofs of the main
results are given in Sections \ref{sec Aux Res},  \ref{sec Proof Main Res} and \ref{sec Proofs res Sec3}.
In Section \ref{sec Lp bounds} we prove some bounds for the $L^{p}$ norm of the increments of the
process $\left( X_{k}\right) _{k\geq 1}$ and, finally,
in Section \ref{sec Appendix} we collect some
auxiliary assertions and general facts.

We conclude this section by setting some notations to be used all over the
paper. For any $x\in \mathbb{R}^{d}, $ denote by $\left\Vert x\right\Vert
_{\infty }=\sup_{1\leq i\leq d}\left\vert x_{i}\right\vert $ the supremum
norm. For any $p>0, $ the $L^{p}$ norm of a random variable $X$ is denoted by
$\left\Vert X\right\Vert _{L^{p}}.$
The equality in distribution of two
stochastic processes $\left( Z_{i}^{\prime }\right) _{i\geq 1}$ and $\left(
Z_{i}^{\prime \prime }\right) _{i\geq 1}$,  possibly defined on two different
probability spaces $\left( \Omega ^{\prime },  \mathcal{F}^{\prime},  \mathbb{P}%
^{\prime }\right) $ and
$\left( \Omega ^{\prime \prime },  \mathcal{F}^{\prime\prime}, \mathbb{P}^{\prime \prime }\right), $ will be denoted
$\mathcal{L}\left( \left(
Z_{i}^{\prime }\right) _{i\geq 1}|\mathbb{P}^{\prime }\right) \overset{d}{=}%
\mathcal{L}\left( \left( Z_{i}^{\prime \prime }\right) _{i\geq 1}|\mathbb{P}%
^{\prime \prime }\right).$
The generalized inverse of a distribution function $F$ on a real line is denoted by $F^{-1},$ i.e. $F^{-1}\left( y\right) =\inf \left\{ x:F\left( x\right) >y\right\} .$
By $c, c^{\prime }, c^{\prime
\prime } ...$,  possibly supplied with indices $1, 2, ...$,  we denote absolute constants
whose values may vary from line to line. The notations $c_{\alpha
_{1}, ..., \alpha _{r}}, c_{\alpha _{1}, ..., \alpha _{r}}^{\prime }...$ will be
used to stress that the constants depend only on the parameters indicated in
their indices: for instance $c_{\alpha, \beta }^{\prime }$ denotes a
constant depending only on the constants $\alpha, \beta.$ All other
constants will be specifically indicated. As usual,  we shall use the
shortcut "standard normal r.v." for a normal random variable of mean $0$ and
variance $1.$

\noindent {  ACKNOWLEDGMENTS.}
We wish to thank   the referee for many useful remarks  and comments; his suggestions have really improved the structure of this manuscript.
\section{Main results\label{sec Intro Res}}

Assume that on the probability space $\left( \Omega, \mathcal{F}, \mathbb{P}%
\right) $ we are given a sequence $\left( X_{i}\right) _{i\geq 1}$ of
dependent r.v.'s with values in the real line $\mathbb{R}.$ The expectation
with respect to $\mathbb{P}$ is denoted by $\mathbb{E}.$

The following condition will be used to ensure that the process $\left(
X_{i}\right) _{i\geq 1}$ has almost independent increments. Given natural
numbers $k_{gap}, M_{1}, M_{2}\in \mathbb{N}$ and a sequence $j_{0}\leq
...\leq j_{M_{1}+M_{2}}$ denote $X_{k+J_{m}}=\sum_{l\in J_{m}}X_{k+l}, $
where $J_{m}=[j_{m-1}, j_{m}), $ $m=1, ..., M_{1}+M_{2}$ and $k\geq 0.$ Consider
the vectors $\overline{X}_{1}=$ $\left( X_{J_{1}}, ..., X_{J_{M_{1}}}\right) $
and
$\overline{X}_{2}=$ $\left(X_{k_{gap}+J_{M_{1}+1}}, ..., \right.$ $\left. X_{k_{gap}+J_{M_{1}+M_{2}}}\right).$
Let $\phi
\left( t_{1}, t_{2}\right) =\mathbb{E}e^{it_{1}\overline{X}_{1}+it_{2}%
\overline{X}_{2}}, $ $\phi _{1}\left( t_{1}\right) =\mathbb{E}e^{it_{1}%
\overline{X}_{1}}$ and $\phi _{2}\left( t_{2}\right) =\mathbb{E}e^{it_{2}%
\overline{X}_{2}}$ be the characteristic functions of $\left( \overline{X}%
_{1}, \overline{X}_{2}\right), $ $\overline{X}_{1}$ and $\overline{X}_{2}$
respectively. We ask that the dependence between the two vectors $\overline{X%
}_{1}$ (the past) and $\overline{X}_{2}$ (the future) decreases
exponentially as a function of the size of the gap $k_{gap}$ in the
following sense.

\vskip0.2cm \noindent \textbf{Condition C1} \label{Cond Dependence}\textit{%
There exist positive constants }$\varepsilon _{0}\leq 1, $\textit{\ }$\lambda
_{0}, \lambda _{1}, \lambda _{2}$\textit{\ such that for any }$%
k_{gap}, $ $M_{1}, $ $M_{2}\in N, $\textit{\ any sequence }$j_{0}<...<j_{M_{1}+M_{2}}$%
\textit{\ and any }$t_{1}\in R^{M_{1}}, t_{2}\in R^{M_{2}}$\textit{\
satisfying }$\left\Vert \left( t_{1}, t_{2}\right) \right\Vert _{\infty }\leq
\varepsilon _{0}, $%
\begin{equation*}
\left\vert \phi \left( t_{1}, t_{2}\right) -\phi _{1}\left( t_{1}\right) \phi
_{2}\left( t_{2}\right) \right\vert \leq \lambda _{0}\exp \left( -\lambda
_{1}k_{gap}\right) \left( 1+\max_{m=1, ..., M_{1}+M_{2}}card\left(
J_{m}\right) \right) ^{\lambda _{2}\left( M_{1}+M_{2}\right) }.
\end{equation*}

All over the paper we suppose that the following moment conditions hold true.

\vskip0.2cm \noindent \textbf{Condition C2}\textit{\ There exist two
constants }$\delta >0$\textit{\ and }$\mu _{\delta }>0$\textit{\ such that }%
\begin{equation*}
\sup_{i\geq 1}\left\Vert X_{i}\right\Vert _{L^{2+2\delta }}\leq \mu _{\delta
}<\infty.
\end{equation*}

We suppose also that the sequence $\left( X_{i}\right) _{i\geq 1}$ possesses
the following asymptotic homogeneity property.

\vskip0.2cm \noindent \textbf{Condition C3} \textit{There exist a constant }$%
\tau >0$\textit{\ and a positive number }$\sigma >0$\textit{\ such that,  for
any }$\gamma >0$\textit{\ and any }$n\geq 1, $\textit{\ }%
\begin{equation*}
\sup_{k\geq 0}\left\vert n^{-1}\mathrm{Var}_{\mathbb{P}}\left(
\sum_{i=k+1}^{k+n}X_{i}\right) -\sigma ^{2}\right\vert \leq \tau \
n^{-1+\gamma }.
\end{equation*}

The main result of the paper is the following theorem. Denote $\mu _{i}=%
\mathbb{E}X_{i}, $ for $i\geq 1.$

\begin{theorem}
\label{Th main res 1}Assume Conditions \textbf{C1},  \textbf{C2} and \textbf{%
C3}. Let $0<\alpha <\delta.$ Then,  on some probability space $(\widetilde{%
\Omega }, \widetilde{\mathcal{F}}, \widetilde{\mathbb{P}})$ there exist
a sequence of random variables $(\widetilde{X}%
_{i})_{i\geq 1}$ such that $\mathcal{L}\left( (\widetilde{X}_{i})_{i\geq 1}|%
\widetilde{\mathbb{P}}\right) \overset{d}{=}\mathcal{L}\left( (X_{i})_{i\geq
1}|\mathbb{P}\right)$ and  a
sequence of independent standard normal random variables $\left(
W_{i}\right) _{i\geq 1}$,  such that  for any $0<\rho <\frac{\alpha }{2\left( 1+2\alpha
\right) }$ and $N\geq 1, $
\begin{equation*}
\widetilde{\mathbb{P}}\left( N^{-1/2}\sup_{k\leq N}\left\vert
\sum_{i=1}^{k}\left( \widetilde{X}_{i}-\mu _{i}-\sigma W_{i}\right)
\right\vert >6N^{-\rho }\right) \leq C_{0}N^{-\alpha \frac{1+\alpha }{%
1+2\alpha }+\rho \left( 2+2\alpha \right) },
\end{equation*}%
where $C_{0}=c_{\lambda _{1}, \lambda _{2}, \alpha, \delta, \sigma }\left(
1+\lambda _{0}+\mu _{\delta }+\sqrt{\tau }\right) ^{2+2\delta }$ and $%
c_{\lambda _{1}, \lambda _{2}, \alpha, \delta, \sigma }$ depends only on the
constants indicated in its indices.
\end{theorem}

Letting $\rho =\frac{\alpha }{3+2\alpha }\frac{1+\alpha }{1+2\alpha }$ from
Theorem \ref{Th main res 1} we get%
\begin{equation}
\widetilde{\mathbb{P}}\left( N^{-1/2}\sup_{k\leq N}\left\vert
\sum_{i=1}^{k}\left( \widetilde{X}_{i}-\mu _{i}-\sigma W_{i}\right)
\right\vert >6N^{-\frac{\alpha }{3+2\alpha }\frac{1+\alpha }{1+2\alpha }%
}\right) \leq C_{0}N^{-\frac{\alpha }{3+2\alpha }\frac{1+\alpha }{1+2\alpha }%
},   \label{corr main res 1}
\end{equation}%
where $C_{0}=c_{\lambda _{1}, \lambda _{2}, \alpha, \delta, \sigma }\left(
1+\lambda _{0}+\mu _{\delta }+\sqrt{\tau }\right) ^{2+2\delta }$ and $%
c_{\lambda _{1}, \lambda _{2}, \alpha, \delta, \sigma }$ depends only on the
constants indicated in its indices. Compared with the optimal rate of
convergence for independent r.v.'s $N^{-\frac{\alpha }{3+2\alpha }}$ the
loss in the power is within the factor $\frac{1+\alpha }{1+2\alpha }.$ As $%
\alpha \rightarrow \infty $ we obtain the limiting power $\frac{1}{4}$ which
is twice worse than the optimal power $\frac{1}{2}$ in the independent case.
In particular,  if $\alpha =\frac{1}{2}$ (which corresponds to $p=2+2\alpha
=3 $) we obtain the rate of convergence $N^{-\frac{\alpha }{3+2\alpha }\frac{%
1+\alpha }{1+2\alpha }}=N^{-\frac{3}{32}}, $ while in the independent case we
have $N^{-\frac{1}{8}}, $ which represents a loss of the power of the order $%
\frac{1}{8}-\frac{3}{32}=\frac{1}{32}.$

We would like to point out that in Theorem \ref{Th main res 1} we figure out
the explicit dependence of the constant $C_{0}$ on the constants $\lambda
_{0}, \mu _{\delta }$ and $\tau $ involved in Conditions \textbf{C1},  \textbf{%
C2} and \textbf{C3}. In the next section we show that Theorem \ref{Th main
res 1} can be applied to Markov walks under spectral gap type assumptions on
the associate Markov chain. It is important to stress that our result is the
first one to figure out the dependence of the constants involved in the
bounds on the initial state of the Markov chain. The results of the paper
can be also applied to a large class of dynamical systems,  however this
stays beyond the scope of the present paper. For a discussion of these type
of applications we refer to \cite{Gouez}.

For the proof of Theorem \ref{Th main res 1},  without loss of generality,  we
shall assume that $\mu _{i}=0, $ $i\geq 1$ and $\sigma =1, $ since the general
case can be reduced to this one by centering and renormalizing the
variables $X_{i}, $ i.e. by replacing $X_{i}$ by $X_{i}^{\prime }=\left(
X_{i}-\mu _{i}\right) /\sigma.$ It is easy to see that Conditions \textbf{C1%
},  \textbf{C2},  \textbf{C3} are satisfied for the new random variables $%
X_{i}^{\prime }$ with the same $\lambda _{0}$ and with $\mu _{\delta }, $ $%
\tau $ replaced by $2\mu _{\delta }/\sigma, $ $\tau /\sigma ^{2}.$

\section{Applications to Markov walks\label{sec Applic}}

Consider a Markov chain $\left( X_{k}\right) _{k\geq 0}$ with values in the
measurable state space $\left( \mathbb{X}, \mathcal{X}\right) $ with the
transition probability $\mathbf{P}\left( x, \cdot \right), $ $x\in \mathbb{X}%
. $ For every $x\in \mathbb{X}$ denote by $\mathbb{P}_{x}$ and $\mathbb{E}%
_{x}$ the probability measure and expectation generated by the finite
dimensional distributions%
\begin{equation*}
\mathbb{P}_{x}\left( X_{0}\in B_{0}, ..., X_{n}\in B_{n}\right)
=1_{B_{0}}\left( x\right) \int_{B_{1}}...\int_{B_{n}}\mathbf{P}\left(
x, dx_{1}\right)...\mathbf{P}\left( x_{n-1}, dx_{n}\right),
\end{equation*}%
for any $B_{k}\in \mathcal{X}, $ $k=1, ..., n, $ $n=1, 2, ...$ on the space of
trajectories $\left( \mathbb{X}, \mathcal{X}\right) ^{\mathbb{N}}.$ In
particular $\mathbb{P}_{x}\left( X_{0}=x\right) =1.$

Let $f$ be a real valued function defined on the state space $\mathbb{X}$ of
the Markov chain $\left( X_{k}\right) _{k\geq 0}.$ Let $\mathcal{B}$ be a
Banach space of real valued functions on $\mathbb{X}$ endowed with the norm $\left\Vert \cdot \right\Vert _{\mathcal{B}}$ and let $\Vert \cdot \Vert _{
\mathcal{B}\rightarrow \mathcal{B}}$ be the operator norm on\textit{\ }$
\mathcal{B}.$ Denote by $\mathcal{B}^{\prime }=\mathcal{L}\left( \mathcal{B}, \mathbb{C}\right) $
the topological dual of $\mathcal{B}$ equipped with the norm $\left\Vert \cdot \right\Vert _{\mathcal{B}'}.$ 
The unit function
on $\mathbb{X}$ is denoted by $e:$ $e\left( x\right) =1$ for $x\in \mathbb{X}.$
The Dirac measure at $x\in \mathbb{X}$  is denoted by $\boldsymbol{\delta }%
_{x}:$ $\boldsymbol{\delta }_{x}\left( g\right) =g\left( x\right) $ for any $%
g\in \mathcal{B}.$

Introduce the following hypotheses.

\vskip0.2cm \noindent \textbf{Hypothesis M1 (Banach space):}

\nobreak
\textit{a) The unit function }$e$\textit{\ belongs to }$\mathcal{B}.$

\textit{b)\ For every }$x\in X$\textit{\ the Dirac measure }$\boldsymbol{%
\delta }_{x}$\textit{\ belongs to }$\mathcal{B}^{\prime }.$

\textit{c)\ }$\mathcal{B}\subseteq L^{1}\left( \mathbf{P}\left( x, \cdot \right)
\right) $\textit{\ for every }$x\in X.$

\textit{d) There exists a constant }$\varepsilon _{0}\in \left( 0, 1\right) $%
\textit{\ such that for any }$g\in \mathcal{B}$\textit{\ the function }$e^{itf}g\in \mathcal{B}$%
\textit{\ for any }$t$\textit{\ satisfying }$\left\vert t\right\vert \leq
\varepsilon _{0}.$

Note that,  for any $x\in \mathbb{X}$ and $g\in L^{1}\left( \mathbf{P}\left(
x, \cdot \right) \right) $,  the quantity $\mathbf{P}g(x):=\int_{\mathbb{X}%
}g(y)\mathbf{P}(x, dy)$ is well defined. In particular,  under the hypothesis
\textbf{M1} c),  $\mathbf{P}g(x)$ exists when $g\in \mathcal{B}$. We thus
consider the following hypothesis:

\vskip0.2cm \noindent \textbf{Hypothesis M2 (Spectral gap): }

\nobreak
\textit{a) The map }$g\mapsto \mathbf{P}g$\textit{\ is a bounded operator on
}$\mathcal{B}$

\textit{b) There exist constants }$C_{Q}>0$\textit{\ and }$\kappa \in (0, 1)$%
\textit{\ such that}%
\begin{equation}
\mathbf{P}=\Pi +Q,   \label{spectr gap}
\end{equation}%
\textit{where }$\Pi $\textit{\ is a one dimensional projector and }$Q$%
\textit{\ is an operator on }$\mathcal{B}$\textit{\ satisfying }$\Pi Q=Q\Pi =0$\textit{%
\ and }$\left\Vert Q^{n}\right\Vert _{\mathcal{B}\rightarrow \mathcal{B}%
}\leq C_{Q}\kappa ^{n}$\textit{.}

Notice that,  since the image of $\Pi $ is generated by the unit function $e, $
there exists a linear form $\nu \in \mathcal{B}^{\prime }$ such that for any
$g\in \mathcal{B}, $%
\begin{equation}
\Pi g=\nu \left( g\right) e.  \label{linear form}
\end{equation}
When hypotheses $\mathbf{M1}$ and $\mathbf{M2}$ hold,  we set $\mathbf{P}%
_{t}g=\mathbf{P}(e^{itf}g)$ for any $g \in \mathcal{B}$ and $\ t \in
[-\varepsilon _{0},  \varepsilon _{0}]$. Notice that $\mathbf{P}=\mathbf{P}%
_{0}$.

\vskip0.2cm \noindent \textbf{Hypothesis M3 (Perturbed transition operator):}

\nobreak
\textit{a) For any }$\left\vert t\right\vert \leq \varepsilon _{0}$\textit{\
the map }$g\in \mathcal{B}\rightarrow P_{t}g\in \mathcal{B}$\textit{\ is a bounded operator on }$%
\mathcal{B}$\textit{, }

\textit{b) There exists a constant }$C_{\mathbf{P}}>0$\textit{\ such that,
for all }$n\geq 1$\textit{\ and }$\left\vert t\right\vert \leq \varepsilon
_{0}, $\textit{\ }%
\begin{equation}
\left\Vert \mathbf{P}_{t}^{n}\right\Vert _{\mathcal{B}\rightarrow \mathcal{B}%
}\leq C_{\mathbf{P}}.  \label{boundness of Ptn}
\end{equation}

\vskip0.2cm \noindent \textbf{Hypothesis M4 (Moment conditions):}\textit{\
There exists }$\delta >0$\textit{\ such that for any }$x\in X, $%
\begin{equation*}
\mu _{\delta }\left( x\right) :=\sup_{k\geq 1}\left( \mathbb{E}%
_{x}\left\vert f\left( X_{k}\right) \right\vert ^{2+2\delta }\right) ^{\frac{%
1}{2+2\delta }}=\sup_{k\geq 1}\left( \left( \mathbf{P}^{k}\left\vert
f\right\vert ^{2+2\delta }\right) \left( x\right) \right) ^{\frac{1}{%
2+2\delta }}<\infty.
\end{equation*}

We show first that under the hypotheses \textbf{M1},  \textbf{M2},  \textbf{M3}
and \textbf{M4},  conditions \textbf{C1,  C2} and \textbf{C3} are satisfied.
As in the previous section let $k_{gap}, M_{1}, M_{2}\in \mathbb{N}$ and $%
j_{0}\leq...\leq j_{M_{1}+M_{2}}$ be natural numbers. Denote $%
Y_{k+Jm}=\sum_{l\in J_{m}}f\left( X_{k+l}\right), $ where $%
J_{m}=[j_{m-1}, j_{m}), $ $m=1, ..., M_{1}+M_{2}$ and $k\geq 0.$ Consider the
vectors $\overline{Y}_{1}=$ $\left( Y_{J_{1}}, ..., Y_{J_{M_{1}}}\right) $ and
$\overline{Y}_{2}=$ $\left(
Y_{k_{gap}+J_{M_{1}+1}}, ..., Y_{k_{gap}+J_{M_{1}+M_{2}}}\right).$ Denote by $\phi
_{x}\left( t_{1}, t_{2}\right) =\mathbb{E}e^{it_{1}\overline{Y}_{1}+it_{2}%
\overline{Y}_{2}}, $ $\phi _{x, 1}\left( t_{1}\right) =\mathbb{E}_{x}e^{it_{1}%
\overline{Y}_{1}}$ and $\phi _{x, 2}\left( t_{2}\right) =\mathbb{E}%
_{x}e^{it_{2}\overline{Y}_{2}}$ the characteristic functions of $\left(
\overline{Y}_{1}, \overline{Y}_{2}\right), $ $\overline{Y}_{1}$ and $%
\overline{Y}_{2}$ respectively.

\begin{proposition}
\label{Proposition MC001}Assume that the Markov chain $(X_{n})_{n\geq 1}$
and the function $f$ satisfy the hypotheses \textbf{M1},  \textbf{M2},
\textbf{M3} and \textbf{M4}. Then Condition \textbf{C1} is satisfied,  i.e.
there exists a positive constant $\varepsilon _{0}\leq 1$ such that for any $%
k_{gap}, M_{1}, M_{2}\in \mathbb{N}, $ any sequence $j_{0}<...<j_{M_{1}+M_{2}}$
and any $t_{1}\in \mathbb{R}^{M_{1}}, t_{2}\in \mathbb{R}^{M_{2}}$ satisfying
$\left\Vert \left( t_{1}, t_{2}\right) \right\Vert _{\infty }\leq \varepsilon
_{0}, $%
\begin{eqnarray*}
\left\vert \phi _{x}\left( t_{1}, t_{2}\right) -\phi _{x, 1}\left(
t_{1}\right) \phi _{x, 2}\left( t_{2}\right) \right\vert &\leq&
 \lambda
_{0}\left( x\right) \exp \left( -\lambda _{1}k_{gap}\right)\\
& & \qquad \qquad \times  \left(
1+\max_{m=1, ..., M_{1}+M_{2}}card\left( J_{m}\right) \right) ^{\lambda
_{2}\left( M_{1}+M_{2}\right) },
\end{eqnarray*}%
where $\lambda _{0}\left( x\right) =2C_{Q}\left( \left\Vert \nu \right\Vert
_{\mathcal{B}^{\prime }}+\left\Vert \boldsymbol{\delta }_{x}\right\Vert _{%
\mathcal{B}^{\prime }}\right) \left\Vert e\right\Vert _{\mathcal{B}}, $ $%
\lambda _{1}=\left\vert \ln \kappa \right\vert $ and $\lambda _{2}=\max
\left\{ 1, \log _{2}C_{\mathbf{P}}\right\}.$
\end{proposition}

\begin{proposition}
\label{Proposition MC002}Assume that the Markov chain $(X_{n})_{n\geq 1}$
and the function $f$ satisfy the hypotheses \textbf{M1},  \textbf{M2},
\textbf{M3} and \textbf{M4}. Then:

a) There exists a constant $\mu $ such that for any $x\in \mathbb{X}$ and $%
k\geq 1$%
\begin{equation}
\left\vert \mathbb{E}_{x}f\left( X_{k}\right) -\mu \right\vert \leq
c_\delta  A_{1}\left( x\right) \kappa ^{k\gamma /4-1},   \label{Proposition MC002-000}
\end{equation}%
for any positive constant $\gamma $ satisfying $0<\gamma \leq \min \left\{1, 2\delta \right\}, $
where $A_{1}\left( x\right) = 1 +\mu _{\delta }\left( x\right)^{1+\gamma } +\left\Vert \boldsymbol{\delta }_{x}\right\Vert _{\mathcal{B}^{\prime }}\left\Vert e\right\Vert _{\mathcal{B}} C_{\mathbf{P}}C_{Q}.$ Moreover%
\begin{equation}
\sum_{k=0}^{\infty }\left\vert \mathbb{E}_{x}f\left( X_{k}\right) -\mu
\right\vert \leq \overline{\mu }\left( x\right) =c_{\delta, \kappa, \gamma
}A_{1}\left( x\right).  \label{Proposition MC002-001}
\end{equation}%

b) There exists a constant $\sigma \geq 0$ such that for any $x\in \mathbb{X}%
, $%
\begin{equation}
\sup_{m\geq 0}\left\vert \mathrm{Var}_{\mathbb{P}_{x}}\left(
\sum_{i=m+1}^{m+n}f\left( X_{i}\right) \right) -n\sigma ^{2}\right\vert \leq
\tau \left( x\right) =c_{\delta, \kappa, \gamma }A_{2}\left( x\right),
\label{Proposition MC002-002}
\end{equation}%
where%
\begin{equation*}
A_{2}\left( x\right) =1+\mu _{\delta }\left( x\right) ^{2+\gamma }+\left(
1+\left\Vert \boldsymbol{\delta }_{x}\right\Vert _{\mathcal{B}^{\prime
}}\right) \left\Vert e\right\Vert _{\mathcal{B}}\left( C_{\mathbf{P}%
}^{2}C_{Q}\left( 1+C_{Q}\right) +C_{\mathbf{P}}C_{Q}\left( 1+\left\Vert \nu
\right\Vert _{\mathcal{B}^{\prime }}C_{\mathbf{P}}\right) \right).
\end{equation*}
\end{proposition}

Note that the constants $\mu $ and $\sigma $ do not depend on the initial
state $x.$

The main result of this section is the following theorem. Let $\widetilde{%
\Omega }=\mathbb{R}^{\infty }\times \mathbb{R}^{\infty }.$ For any $%
\widetilde{\omega }=\left( \widetilde{\omega }_{1}, \widetilde{\omega }%
_{2}\right) \in \widetilde{\Omega }$ denote by $\widetilde{Y}_{i}=\widetilde{%
\omega }_{1, i}$ and $W_{i}=\widetilde{\omega }_{2, i}, $ $i\geq 1, $ the
coordinate processes in $\widetilde{\Omega }.$

\begin{theorem}
\label{Th main res 2}Assume that the Markov chain $(X_{n})_{n\geq 0}$ and
the function $f$ satisfy the hypotheses \textbf{M1},  \textbf{M2},  \textbf{M3}
and \textbf{M4}, with $\sigma >0.$ Let $0<\alpha <\delta.$
Then there exists a Markov transition kernel $x\rightarrow \widetilde{\mathbb{P}}%
_{x}\left( \cdot \right) $ from $\left( \mathbb{X}, \mathcal{X}\right) $ to $(%
\widetilde{\Omega }, \mathcal{B}(\widetilde{\Omega }))$ such that $\mathcal{L}%
\left( \left( \widetilde{Y}_{i}\right) _{i\geq 1}|\widetilde{\mathbb{P}}%
_{x}\right) \overset{d}{=}\mathcal{L}\left( \left( f\left( X_{i}\right)
\right) _{i\geq 1}|\mathbb{P}_{x}\right), $  the $W_{i},  i\geq 1,  $ are
independent standard normal r.v.'s under $\widetilde{\mathbb{P}}_{x}$ and
for any $0<\rho <\frac{\alpha }{2\left( 1+2\alpha \right) }$%
\begin{equation}
\widetilde{\mathbb{P}}_{x}\left( N^{-\frac{1}{2}}\sup_{k\leq N}\left\vert
\sum_{i=1}^{k}\left( \widetilde{Y}_{i}-\mu -\sigma W_{i}\right) \right\vert
>6N^{-\rho }\right) \leq C\left( x\right) N^{-\alpha \frac{1+\alpha }{%
1+2\alpha }+\rho \left( 2+2\alpha \right) },   \label{th main res 2 001}
\end{equation}%
with%
\begin{equation*}
C\left( x\right) =C_{1}\left( 1+\mu _{\delta }\left( x\right) +\left\Vert
\boldsymbol{\delta }_{x}\right\Vert _{\mathcal{B}^{\prime }}\right)
^{2+2\delta },
\end{equation*}%
where $C_{1}$ is a constant depending only on $\delta, \alpha, \kappa, C_{%
\mathbf{P}}, C_{Q}, \left\Vert e\right\Vert _{\mathcal{B}}$ and $\left\Vert
\nu \right\Vert _{\mathcal{B}^{\prime }}.$
\end{theorem}

Note that only the probability $\widetilde{\mathbb{P}}_{x}$ depends on the
initial state $x$ while the processes $\left( \widetilde{Y}_{k}\right)
_{k\geq 0}$ and $\left( W_{k}\right) _{k\geq 0}$ do not depend on it.

As in the previous section,  letting $\rho =\frac{\alpha }{3+2\alpha }\frac{%
1+\alpha }{1+2\alpha }, $ under the conditions of Theorem \ref{Th main res 2}
we obtain, %
\begin{equation}
\widetilde{\mathbb{P}}_{x}\left( N^{-\frac{1}{2}}\sup_{k\leq N}\left\vert
\sum_{i=1}^{k}\left( \widetilde{Y}_{i}-\mu -\sigma W_{i}\right) \right\vert
>6N^{-\frac{\alpha }{3+2\alpha }\frac{1+\alpha }{1+2\alpha }}\right) \leq
C\left( x\right) N^{-\frac{\alpha }{3+2\alpha }\frac{1+\alpha }{1+2\alpha }}.
\label{corr of Th 2}
\end{equation}%
Compared to the rate $N^{-\frac{\alpha }{3+2\alpha }}$ which is optimal in
the independent case the rate of convergence $N^{-\frac{\alpha }{3+2\alpha }%
\frac{1+\alpha }{1+2\alpha }}$ in (\ref{corr of Th 2}) is slower by the
factor $N^{\frac{\alpha }{3+2\alpha }\frac{\alpha }{1+2\alpha }}.$ As $%
\alpha \rightarrow \infty $ we obtain $N^{-\frac{1}{4}}$ which is the best
rate in the invariance principle that is known for dependent random
variables.

In Theorem \ref{Th main res 2} we do not suppose the existence of the
stationary measure. Assume that there exists a stationary probability
measure $\nu $ on $\mathbb{X}$; it thus coincides with the linear form $%
\nu $ introduced in (\ref{linear form}). Let $\mathbb{P}_{\nu }$ and $%
\mathbb{E}_{\nu }$ the probability measure and expectation generated by the
finite dimensional distributions of the chain under the stationary measure $%
\nu.$ Note that,  the means $\mathbb{E}_{\nu }X_{k}$ and the covariances
\textrm{Cov}$_{\mathbb{P}_{\nu }}\left( f\left( X_{l}\right), f\left(
X_{l+k}\right) \right) $ with respect to $\nu $ may not exist,  under the
hypotheses \textbf{M1,  M2, \ M3} and\textbf{\ M4}. To ensure their existence,
we require the following additional condition (where generally $\left\vert
f\right\vert ^{2}\notin \mathcal{B}$).

\vskip0.2cm \noindent \textbf{Hypothesis M5 (Stationary measure): }
\nobreak
\textit{On the state space }$\mathbb{X}$ \textit{there exists} a\textit{\ stationary
probability measure }$\nu $\textit{\ such that }$\nu \left( \sup_{k\geq 0}%
\mathbf{P}^{k}\left( \left\vert f\right\vert ^{2}\right) \right) <\infty.$

Under Hypothesis \textbf{M5} for $\mu $ and $\sigma $ we find the usual
expressions of the means and of the variance in the central limit theorem
for dependent r.v.'s.

\begin{theorem}
\label{Th main res 3}Assume that the Markov chain $(X_{n})_{n \geq 0}$ and
the function $f$ satisfy the hypotheses \textbf{M1},  \textbf{M2},  \textbf{M3}%
,  \textbf{M4} and \textbf{M5}. Assume also that $\sigma _{\nu }>0.$ Then
Proposition \ref{Proposition MC002} holds true with $\mu =\nu \left(
f\right) $ and $\sigma =\sigma _{\nu }, $ where%
\begin{equation*}
\nu \left( f\right) =\int f\left( x\right) \nu \left( dx\right)
\end{equation*}%
and
\begin{equation*}
\sigma _{\nu }^{2}=\mathrm{Var}_{\mathbb{P}_{\nu }}\left( f\left(
X_{0}\right) \right) +2\sum_{k=1}^{\infty }\mathrm{Cov}_{\mathbb{P}_{\nu
}}\left( f\left(X_{0}\right),  f\left( X_{k}\right) \right).
\end{equation*}%
Moreover,  if $\sigma _{\nu }>0$ the assertions of Theorem \ref{Th main res 2}
and (\ref{corr of Th 2}) hold true with $\mu =\nu \left( f\right) $ and $%
\sigma =\sigma _{\nu }.$
\end{theorem}

From Theorem \ref{Th main res 3} one can derive a bound when the Markov
chain $\left( X_{n}\right) _{n\geq 0}$ is in the stationary regime. If we
assume that $\nu \left( \sup_{k\geq 0}\mathbf{P}^{k}\left( \left\vert
f\right\vert ^{2+2\delta }\right) \right) \leq c_{\nu, \delta }<\infty $ and
$\int \left\Vert \boldsymbol{\delta }_{x}\right\Vert _{\mathcal{B}^{\prime
}}^{2+2\delta }\nu \left( dx\right) \leq c_{\mathcal{B}^{\prime }, \delta
}<\infty, $ then integrating (\ref{th main res 2 001}) with respect to $\nu $
we obtain%
\begin{equation*}
\widetilde{\mathbb{P}}_{\nu }\left( N^{-\frac{1}{2}}\sup_{k\leq N}\left\vert
\sum_{i=1}^{k}\left( \widetilde{Y}_{i}-\mu -\sigma W_{i}\right) \right\vert
>6N^{-\rho }\right) \leq CN^{-\alpha \frac{1+\alpha }{1+2\alpha }+\rho
\left( 2+2\alpha \right) },
\end{equation*}%
where $C$ is a constant depending on $\delta, \alpha, \kappa, C_{\mathbf{P}%
}, C_{Q}, \left\Vert e\right\Vert _{\mathcal{B}}, \left\Vert \nu \right\Vert _{%
\mathcal{B}^{\prime }}$ and $c_{\nu, \delta }, $ $c_{\mathcal{B}^{\prime
}, \delta }.$

The hypotheses \textbf{M1,  M2,  M4} and \textbf{M5} formulated above can be
easily verified by standard methods. As to \textbf{M3} it can be verified
using two approaches. The first approach is based on the assumption that the
family of operators $\left( \mathbf{P}_{t}\right) _{\left\vert t\right\vert
\leq \varepsilon _{0}}$ is continuous in $t$ at $t=0.$ In this case,  \textbf{%
M3} is satisfied by classical perturbation theory (see,  for instance Dunford
and Schwartz \cite{DunfordSchwartz58}). The second approach is based on a
weaker form of continuity of the family $\left( \mathbf{P}_{t}\right)
_{\left\vert t\right\vert \leq \varepsilon _{0}}$ as developed in Keller and
Liverani \cite{Kell-Liv99}.

We end this section by giving three examples where these hypotheses are
satisfied.

\subsection{Example 1 (Markov chains with finite state spaces)}

Suppose that $\left( X_{n}\right) _{n\geq 0}$ is an irreducible ergodic
aperiodic Markov chain with finite state space. It is easy to verify that
the hypotheses \textbf{M1,  M2,  M3,  M4} and \textbf{M5} are satisfied and
that there exists a unique invariant probability measure $\nu.$ Then the
conclusions of Theorem \ref{Th main res 3} hold true.

\subsection{Example 2 (autoregressive random walk with Bernoulli noise)}

Consider the autoregressive model $x_{n+1}=\alpha x_{n}+b_{n}, \;n\geq 0, $
where $\alpha $ is a constant satisfying $\alpha \in (-1, 1)$ and $\left(
b_{n}\right) _{n\geq 0}$ are i.i.d.~Bernoulli r.v.'s with $P\left(
b=1\right) =P\left( b=-1\right) =\frac{1}{2}$ and $x_{0}=x.$ Consider the
Banach space $\mathcal{B} = \mathcal{L}$ of continuous functions $f$ on $%
\mathbb{R}$ such that $\left\Vert f\right\Vert =\left\vert f\right\vert +%
\left[ f\right] <\infty, $ where%
\begin{equation*}
\left\vert f\right\vert =\sup_{x\in \mathbb{R}}\frac{\left\vert f\left(
x\right) \right\vert }{1+x^{2}}, \quad\quad\;\left[ f\right] =\sup_{\overset{%
x, y\in \mathbb{R}}{x\neq y}}\frac{\left\vert f\left( x\right) -f\left(
y\right) \right\vert }{\left\vert x-y\right\vert \left( 1+x^{2}\right)
\left( 1+y^{2}\right) }.
\end{equation*}%
Since $\alpha \in (-1, 1), $ the invariant measure $\nu $ exists and coincides
with the law of the random variable $Z=\sum_{i=1}^{\infty }\alpha
^{i-1}b_{i}.$ It is easy to verify that hypotheses \textbf{M1,  M2,  M3,  M4}
and \textbf{M5} are satisfied for the function $f\left( x\right) =x.$ For
the mean $\nu\left(f\right)$ we have
\begin{equation*}
\nu \left( f\right) =\int x\nu\left( dx\right) =\mathbb{E}%
Z=\sum_{i=1}^{\infty }\alpha ^{i-1}\mathbb{E}b_{1}=\frac{\mathbb{E}b_{1}}{%
1-\alpha }.
\end{equation*}%
Since $\mathbb{E }b_1=0$,  one gets $\nu\left(f\right)=0$ and the variance is
computed as follows:
\begin{equation*}
\sigma _{\nu }^{2}=\lim_{n\rightarrow \infty }\mathbb{E}\left(
\sum_{i=1}^{n}\alpha ^{i-1}b_{i}\right) ^{2}=\lim_{n\rightarrow \infty
}\sum_{i=1}^{n}\alpha ^{2\left( i-1\right) }\mathbb{E}b_{1}^{2}=\frac{1}{%
1-\alpha ^{2}}.
\end{equation*}%
Thus the conclusions of Theorem \ref{Th main res 3} hold true with $\nu(f)=0$
and $\sigma _{\nu }^{2}=\frac{1}{1-\alpha ^{2}}.$

\subsection{Example 3 (stochastic recursion)}

On the probability space $\left( \Omega, \mathcal{F}, \mathbb{P}\right) $
consider the stochastic recursion%
\begin{equation*}
x_{n+1}=a_{n+1}x_{n}+b_{n+1}, \;n\geq 0,
\end{equation*}%
where $\left( a_{n}, b_{n}\right) _{n\geq 0}$ are i.i.d.~ r.v.'s with values
in $(0, \infty)\times \mathbb{R}$ of the same distribution $\widehat{\mu }$
and $x_{0}=x.$ Following Guivarc'h and Le Page \cite{guiv-lepage},  we assume
the conditions:

\vskip0.2cm \noindent \textbf{H1}\textit{\ There exists }$\alpha >2$\textit{%
\ such that }$\varphi \left( \alpha \right) :=\int \left\vert a\right\vert
^{\alpha }\widehat{\mu }\left( da, db\right) <1$\textit{\ and }$\int
\left\vert b\right\vert ^{\alpha }\widehat{\mu }\left( da, db\right) <+\infty
.$

\vskip0.2cm \noindent \textbf{H2}\textit{\ }$\widehat{\mu }\left( \left\{
\left( a, b\right) :ax_{0}+b=x_{0}\right\} \right) <1$\textit{\ for any }$%
x_{0}\in \mathbb{R}$\textit{.}

\vskip0.2cm \noindent \textbf{H3} \textit{The set }$\left\{ \ln \left\vert
a\right\vert :\left( a, b\right) \in \text{supp\ }\widehat{\mu }\right\} $%
\textit{\ generates a dense subgroup of }$\mathbb{R}.$

Let $\varepsilon \in (0, 1), $ $\theta $ and $c$ be positive such that $\alpha
-1<c+\varepsilon <\theta \leq 2c<\alpha -\varepsilon.$ Consider the Banach
space $\mathcal{B}=\mathcal{L}_{\varepsilon, c, \theta }$ of continuous
functions $f$ on $\mathbb{R}$ such that $\left\Vert f\right\Vert =\left\vert
f\right\vert +\left[ f\right] <\infty, $ where%
\begin{equation*}
\left\vert f\right\vert =\sup_{x\in \mathbb{R}}\frac{\left\vert f\left(
x\right) \right\vert }{1+\left\vert x\right\vert ^{\theta }}, \quad \quad \;%
\left[ f\right] =\sup_{\overset{x, y\in \mathbb{R}}{x\neq y}}\frac{\left\vert
f\left( x\right) -f\left( y\right) \right\vert }{\left\vert x-y\right\vert
^{\varepsilon }\left( 1+\left\vert x\right\vert ^{c}\right) \left(
1+\left\vert y\right\vert ^{c}\right) }.
\end{equation*}%
The transition probability $\mathbf{P}\left( x, \cdot \right) $ of the Markov
chain $\left( x_{n}\right) _{n\geq 0}$ is defined by%
\begin{equation*}
\int h\left( y\right) \mathbf{P}\left( x, dy\right) =\int h\left( ax+b\right)
\widehat{\mu }\left( da, db\right),
\end{equation*}%
for any bounded Borel measurable function $h:\mathbb{R}\rightarrow \mathbb{R}
$ and $x\in \mathbb{R}.$ For any $x\in \mathbb{R}$ denote by $\mathbb{P}_{x}$
and $\mathbb{E}_{x}$ the corresponding probability measure and expectation
generated by the finite dimensional distributions on the space of
trajectories. It is proved in \cite{guiv-lepage} (Proposition 1) that the
series $\sum_{i=1}^{\infty }a_{1}...a_{i-1}b_{i}$ is $\mathbb{P}$-a.s.
convergent and the Markov chain $\left( x_{n}\right) _{n\geq 0}$ has a
unique invariant probability measure $\nu $ which coincides with the law of $%
Z=\sum_{i=1}^{\infty }a_{1}...a_{i-1}b_{i}.$ Moreover,  it holds $\int
\left\vert x\right\vert ^{t}\nu \left( dx\right) <\infty $ for any $t\in
\lbrack 0, \alpha ).$

We shall verify that hypotheses \textbf{M1,  M2,  M3,  M4} and \textbf{M5} are
satisfied with the function $f\left( x\right) =x.$ Hypothesis \textbf{M1} is
obvious and hypotheses \textbf{M2 }and \textbf{M3} follow from Theorem 1 and
Proposition 4 in \cite{guiv-lepage}. If $\delta >0$ is such that $2+2\delta
\leq \alpha, $ by simple calculations we obtain%
\begin{equation*}
\left( \mathbb{E}_{x}\left\vert x_{n}\right\vert ^{2+2\delta }\right) ^{%
\frac{1}{2+2\delta }}\leq \varphi \left( 2+2\delta \right) ^{\frac{n}{%
2+2\delta }}\left\vert x\right\vert +\frac{\left\Vert b_{1}\right\Vert
_{2+2\delta }}{1-\varphi \left( 2+2\delta \right) ^{\frac{1}{2+2\delta }}}.
\end{equation*}%
Taking the $\sup $ in $n\geq 1, $ we get%
\begin{eqnarray*}
\mu _{\delta }\left( x\right) &=&\sup_{n\geq 1}\left( \mathbb{E}%
_{x}\left\vert f\left( x_{n}\right) \right\vert ^{2+2\delta }\right) ^{\frac{%
1}{2+2\delta }} \\
&\leq &\varphi \left( 2+2\delta \right) ^{\frac{1}{2+2\delta }}\left\vert
x\right\vert +\frac{\left\Vert b_{1}\right\Vert _{2+2\delta }}{1-\varphi
\left( 2+2\delta \right) ^{\frac{1}{2+2\delta }}},
\end{eqnarray*}%
which proves that hypothesis \textbf{M4 }is satisfied. Finally,  hypothesis
\textbf{M5} is verified since%
\begin{equation*}
\int \mu _{\delta }\left( x\right) ^{2 }\nu\left(dx\right) \leq 2\left(
\varphi \left( 2+2\delta \right) ^{\frac{1}{1+\delta }}\int x^{2
}\nu\left(dx\right) +\left( \frac{\left\Vert b_{1}\right\Vert _{2+2\delta }
}{1-\varphi \left( 2+2\delta \right) ^{\frac{1}{2+2\delta }}}\right) ^{2}
\right)<\infty.
\end{equation*}%
The mean is given by $\nu \left( f\right) =\mathbb{E}Z=\sum_{i=1}^{\infty
}\left( \mathbb{E}a_{1}\right) ^{i-1}\mathbb{E}b_{1}=\frac{\mathbb{E}b_{1}}{%
1-\mathbb{E}a_{1}}.$ Without loss of generality we can assume that $\nu
\left( f\right) =0, $ i.e. that $\mathbb{E}b_{1}=0;$ then the variance is%
\begin{equation*}
\sigma _{\nu }^{2}=Var_{\mathbb{P}}\left( Z\right) =\lim_{n\rightarrow
\infty }\mathbb{E}\left( \sum_{i=1}^{n}a_{1}...a_{i-1}b_{i}\right)
^{2}=\lim_{n\rightarrow \infty }\sum_{i=1}^{n}\left( \mathbb{E}%
a_{1}^{2}\right) ^{i-1}\mathbb{E}b_{1}^{2}=\frac{\mathbb{E}b_{1}^{2}}{1-%
\mathbb{E}a_{1}^{2}}.
\end{equation*}%
Then the conclusions of Theorem \ref{Th main res 3} hold true with $\mu
=\nu(f)=0$ and $\sigma =\sigma _{\nu }^{2}=\frac{\mathbb{E}b_{1}^{2}}{1-%
\mathbb{E}a_{1}^{2}}.$

A multivariate version of the stochastic recursion has been considered in
Guivarc'h and Le Page \cite{GuivLePage2004}, \cite{GuivLePage2013} and can be treated in the same
manner.

\section{Partition of the set $\mathbb{N}$ and notations\label{sec Notations}%
}

In the sequel $\varepsilon, \beta \in \left( 0, 1\right) $ will be such that $%
\varepsilon +\beta <1$ (all over the paper $\varepsilon $ is supposed to be
very small,  while $\beta $ will be optimized). Denote for simplicity $%
[a, b)=\left\{ l\in \mathbb{N}:a\leq l<b\right\}.$ Let $k_{0}\geq 1$ be a
natural number. We start by splitting the set $\mathbb{N}$ into subsets $%
[2^{k}, 2^{k+1}), $ $k=k_{0}, k_{0}+1, ...$ called blocks. Consider the $k$-th
block $[2^{k}, 2^{k+1}).$ We leave a large gap $J_{k, 1}$ of length $%
2^{[\varepsilon k]+\left[ \beta k\right] }$ at the left end of the $k$-th
block. Then,  following a triadic Cantor-like scheme,  we split the remaining
part $[2^{k}+2^{[\varepsilon k]+\left[ \beta k\right] }, 2^{k+1})$ into
subsets $I_{k, j}$ and $J_{k, j}$ called islands and gaps as explained below.
At the resolution level $0$ a gap of size $2^{[\varepsilon k]+\left[ \beta k%
\right] }/2$ is put in the middle of the interval $[2^{k}+2^{[\varepsilon k]+%
\left[ \beta k\right] }, 2^{k+1}).$ This yields two intervals of equal
length. At the resolution level $1$ two additional gaps of length $%
2^{[\varepsilon k]+\left[ \beta k\right] }/2^{2}$ are put in the middle of
the each obtained interval which yields four intervals of equal length.
Continuing in the same way,  at the resolution level $[\beta k]$ we obtain $%
2^{[\beta k]}$ intervals $I_{k, j}, $ $j=1, ..., 2^{[\beta k]}$ called islands
and the same number of gaps $J_{k, j}, $ $j=1, ..., 2^{[\beta k]}$ which we
index from left to right (recall that $J_{k, 1}=J_{k, 2^{0}}$ denotes the
large gap at the left end of the $k$-th block). It is obvious that $%
[2^{k}, 2^{k+1})$ is the union of the constructed island and gaps,  so that
\begin{equation}
\lbrack 2^{k}, 2^{k+1})=J_{k, 1}\cup I_{k, 1}\cup...\cup J_{k, 2^{[\beta
k]}}\cup I_{k, 2^{\left[ \beta k\right] }}.  \label{union}
\end{equation}%
Note that in the block $k$ there are one gap of the length $2^{\left[[\varepsilon k]%
\right] +[\beta k] }$ and $2^{l}$ gaps of length $2^{\left[[\varepsilon k]%
\right] +[\beta k] -l-1}, $ where $l=0, ..., [\beta k]-1.$ The length of the
finest gap (for example $J_{k, 2^{[\beta k]}}$) is $2^{[\varepsilon k]}.$ The
total length of the gaps in the block $k$ is
\begin{equation*}
L_{k}^{gap}=2^{\left[ [\varepsilon k]\right]+[\beta k] }+\sum_{l=0}^{[\beta
k]-1}2^{l}2^{\left[ [\varepsilon k]\right]+[\beta k] -l-1}=\left( 2+[\beta
k]\right) 2^{\left[ [\varepsilon k]\right]+[\beta k] -1}.
\end{equation*}%
Recall that,  according to the construction,  the islands of the $k$-th block
have the same length%
\begin{eqnarray*}
\left\vert I_{k, j}\right\vert &=&\left( 2^{k+1}-2^{k}-\left( 2+[\beta
k]\right) 2^{\left[ [\varepsilon k]\right]+[\beta k] -1}\right) /2^{[\beta
k]} \\
&=&2^{k-[\beta k]}-\left( 1+\left[ \beta k\right] 2^{\left[ [\varepsilon k]%
\right] -1}\right).
\end{eqnarray*}%
An obvious upper bound is $\left\vert I_{k, j}\right\vert \leq 2^{k-[\beta
k]}. $ Since $\varepsilon <1-\beta $ we have $\left\vert I_{k, j}\right\vert
\geq 2^{k-[\beta k]}-2^{\left[ [\varepsilon k]\right] -c_{\beta, \varepsilon
}^{\prime }\ln k}\geq c_{\varepsilon, \beta }2^{k\left( 1-\beta \right) }, $
with some $c_{\varepsilon, \beta }\in (0, \frac{1}{2}).$ Since the length of
the $k$-th block is $2^{k},$ the total length of the islands in this block is equal to
\begin{equation*}
L_{k}^{isl}=2^{k}-2^{\left[ [\varepsilon k] \right]+[\beta k] -1}\left(
2+[\beta k]\right).
\end{equation*}%
Note that,  for some constant $c_{\beta }>0, $
\begin{equation}
c_{\beta }2^{k}\leq L_{k}^{isl}\leq 2^{k}.  \label{length of the islands}
\end{equation}

Denote by $\mathcal{K}$ the set of double indices $(k, j), $ with $k=1, 2, ...$
being the index of the block and $j=1, ..., 2^{[\beta k]}$ being the index of
the island in the block $k.$ The set $\mathcal{K}$ will be endowed with
lexicographical order $\preceq.$ Then the sets $I_{k, j}$ and $J_{k, j}, \
\left( k, j\right) \in \mathcal{K}, $ will be also endowed with the
lexicographical order. Let $N\in \mathbb{N}.$ From (\ref{union}),
there exists a unique $\left( n, m\right) \in \mathcal{K}$ such that $%
2^{n}\leq N<2^{n+1}$ and $N\in J_{n, m}\cup I_{n, m}, $ where the dependence of
$n$ and $m$ on $N$ is suppressed from the notation; denote $\mathcal{K}%
_{N}=\left\{ \left( k, j\right) :\left( k, j\right) \preceq \left( n, m\right)
\right\}.$

For ease of reading we recall the notations and properties that will be used
throughout the paper:

\vskip0.2cm \noindent \textbf{P1:} $\varepsilon $ and $\beta $ are positive
numbers such that $\varepsilon +\beta <1.$ Later on,  the constant $%
\varepsilon $ will be chosen to be small enough.

\vskip0.2cm \noindent \textbf{P2:} $\mathcal{K}=\left\{ \left( k, j\right)
:k=1, 2, ..., \;j=1, ..., 2^{[\beta k]}\right\}.$

\vskip0.2cm \noindent \textbf{P3:} For any $N\in \mathbb{N}$ the unique
couple $\left( n, m\right) \in \mathcal{K}$ is such that $N\in J_{n, m}\cup
I_{n, m}.$

\vskip0.2cm \noindent \textbf{P4:} $\mathcal{K}_{N}=\left\{ \left(
k, j\right) :\left( k, j\right) \preceq \left( n, m\right) \right\}.$

\vskip0.2cm \noindent \textbf{P5:} $I_{k, j}, $ $j=1, ..., 2^{\left[ \beta k%
\right] }$ are the islands and $J_{k, j}, $ $j=1, ..., 2^{\left[ \beta k\right]
} $ are the gaps in the $k$-th block.

\vskip0.2cm \noindent \textbf{P6:} The number of islands and the number of
gaps in the $k$-th block are both equal to $m_{k}=2^{[\beta k]}.$ Set $%
m_{k, n}=m_{k}+...+m_{n}.$

\vskip0.2cm \noindent \textbf{P7:} The islands in $k$-th block have the same
length $\left\vert I_{k, j}\right\vert =2^{k-[\beta k]}-\left( 1+\left[ \beta
k\right] 2^{\left[ [\varepsilon k]\right] -1}\right) \leq 2^{k-[\beta k]}.$
This implies $\left\vert I_{k, j}\right\vert \geq c_{\varepsilon, \beta
}2^{k\left( 1-\beta \right) }$ for some constant $c_{\varepsilon, \beta }\in
(0, \frac{1}{2}).$

\vskip0.2cm \noindent \textbf{P8:} The length of the finest gap in the $k$%
-th block is $\left\vert J_{k, j}\right\vert =2^{\left[ [\varepsilon k]\right]
}.$ This implies $\left\vert J_{k, j}\right\vert \geq 2^{\left[ [\varepsilon
k]\right] }.$

\vskip0.2cm \noindent \textbf{P9:} The length $\left\vert J_{k, 1}\right\vert
$ of the gap at the left end of the $k$-th block is $2^{[\varepsilon k]+%
\left[ \beta k\right] }.$

\vskip0.2cm \noindent \textbf{P10:} For each pair $\left( k, j\right) \in
\mathcal{K}, $ we denote $X_{\left( k, j\right) }=\sum_{i\in I_{k, j}}X_{i}$
and $W_{\left( k, j\right) }=\sum_{i\in I_{k, j}}W_{i}.$

\vskip0.2cm \noindent \textbf{P11:} $\mathcal{L}_{X_{1}, ..., X_{d}}$ denotes
the probability law of the random vector $\left( X_{1}, ..., X_{d}\right).$

\section{Auxiliary results\label{sec Aux Res}}

Without loss of generality we assume that on the initial probability space
there is a sequence of independent r.v.'s $\left( Y_{\left( k, j\right)
}\right) _{\left( k, j\right) \in \mathcal{K}}$ such that $Y_{\left(
k, j\right) }\overset{d}{=}X_{\left( k, j\right) }, $ $\left( k, j\right) \in
\mathcal{K}.$ Let $k_{0}\in \mathbb{N}_{+}$ and $n>k_{0}.$ Suppose that on
the same probability space there is an i.i.d. sequence of $\mathbb{R}^{1}$%
-valued r.v.'s $\left( V_{\left( k, j\right) }\right) _{\left( k, j\right) \in
\mathcal{K}}$ with means $0$ whose characteristic function has as support
the interval $[-\varepsilon _{0}, \varepsilon _{0}]$ and such that $\mathbb{E}%
\left\vert V_{\left( k, j\right) }\right\vert ^{r_{0}}<\infty $ for any $%
r_{0}>0.$ We suppose that the sequence $\left( V_{\left( k, j\right) }\right)
_{\left( k, j\right) \in \mathcal{K}}$ is independent of $\left( X_{\left(
k, j\right) }\right) _{\left( k, j\right) \in \mathcal{K}}$ and $\left(
Y_{\left( k, j\right) }\right) _{\left( k, j\right) \in \mathcal{K}}.$ Denote $%
X_{\left( k\right) }=\left( X_{\left( k, 1\right) }, ..., X_{(k, m_{k})}\right), $
$Y_{\left( k\right) }=\left( Y_{\left( k, 1\right) }, ..., Y_{\left(
k, m_{k}\right) }\right) $ and $V_{\left( k\right) }=\left( V_{\left(
k, 1\right) }, ..., V_{(k, m_{k})}\right).$ In the sequel $\pi $ denotes the
Prokhorov distance (for details see Section \ref{sec Prokhorov dist} of the
Appendix).

Assume conditions \textbf{C1} and \textbf{C2}.  The main result of this section is the following proposition,  which is of
independent interest.

\begin{proposition}
\label{Prop Aux Main} There exists a constant $c_{\varepsilon, \beta, \lambda _{1}, \lambda _{2}} $ such that,
for any $k_{0}=1, 2, ...$ and $n>k_{0}, $
\begin{equation*}
\pi \left( \mathcal{L}_{X_{\left( k_{0}\right) }+V_{\left( k_{0}\right)
}, ..., X_{\left( n\right) }+V_{\left( n\right) }}, \mathcal{L}_{Y_{\left(
k_{0}\right) }+V_{\left( k_{0}\right) }, ..., Y_{\left( n\right) }+V_{\left(
n\right) }}\right) \leq c_{\varepsilon, \beta, \lambda _{1}, \lambda _{2}}
\left( 1+\lambda _{0}+\mu _{\delta }\right) \exp \left( -\frac{\lambda
_{1}}{4}2^{\frac{\varepsilon }{2}k_{0}}\right).
\end{equation*}
\end{proposition}
\begin{proof}
Without loss of generality we assume that there exists a sequence of independent random
vectors $R_{\left( k\right) }, $ $k=1, ..., n$ such that $R_{\left( k\right) }%
\overset{d}{=}X_{\left( k\right) }+V_{\left( k\right) }$ and such that the
sequence $\left( R_{\left( k\right) }\right) _{k=1, ..., n}$ is independent of
the sequences $\left( X_{\left( k\right) }+V_{\left( k\right) }\right)
_{k=1, ..., n}$,  $\left( Y_{\left( k, j\right) }\right) _{\left( k, j\right) \in
\mathcal{K}} $ and $\left( V_{\left( k, j\right) }\right) _{\left( k, j\right)
\in \mathcal{K}}.$

The further proof is slit into two parts a) and  b).
In the part  a) we give a bound for the Prokhorov
distance between the vectors $\left( X_{\left( k_{0}\right) }+V_{\left(
k_{0}\right) }, ..., X_{\left( n\right) }+V_{\left( n\right) }\right) $ and $%
\left( R_{\left( k_{0}\right) }..., R_{\left( n\right) }\right),$ while in the part b) we give a bound for the Prokhorov distance between the
vectors $\left( R_{\left( k_{0}\right) }, ..., R_{\left( n\right) }\right) $
and $\left( Y_{\left( k_{0}\right) }+V_{\left( k_{0}\right) }, ..., Y_{\left(
n\right) }+V_{\left( n\right) }\right).$
Proposition \ref{Prop Aux Main} follows from  (\ref{smoocoupling 001}) and (\ref{smoocoupling 002}) by triangle inequality.

{\it Part a).} We show that there exists a constant $c_{\varepsilon, \beta, \lambda _{1}, \lambda
_{2}} $ such that,  for any $k_{0}=1, 2, ...$ and
$n>k_{0}, $%
\begin{equation}
\pi \left( \mathcal{L}_{X_{\left( k_{0}\right) }+V_{\left( k_{0}\right)
}, ..., X_{\left( n\right) }+V_{\left( n\right) }}, \mathcal{L}_{R_{\left(
k_{0}\right) }..., R_{\left( n\right) }}\right) \leq c_{\varepsilon, \beta
, \lambda _{1}, \lambda _{2}}\left( 1+\lambda _{0}+\mu _{\delta }\right) \exp
\left( -\frac{\lambda _{1}}{4}2^{\frac{\varepsilon }{2}k_{0}}\right).
\label{smoocoupling 001}
\end{equation}

For $k=k_{0}, ..., n, $ define the vectors
$
Z_{\left( k\right) }=\left( X_{\left( k_{0}\right) }+V_{\left( k_{0}\right)
}, ..., X_{\left( k\right) }+V_{\left( k\right) }\right)$
and
$
\widetilde{Z}_{\left( k\right) }=\left( Z_{\left( k-1\right) },  R_{\left( k\right) }\right).
$
By Lemma \ref{properties 003},  one gets%
\begin{equation}
\pi \left( \mathcal{L}_{Z_{\left( n\right) }}, \mathcal{L}_{R_{\left(
k_{0}\right) }, ..., R_{\left( n\right) }}\right) \leq \sum_{k=k_{0}}^{n}\pi
\left( \mathcal{L}_{Z_{\left( k\right) }}, \mathcal{L}_{\widetilde{Z}_{\left(
k\right) }}\right).  \label{telescope001}
\end{equation}%
Let $\phi _{\left( k\right) }$ (resp. $\widetilde{\phi }_{\left(
k\right) }$) be the characteristic function of the vector $Z_{\left(
k\right) }$ (resp. $\widetilde{Z}_{\left( k\right) }$) and $%
m_{k_{0}, k}=m_{k_{0}}+...+m_{k}.$ Then,  by Lemma \ref{Lemma smoo},  for any $%
T>0, $%
\begin{eqnarray}
\pi \left( \mathcal{L}_{Z_{\left( k\right) }}, \mathcal{L}_{\widetilde{Z}%
_{\left( k\right) }}\right) &\leq &\left( T/\pi \right)
^{m_{k_{0}, k}/2}\left( \int_{t\in \mathbb{R}^{m_{k_{0}, k}}}\left\vert \phi
_{\left( k\right) }\left( t\right) -\widetilde{\phi }_{\left( k\right)
}\left( t\right) \right\vert ^{2}dt\right) ^{1/2}  \notag \\
&&+\mathbf{P}\left( \max_{k_{0}\leq l\leq k}\max_{1\leq j\leq
m_{l}}\left\vert X_{\left( l, j\right) }\right\vert >T\right).
\label{Block 005}
\end{eqnarray}%
Denote by $\varphi _{\left( k\right) }$ and $\psi _{\left(k\right) }$  the characteristic functions of the vectors $X_{\left( k\right) }$
and $\left( X_{\left( k_{0}\right) }, ..., X_{\left( k\right)
}\right) )$ respectively. Since $V_{\left( k_{0}\right) }, ..., V_{\left( k\right) }$ are
independent of $X_{\left( k_{0}\right) }, ..., X_{\left( k\right) }$ and $%
Y_{\left( k_{0}\right) }, ..., Y_{\left( k\right) }, $
we have%
\begin{eqnarray}
&&\int_{t\in \mathbb{R}^{m_{k_{0}, k}}, }\left\vert \phi _{\left( k\right)
}\left( t\right) -\varphi _{\left( k\right) }\left( t\right) \right\vert
^{2}dt  \notag \\
&=&\int_{t_{1}\in \mathbb{R}^{m_{k_{0}}}}...\int_{t_{k}\in \mathbb{R}%
^{m_{k}}}\left\vert \phi _{\left( k\right) }\left(
t_{k_{0}}, ..., t_{k}\right) -\varphi _{\left( k\right) }\left(
t_{k_{0}}, ..., t_{k}\right) \right\vert ^{2}dt_{k_{0}}...dt_{k}  \notag \\
&\leq &I_{1}  \notag \\
&\equiv &\int_{t_{1}\in \mathbb{R}^{m_{k_{0}}}}...\int_{t_{k}\in \mathbb{R}%
^{m_{k}}}\left\vert \psi _{\left( k\right) }\left(
t_{k_{0}}, ..., t_{k}\right) -\psi _{\left( k-1\right) }\left(
t_{k_{0}}, ..., t_{k-1}\right) \varphi _{\left( k\right) }\left( t_{k}\right)
\right\vert ^{2}dt_{k_{0}}...dt_{k}.  \label{Block 006}
\end{eqnarray}%
To bound the right-hand side of (\ref{Block 006}),  note that $%
m_{k_{0}, k}=\left( 2^{\left[ \beta k_{0}\right] }+...+2^{\left[ \beta k%
\right] }\right) \leq 2^{\left[ \beta k\right] +1}$ and,  according to the
construction,  the length of the gap between the vectors $X_{\left(
k-1\right) }$ and $X_{\left( k\right) }$ is $k_{gap}=2^{[\varepsilon k]+%
\left[ \beta k\right] }.$ Note also that $\left\vert I_{k, j}\right\vert \leq
2^{k-\left[ \beta k\right] }$ and $\left\vert \varepsilon _{0}\right\vert
\leq 1.$ Let us remind that the characteristic functions of the r.v. $%
V_{(k, j)}$ have support the interval $[-\epsilon _{0}, \epsilon _{0}]$ and
that the sequence $(V_{(k, j)})_{(k, j)\in \mathcal{K}}$ is independent of $%
(X_{(k, j)})_{(k, j)\in \mathcal{K}}$; it readily implies that the integrals
above are in fact over the set $[-\epsilon _{0}, \epsilon _{0}]^{m_{k_{0}, k}}$%
. Using condition \textbf{C1},  with $M_{1}=m_{k_{0}, k-1}$ and $M_{2}=m_{k}, $
one may thus write%
\begin{eqnarray}
I_{1} &\leq &\lambda _{0}\left( 1+\max_{l\leq k, ~j\leq m_{k}}\left\vert
I_{l, j}\right\vert \right) ^{\lambda _{2}(M_{1}+M_{2})}\exp \left( -\lambda
_{1}k_{gap}\right) \varepsilon _{0}^{m_{k_{0}, k}}  \notag \\
&\leq &\lambda _{0}\left( 1+2^{k-\left[ \beta k\right] }\right) ^{\lambda
_{2}2^{\left[ \beta k\right] +1}}\exp \left( -\lambda _{1}k_{gap}\right)
\notag \\
&\leq &\lambda _{0}\exp \left( -\lambda _{1}2^{[\varepsilon k]+\left[ \beta k%
\right] }+\lambda _{2}2^{\left[ \beta k\right] +1}\ln \left( 1+2^{k-\left[
\beta k\right] }\right) \right)  \notag \\
&\leq &c_{\varepsilon, \beta, \lambda _{1}, \lambda _{2}}\lambda _{0}\exp
\left( -\frac{\lambda _{1}}{2}2^{[\varepsilon k]+\left[ \beta k\right]
}\right).  \label{Block 007}
\end{eqnarray}%
Putting together (\ref{Block 005}),  (\ref{Block 006}) and (\ref{Block 007}),
we get%
\begin{eqnarray}
\pi \left( \mathcal{L}_{Z_{\left( k\right) }}, \mathcal{L}_{\widetilde{Z}%
_{\left( k\right) }}\right) &\leq &c_{\varepsilon, \beta, \lambda
_{1}, \lambda _{2}}\lambda _{0}\left( T/\pi \right) ^{m_{k_{0}, k}/2}\exp
\left( -\frac{\lambda _{1}}{2}2^{[\varepsilon k]+\left[ \beta k\right]
}\right)  \notag \\
&&+\sum_{k_{0}\leq l\leq k}\sum_{1\leq j\leq m_{l}}\mathbb{P}\left(
\left\vert X_{\left( l, j\right) }\right\vert >T\right).  \label{Block 008a}
\end{eqnarray}%
Since $\left\vert I_{\left( l, j\right) }\right\vert \leq 2^{l}, $ by Markov's
inequality and condition \textbf{C2}, %
\begin{equation*}
\mathbb{P}\left( \left\vert X_{\left( l, j\right) }\right\vert >T\right) \leq
T^{-1}\mathbb{E}\left\vert X_{\left( l, j\right) }\right\vert \leq
T^{-1}2^{l}\max_{i}\mathbb{E}\left\vert X_{i}\right\vert \leq \mu _{\delta
}T^{-1}2^{l}.
\end{equation*}%
Choosing $T=\exp \left( 2^{[\varepsilon k]/2}\right), $ one gets%
\begin{eqnarray}
\sum_{k_{0}\leq l\leq k}\sum_{1\leq j\leq m_{l}}\mathbb{P}\left( \left\vert
X_{\left( l, j\right) }\right\vert >T\right) &\leq &\mu _{\delta
}T^{-1}\sum_{k_{0}\leq l\leq k}m_{l}2^{l}  \notag \\
&\leq &\mu _{\delta }\exp \left( -2^{[\varepsilon k]/2}\right)
\sum_{k_{0}\leq l\leq k}2^{\left[ \beta l\right] }2^{l}  \notag \\
&\leq &c_{\beta }\mu _{\delta }\exp \left( -2^{[\varepsilon k]/2}/2\right).
\label{Block 009}
\end{eqnarray}%
Since $m_{k_{0}, k}\leq 2^{\beta k}, $ one gets%
\begin{equation}
\left( T/\pi \right) ^{m_{k_{0}, k}/2}\leq \exp \left( \frac{1}{2}%
2^{[\varepsilon k]/2+\beta k}\right).  \label{Block 009a}
\end{equation}%
From (\ref{Block 008a}),  (\ref{Block 009}) and (\ref{Block 009a}),  we deduce%
\begin{eqnarray*}
\pi \left( \mathcal{L}_{Z_{\left( k\right) }}, \mathcal{L}_{\widetilde{Z}%
_{\left( k\right) }}\right) &\leq &c_{\varepsilon, \beta, \lambda
_{1}, \lambda _{2}}\lambda _{0}\exp \left( \frac{1}{2}2^{[\varepsilon
k]/2+[\beta k]}\right) \exp \left( -\frac{\lambda _{1}}{2}2^{[\varepsilon k]+%
\left[ \beta k\right] }\right) \\
&&+c_{\beta }\mu _{\delta }\exp \left( -2^{[\varepsilon k]/2}/2\right) \\
&\leq &\left( 1+\lambda _{0}+\mu _{\delta }\right) c_{\varepsilon, \beta
, \lambda _{1}, \lambda _{2}}\exp \left( -\frac{\lambda _{1}}{4}2^{\varepsilon
k/2}\right).
\end{eqnarray*}%
Using (\ref{telescope001})%
\begin{eqnarray*}
\pi \left( \mathcal{L}_{ Z_{\left( n\right)} }, %
\mathcal{L}_{\left( R_{\left( k_{0}\right) }..., R_{\left( n\right) }\right)
}\right) &\leq &\left( 1+\lambda _{0}+\mu _{\delta }\right) c_{\varepsilon
, \beta, \lambda _{1}, \lambda _{2}}\sum_{k=k_{0}}^{n}\exp \left( -\frac{%
\lambda _{1}}{4}2^{[\varepsilon k]/2}\right) \\
&\leq &\left( 1+\lambda _{0}+\mu _{\delta }\right) c_{\varepsilon, \beta
, \lambda _{1}, \lambda _{2}}^{\prime }\exp \left( -\frac{\lambda _{1}}{4}%
2^{[\varepsilon k]_{0}/2}\right).
\end{eqnarray*}%
This concludes the proof of the part a).

{\it Part b)}. We show that exists a constant $c_{\varepsilon, \beta, \lambda _{1}, \lambda _{2}} $  such that,  for any $k_{0}=1, 2, ...$ and $%
n>k_{0}, $%
\begin{equation}
\pi \left( \mathcal{L}_{R_{\left( k_{0}\right) }, ..., R_{\left( n\right) }}, %
\mathcal{L}_{Y_{\left( k_{0}\right) }+V_{\left( k_{0}\right) }, ..., Y_{\left(
n\right) }+V_{\left( n\right) }}\right) \leq c_{\varepsilon, \beta, \lambda
_{1}, \lambda _{2}}\left( 1+\lambda _{0}+\mu _{\delta }\right) \exp \left( -%
\frac{\lambda _{1}}{8}2^{[\varepsilon k]_{0}/2}\right).
\label{smoocoupling 002}
\end{equation}

By Lemma \ref{properties 004},  since $R_{\left( k_{0}\right) }, ..., R_{\left(
n\right) }$ and $Y_{\left( k_{0}\right) }+V_{\left( k_{0}\right)
}, ..., Y_{\left( n\right) }+V_{\left( n\right) }$ are independent r.v.'s,  one
may write%
\begin{equation}
\pi \left( \mathcal{L}_{R_{\left( k_{0}\right) }, ..., R_{\left( n\right) }}, %
\mathcal{L}_{Y_{\left( k_{0}\right) }+V_{\left( k_{0}\right) }, ..., Y_{\left(
n\right) }+V_{\left( n\right) }}\right) =\sum_{k=k_{0}}^{n}\pi \left(
\mathcal{L}_{R_{\left( k\right) }}, \mathcal{L}_{Y_{\left( k\right)
}+V_{\left( k\right) }}\right)  \label{DDD000a}
\end{equation}%
and its suffices to prove that,  for any $k=1, 2, ..., $ one gets
\begin{equation}
\pi \left( \mathcal{L}_{R_{\left( k\right) }}, \mathcal{L}_{Y_{\left(
k\right) }+V_{\left( k\right) }}\right) \leq \left( 1+\lambda_0+\mu _{\delta
}\right) c_{\varepsilon, \beta, \lambda _{1}, \lambda _{2}}^{\prime }\exp
\left( -\frac{\lambda _{1}}{8}2^{[\varepsilon k]/2}\right).  \label{DDD001}
\end{equation}

For this,  recall that,  according to the construction in Section \ref{sec
Notations},  at the resolution level $0, $ a gap of length $2^{\left[
[\varepsilon k]\right] +\left[ \beta k\right] }/2$ in the middle of the
block $R_{\left( k\right) }^{0, 0}=R_{\left( k\right) }$ splits it into two
vectors $\widetilde{R}_{\left( k\right) }^{0, 1}$ and $\widetilde{R}_{\left(
k\right) }^{0, 2}$; let $R_{\left( k\right) }^{0, 1}$ and $R_{\left( k\right)
}^{0, 2}$ be independent versions of $\widetilde{R}_{\left( k\right) }^{0, 1}$
and $\widetilde{R}_{\left( k\right) }^{0, 2}$ respectively. Next,  at the
resolution level $1, $ for any $j\in \left\{ 1, 2\right\}, $ a gap of length $%
2^{\left[ \left( \varepsilon +\beta \right) k\right] }/4$ in the middle of
the block $R_{\left( k\right) }^{0, j}$ splits it into two vectors $%
\widetilde{R}_{\left( k\right) }^{1, 2j-1}$ and $\widetilde{R}_{\left(
k\right) }^{1, 2j};$ let $R_{\left( k\right) }^{1, 2j-1}$ and $R_{\left(
k\right) }^{1, 2j}$ be independent versions of $\widetilde{R}_{\left(
k\right) }^{1, 2j-1}$ and $\widetilde{R}_{\left( k\right) }^{1, 2j}$
respectively. Assuming that at the resolution level $l\in \left\{ 1, ..., %
\left[ \beta k\right] \right\} $ the independent r.v.'s $R_{\left( k\right)
}^{l, j}, $ $j\in \left\{ 1, ..., 2^{l}\right\},  $ are already constructed,  we
shall perform the construction at the resolution level $l+1.$ Note that,  at
the resolution level $l$ for any $j\in \left\{ 1, ..., 2^{l}\right\}, $ a gap
of length $2^{\left[ \left( \varepsilon +\beta \right) k\right] }/2^{l+1}$
in the middle of the block $R_{\left( k\right) }^{l, j}$ splits it into two
vectors $\widetilde{R}_{\left( k\right) }^{l+1, 2j-1}$ and $\widetilde{R}%
_{\left( k\right) }^{l+1, 2j}; $ it is enough to set $R_{\left( k\right)
}^{l+1, 2j-1}$ and $R_{\left( k\right) }^{l+1, 2j}$ to be independent versions
of $\widetilde{R}_{\left( k\right) }^{l+1, 2j-1}$ and $\widetilde{R}_{\left(
k\right) }^{l+1, 2j}$ respectively. It is easy to see that at the final level
$l_{k}=\left[ \beta k\right] $ we have $R_{\left( k\right) }^{l_{k}, j}%
\overset{d}{=}Y_{\left( k, j\right) }+V_{\left( k, j\right) }, $ for $%
j=1, ..., m_{k}=2^{\left[ \beta k\right] }.$

Let $l\in \left\{ 0, ..., \left[ \beta k\right] \right\}.$ For $j\in \left\{
1, ..., 2^{l}\right\}, $ denote by $\psi _{k}^{l, 2j-1}$ and $\psi _{k}^{l, 2j}$
the characteristic functions of $R_{\left( k\right) }^{l, 2j-1}$ and $%
R_{\left( k\right) }^{l, 2j}.$ Using Lemma \ref{Lemma smoo} and the
independence of the vectors $\widetilde{R}_{\left( k\right) }^{l, 2j-1}$ and $%
\widetilde{R}_{\left( k\right) }^{l, 2j}$,  we get
\begin{eqnarray}
&&\pi \left( \mathcal{L}_{R_{\left( k\right) }^{l, j}, }\mathcal{L}_{R_{\left(
k\right) }^{l+1, 2j-1}, R_{\left( k\right) }^{l+1, 2j}}\right)  \notag \\
&\leq &\left( \left( T/\pi \right) ^{2^{-l}m_{k}}\int_{\left( t, s\right) \in
\mathbb{R}^{2^{-l}m_{k}}}\left\vert \psi _{k}^{l, j}\left( t, s\right) -\psi
_{k}^{l+1, 2j-1}\left( t\right) \psi _{k}^{l+1, 2j}\left( s\right) \right\vert
^{2}dtds\right) ^{1/2}  \notag \\
&&+\sum_{1\leq j\leq 2^{-l}m_{k}}\mathbb{P}\left( \left\vert X_{\left(
k, j\right) }+V_{\left( k, j\right) }\right\vert >T\right).  \label{DDD002}
\end{eqnarray}%
By Condition \textbf{C1} with $N=M=\frac{m_{k}}{2}2^{-l}$ and $%
k_{gap}=2^{[\varepsilon k]+[\beta k]-l-1}, $ we obtain%
\begin{eqnarray}
&&\int_{\left( t, s\right) \in \mathbb{R}^{m_{k}2^{-l}}}\left\vert \psi
_{k}^{l, j}\left( t, s\right) -\psi _{k}^{l+1, 2j-1}\left( t\right) \psi
_{k}^{l+1, 2j}\left( s\right) \right\vert ^{2}dtds  \notag \\
&=&\int_{\left( t, s\right) \in \mathbb{R}^{m_{k}2^{-l}}, ~\left\Vert
t\right\Vert _{\infty }\leq \varepsilon _{0}, \left\Vert s\right\Vert
_{\infty }\leq \varepsilon _{0}}\left\vert \psi _{k}^{l, j}\left( t, s\right)
-\psi _{k}^{l+1, 2j-1}\left( t\right) \psi _{k}^{l+1, 2j}\left( s\right)
\right\vert ^{2}dtds  \notag \\
&\leq &\lambda _{0}\exp \Bigl(\lambda _{2}m_{k}2^{-l}\ln \left( 1+2^{k-\left[
\beta k\right] }\right) -\lambda _{1}2^{[\varepsilon k]+[\beta k]-1-l}\Bigr)%
\left( 2\varepsilon _{0}\right) ^{m_{k}2^{-l}}  \notag \\
&\leq &\lambda _{0}c_{\varepsilon, \beta, \lambda _{1}, \lambda _{2}}^{\prime
\prime }\exp \left( -\frac{\lambda _{1}}{4}2^{[\varepsilon k]+[\beta
k]-l}\right).  \label{DDD003}
\end{eqnarray}%
We will thus take $T=\exp \left( \lambda _{1}2^{\left[ [\varepsilon k]\right]
/2}\right) $ so that
\begin{equation*}
\left( T/\pi \right) ^{2^{-l}m_{k}}\leq \exp \left( \lambda
_{1}2^{-l}m_{k}2^{\left[ [\varepsilon k]\right] /2}\right) \leq \exp \left(
\lambda _{1}2^{\left[ [\varepsilon k]\right] /2+\left[ \beta k\right]
-l}\right).
\end{equation*}%
In order to control the terms $\mathbb{P}\left( \left\vert X_{\left(
k, j\right) }+V_{\left( k, j\right) }\right\vert >T\right) $,  we use Markov's
inequality,  condition \textbf{C2} and the fact that $\left\vert
I_{k, j}\right\vert \leq 2^{k}$; it readily follows that
\begin{eqnarray*}
\mathbb{P}\left( \left\vert X_{\left( k, j\right) }+V_{\left( k, j\right)
}\right\vert >T\right) &\leq &T^{-1}\left( \mathbb{E}\left\vert X_{\left(
k, j\right) }\right\vert +\mathbb{E}\left\vert V_{\left( k, j\right)
}\right\vert \right) \\
&\leq &T^{-1}\left( 2^{k}\max_{i}\mathbb{E}\left\vert X_{i}\right\vert
+c2^{k}\right) \\
&\leq &\left( 1+\mu _{\delta }\right) c2^{k}\exp \left( -\lambda _{1}2^{-
\left[ [\varepsilon k]\right] /2}\right).
\end{eqnarray*}%
Therefore%
\begin{eqnarray}
\sum_{1\leq j\leq 2^{-l}m_{k}}\mathbb{P}\left( \left\vert X_{\left(
k, j\right) }+V_{\left( k, j\right) }\right\vert >T\right) &\leq
&2^{-l}m_{k}\left( 1+\mu _{\delta }\right) c2^{k}\exp \left( -2^{\left[
[\varepsilon k]\right] /2}\right)  \notag \\
&\leq &\left( 1+\mu _{\delta }\right) \exp \left( -\lambda _{1}2^{\left[
[\varepsilon k]\right] /2}\right) 2^{-l}2^{2\left[ \beta k\right] +k}  \notag
\\
&\leq &\left( 1+\mu _{\delta }\right) c_{\varepsilon, \beta, \lambda
_{1}, \lambda _{2}}\exp \left( -\frac{\lambda _{1}}{2}2^{\left[ [\varepsilon
k]\right] /2}\right).  \label{DDD004}
\end{eqnarray}%
From (\ref{DDD002}),  (\ref{DDD003}) and (\ref{DDD004}),  we get%
\begin{eqnarray}
&&\pi \left( \mathcal{L}_{R_{\left( k\right) }^{l, j}, }\mathcal{L}_{R_{\left(
k\right) }^{l+1, 2j-1}, R_{\left( k\right) }^{l+1, 2j}}\right)  \notag \\
&\leq &\left( 1+\lambda _{0}+\mu _{\delta }\right) c_{\varepsilon, \beta
, \lambda _{1}, \lambda _{2}}  \notag \\
&&\times \left[ \exp \left( \lambda _{1}2^{-l}2^{\left[ [\varepsilon k]%
\right] /2+\left[ \beta k\right] }\right) \exp \left( -\frac{\lambda _{1}}{2}%
2^{-l}2^{[\varepsilon k]+[\beta k]}\right) +c\exp \left( -\frac{\lambda _{1}%
}{2}2^{\left[ [\varepsilon k]\right] /2}\right) \right]  \notag \\
&\leq &\left( 1+\lambda _{0}+\mu _{\delta }\right) c_{\varepsilon, \beta
, \lambda _{1}, \lambda _{2}}\exp \left( -\frac{\lambda _{1}}{4}2^{\left[
[\varepsilon k]\right] /2}\right).  \label{DDD005}
\end{eqnarray}

Since $R_{\left( k\right) }^{l, j}, $ $j=1, ..., 2^{l}$ are independent r.v.'s,
by triangle inequality,  one gets%
\begin{eqnarray}
\pi \left( \mathcal{L}_{R_{\left( k\right) }}, \mathcal{L}_{Y_{\left(
k\right) }+V_{\left( k\right) }}\right) &=&\pi \left( \mathcal{L}_{R_{\left(
k\right) }^{0,  0}}, \mathcal{L}_{Y_{\left( k\right) }+V_{\left( k\right)
}}\right)  \notag \\
&\leq & \pi \left( \mathcal{L}_{R_{\left( k\right) }^{0, 0}}, \mathcal{L}%
_{R_{\left( k\right) }^{0, 1}, R_{\left( k\right) }^{0, 2}}\right) +\pi \left(
\mathcal{L}_{R_{\left( k\right) }^{0, 1}, R_{\left( k\right) }^{0, 2}}, \mathcal{%
L}_{Y_{\left( k\right) }+V_{\left( k\right) }}\right)  \notag \\
&\leq & \pi \left( \mathcal{L}_{R_{\left( k\right) }^{0, 0}}, \mathcal{L}%
_{R_{\left( k\right) }^{0, 1}, R_{\left( k\right) }^{0, 2}}\right) +\pi \left(
\mathcal{L}_{R_{\left( k\right) }^{0, 1}, R_{\left( k\right) }^{0, 2}}, \mathcal{%
L}_{R_{\left( k\right) }^{1, 1}, ..., R_{\left( k\right) }^{1, 4}}\right)  \notag
\\
&& \qquad \qquad \qquad +\pi \left( \mathcal{L}_{R_{\left( k\right)
}^{1, 1}, ..., R_{\left( k\right) }^{1, 4}}, \mathcal{L}_{Y_{\left( k\right)
}+V_{\left( k\right) }}\right)  \notag \\
&&...  \notag \\
&\leq &\sum_{l=0}^{\left[ \beta k\right] -1}\pi \left( \mathcal{L}%
_{R_{\left( k\right) }^{l, 1}, ..., R_{\left( k\right) }^{l, 2^{l}}}, \mathcal{L}%
_{R_{\left( k\right) }^{l+1, 1}, ..., R_{\left( k\right)
}^{l+1, 2^{l+1}}}\right).  \label{DDD006}
\end{eqnarray}%
By Lemma \ref{properties 004} and (\ref{DDD005}) one gets%
\begin{eqnarray}
\pi \left( \mathcal{L}_{R_{\left( k\right) }^{l, 1}, ..., R_{\left( k\right)
}^{l, 2^{l}}}, \mathcal{L}_{R_{\left( k\right) }^{l+1, 1}, ..., R_{\left(
k\right) }^{l+1, 2^{l+1}}}\right) &\leq &\sum_{j=1}^{2^{l}}\pi \left(
\mathcal{L}_{R_{\left( k\right) }^{l, j}}, \mathcal{L}_{R_{\left( k\right)
}^{l+1, 2j-1}, R_{\left( k\right) }^{l+1, 2j}}\right)  \notag \\
&\leq &
c_{\varepsilon
, \beta, \lambda _{1}, \lambda _{2}}2^{l}\left( 1+\lambda _{0}+\mu _{\delta }\right)
\\
& & \qquad \qquad\qquad \times  \exp \left( -\frac{\lambda _{1}}{4}%
2^{[\varepsilon k]/2}\right).  \label{DDD007}
\end{eqnarray}%
From (\ref{DDD006}) and (\ref{DDD007}),  it follows%
\begin{eqnarray*}
\pi \left( \mathcal{L}_{R_{\left( k\right) }}, \mathcal{L}_{Y_{\left(
k\right) }+V_{\left( k\right) }}\right) &\leq &\sum_{l=0}^{\left[ \beta k%
\right] -1}2^{l}\left( 1+\lambda _{0}+\mu _{\delta }\right) c_{\varepsilon
, \beta, \lambda _{1}, \lambda _{2}}\exp \left( -\frac{\lambda _{1}}{4}%
2^{[\varepsilon k]/2}\right) \\
&\leq &2^{\left[ \beta k\right] }\left( 1+\lambda _{0}+\mu _{\delta }\right)
c_{\varepsilon, \beta, \lambda _{1}, \lambda _{2}}\exp \left( -\frac{\lambda
_{1}}{4}2^{[\varepsilon k]/2}\right) \\
&\leq & (1+\lambda_0+\mu_\delta)c^{\prime }_{\varepsilon,  \beta,  \lambda_1,
\lambda_2} \exp \left(-{\frac{\lambda_1}{8}}2^{[\varepsilon k]/2}\right). \\
\end{eqnarray*}%
Finally, using (\ref{DDD001}) finishes the proof of the part b).
\end{proof}

\section{Proof of Theorem \protect\ref{Th main res 1}\label{sec Proof Main
Res}}

This section is devoted to the proof of Theorem \ref{Th main res 1}; it is
separated in several steps. We first construct the coupling with independent
r.v.'s. (Section \ref{sec Coupling Indep}) and thus with independent normal
r.v.'s. (Section \ref{sec Coupling Norm}). In Section \ref{sec Construction}%
,  we explicit the construction of the sequences $(\widetilde{X}_{i})_{1\leq
i\leq N}$ and $\left( W_{i}\right) _{1\leq i\leq N}$ and in Sections \ref%
{sec Putting Together},  \ref{bound for the partial sums of gaps},  \ref{Bound
for the oscillation term} and \ref{Optimizing the bounds} we put together
and optimize the bounds.

\subsection{Coupling with independent r.v.'s\label{sec Coupling Indep}}

Assume conditions \textbf{C1} and \textbf{C2}.  The following proposition shows that  the partial sums $%
\sum_{\left( l, i\right) \preceq \left( k, j\right) }X_{\left( l, i\right) }$
can be coupled with high probability with the partial sums $\sum_{\left( l, i\right) \preceq \left(
k, j\right) }Y_{\left( l, i\right) }.$

\begin{proposition}
\label{Prop-Indep 001} Let $%
\alpha <\delta, $ $\beta >\frac{1}{2}$ and $0<\rho <\frac{1-\beta }{2}.$
Then,  for any $N\in \mathbb{N}, $ on some extension of the initial
probability space there is a version $(X_{\left( k, j\right) }^{\prime
})_{\left( k, j\right) \in \mathcal{K}_{N}}$ of the sequence $(X_{\left(
k, j\right) })_{\left( k, j\right) \in \mathcal{K}_{N}}$ and a version $%
(Y_{\left( k, j\right) }^{\prime })_{\left( k, j\right) \in \mathcal{K}_{N}}$
of the sequence $(Y_{\left( k, j\right) })_{\left( k, j\right) \in \mathcal{K}%
_{N}}$ such that%
\begin{eqnarray*}
&&\mathbb{P}\left( \left( 2^{n}\right) ^{-\frac{1}{2}}\sup_{\left(
k, j\right) \in \mathcal{K}_{N}}\left\vert \sum_{\left( l, i\right) \preceq
\left( k, j\right) }\left( X_{\left( l, i\right) }^{\prime }-Y_{\left(
l, i\right) }^{\prime }\right) \right\vert \geq \left( 2^{n}\right) ^{-\rho
}\right) \\
&\leq &C_{1}\left( 2^{n}\right) ^{-1-\alpha +\left( \varepsilon +\rho
\right) \left( 2+2\alpha \right) },
\end{eqnarray*}%
where $\varepsilon \in (0, \frac{1}{2})$ is arbitrary chosen and $%
C_{1}=c_{\varepsilon, \beta, \lambda _{1}, \lambda _{2}, \alpha, \rho }\left(
1+\lambda _{0}+\mu _{\delta }\right) ^{2+2\delta }$ for some positive
constant $c_{\varepsilon, \beta, \lambda _{1}, \lambda _{2}, \alpha, \rho }.$
\end{proposition}

\begin{proof}
For the sake of brevity,  it is convenient to set $k_{0}=[\varepsilon n]$,  $%
X_{k_{0}, n}=\left( X_{\left( k_{0}\right) }, ..., X_{\left( n\right) }\right),$
$Y_{k_{0}, n}=\left( Y_{\left( k_{0}\right) }, ..., Y_{\left( n\right) }\right)
$ and $V_{k_{0}, n}=\left( V_{\left( k_{0}\right) }, ..., V_{\left( n\right)
}\right);$ the variables   $\widetilde{X}_{k_{0}, n}=X_{k_{0}, n}+V_{k_{0}, n}$ and $%
\widetilde{Y}_{k_{0}, n}=Y_{k_{0}, n}+V_{k_{0}, n}$ are  the smoothed versions of
$X_{k_{0}, n}$ and $Y_{k_{0}, n}.$ By Proposition \ref{Prop Aux Main},  with $%
k_{0}=[\varepsilon n], $ there exists a constant $c_{\varepsilon, \beta
, \lambda _{1}, \lambda _{2}}$ such that%
\begin{equation}
\pi \left( \mathcal{L}_{\widetilde{X}_{k_{0}, n}}, \mathcal{L}_{\widetilde{Y}%
_{k_{0}, n}}\right) \leq \Delta =\left( 1+\lambda _{0}+\mu _{\delta }\right)
c_{\varepsilon, \beta, \lambda _{1}, \lambda _{2}}\exp \left( -\frac{\lambda
_{1}}{4}2^{\varepsilon ^{2}n/2}\right).  \label{Prokhorov dist}
\end{equation}%
Using Strassen-Dudley's theorem (see Lemma \ref{lemma strassen-dudley}),  we
conclude that on some extension of the initial probability space there are
random vectors $\widetilde{S}_{k_{0}, n}=\left( S_{\left( k_{0}\right)
}, ..., S_{\left( n\right) }\right) $ and $\widetilde{T}_{k_{0}, n}=\left(
T_{\left( k_{0}\right) }, ..., T_{\left( n\right) }\right) $ such that $%
\widetilde{S}_{k_{0}, n}\overset{d}{=}\widetilde{X}_{k_{0}, n}, $ $\widetilde{T}%
_{k_{0}, n}\overset{d}{=}\widetilde{Y}_{k_{0}, n}$ and
\begin{equation}
\mathbb{P}\left( \left\Vert \widetilde{S}_{k_{0}, n}-\widetilde{T}%
_{k_{0}, n}\right\Vert _{\infty }\geq \Delta \right) \leq \Delta.
\label{Str-Dud}
\end{equation}%
We shall remove the smoothing from the vectors $\widetilde{S}_{k_{0}, n}$ and $%
\widetilde{T}_{k_{0}, n}.$ Without loss of generality we may assume that there
is a random vector $U$ with uniform distribution on $[0, 1]^{m_{k_{0}, n}}$
and independent of $\left( \widetilde{S}_{k_{0}, n}, \widetilde{T}%
_{k_{0}, n}\right).$ We thus consider the transition kernels $G_{1}\left( x|y\right)
:=\mathbb{P}\left( X_{k_{0}, n}\leq x|\widetilde{X}_{k_{0}, n}=y\right) $ and $%
G_{2}\left( x|y\right) :=\mathbb{P}\left( Y_{k_{0}, n}\leq x|\widetilde{Y}%
_{k_{0}, n}=y\right)$ and set  $X_{k_{0}, n}^{\prime }:=G_{1}^{-1}\left( U|%
\widetilde{S}_{k_{0}, n}\right), $ $V_{k_{0}, n}^{\prime }:=\widetilde{S}%
_{k_{0}, n}-X_{k_{0}, n}^{\prime }$, $Y_{k_{0}, n}^{\prime
}:=G_{1}^{-1}\left( U|T_{k_{0}, n}\right)$ and $V_{k_{0}, n}^{\prime \prime }:=%
\widetilde{T}_{k_{0}, n}-Y_{k_{0}, n}^{\prime }.$ The two sequences $%
X_{k_{0}, n}^{\prime }$ and $Y_{k_{0}, n}^{\prime }$ are such that $\widetilde{%
S}_{k_{0}, n}=X_{k_{0}, n}^{\prime }+{V}_{k_{0}, n}^{\prime }, $ $\widetilde{T}%
_{k_{0}, n}=Y_{k_{0}, n}^{\prime }+V_{k_{0}, n}^{\prime \prime }$ and $%
X_{k_{0}, n}^{\prime }\overset{d}{=}X_{k_{0}, n}, $ $Y_{k_{0}, n}^{\prime }%
\overset{d}{=}Y_{k_{0}, n}, $ $V_{k_{0}, n}^{\prime }\overset{d}{=}%
V_{k_{0}, n}^{\prime \prime }\overset{d}{=}V_{k_{0}, n}.$ The coordinates of
the vectors $X_{k_{0}, n}^{\prime }$ and $Y_{k_{0}, n}^{\prime }$ are denoted
by $X_{\left( k, j\right) }^{\prime }$ and $Y_{\left( k, j\right) }^{\prime }, $
$\left( k, j\right) \in \mathcal{K}.$ Since $\widetilde{S}_{\left( k, j\right)
}=X_{\left( k, j\right) }^{\prime }+V_{\left( k, j\right) }^{\prime }$ and $%
\widetilde{T}_{\left( k, j\right) }=Y_{\left( k, j\right) }^{\prime
}+V_{\left( k, j\right) }^{\prime \prime }, $ we have,  for any $x\geq 1, $
\begin{equation*}
R=\mathbb{P}\left( \sup_{k_{0}\leq k, ~\left( k, j\right) \in \mathcal{K}%
_{N}}\left\vert \sum_{\left( l, i\right) \preceq \left( k, j\right) }\left(
X_{\left( l, i\right) }^{\prime }-Y_{\left( l, i\right) }^{\prime }\right)
\right\vert \geq 2x\right) \leq R_{1}+R_{2},
\end{equation*}%
where
\begin{eqnarray*}
R_{1} &=&\mathbb{P}\left( \sup_{k_{0}\leq k, ~\left( k, j\right) \in \mathcal{K%
}_{N}}\left\vert \sum_{\left( l, i\right) \preceq \left( k, j\right) }%
\widetilde{S}_{\left( l, i\right) }-\widetilde{T}_{\left( l, i\right)
}\right\vert \geq x\right),  \\
R_{2} &=&\mathbb{P}\left( \sup_{k_{0}\leq k, ~\left( k, j\right) \in \mathcal{K%
}_{N}}\left\vert \sum_{\left( l, i\right) \preceq \left( k, j\right) }\left(
V_{\left( l, i\right) }^{\prime }-V_{\left( l, i\right) }^{\prime \prime
}\right) \right\vert \geq x\right).
\end{eqnarray*}

First,  we shall give a control for $R_{1}.$ Note that \textrm{card}$\
\mathcal{K}_{N}\leq c2^{\beta n}.$ For any sequence of positive numbers $%
\left( \alpha _{\left( k, j\right) }\right) _{\left( k, j\right) \in \mathcal{K%
}}$ such that $\sum_{\left( k, j\right) \in \mathcal{K}}\alpha _{\left(
k, j\right) }\leq 1, $ it holds%
\begin{eqnarray*}
&&\left\{ \sup_{k_{0}\leq k, ~\left( k, j\right) \in \mathcal{K}%
_{N}}\left\vert \sum_{k_{0}\leq l, ~\left( l, i\right) \preceq \left(
k, j\right) }\left( \widetilde{S}_{\left( l, i\right) }-\widetilde{T}_{\left(
l, i\right) }\right) \right\vert \geq x\right\} \\
&\subseteq &\bigcup_{k_{0}\leq k, ~\left( k, j\right) \in \mathcal{K}%
_{N}}\left\{ \left\vert \sum_{k_{0}\leq l, ~\left( l, i\right) \preceq \left(
k, j\right) }\left( \widetilde{S}_{\left( l, i\right) }-\widetilde{T}_{\left(
l, i\right) }\right) \right\vert \geq x\right\} \\
&\subseteq &\bigcup_{k_{0}\leq k, ~\left( k, j\right) \in \mathcal{K}%
_{N}}\bigcup\limits_{k_{0}\leq k, ~\left( l, i\right) \preceq \left(
k, j\right) }\left\{ \left\vert \widetilde{S}_{\left( l, i\right) }-\widetilde{%
T}_{\left( l, i\right) }\right\vert \geq x\alpha _{\left( l, i\right) }\right\}
\\
&=&\bigcup\limits_{k_{0}\leq k, ~\left( k, j\right) \in \mathcal{K}%
_{N}}\left\{ \left\vert \widetilde{S}_{\left( k, j\right) }-\widetilde{T}%
_{\left( k, j\right) }\right\vert \geq x\alpha _{\left( k, j\right) }\right\},
\end{eqnarray*}%
which implies that%
\begin{equation*}
R_{1}\leq \sum_{k_{0}\leq k, ~\left( k, j\right) \in \mathcal{K}_{N}}\mathbb{P}%
\left( \left\vert \widetilde{S}_{\left( k, j\right) }-\widetilde{T}_{\left(
k, j\right) }\right\vert \geq x\alpha _{\left( k, j\right) }\right).
\end{equation*}%
Let $p=2+2\alpha <2+2\delta.$ By Chebyshev's inequality%
\begin{equation*}
R_{1}\leq x^{-p}\sum_{k_{0}\leq k, ~\left( k, j\right) \in \mathcal{K}%
_{N}}\alpha _{\left( k, j\right) }^{-p}\mathbb{E}\left\vert \widetilde{S}%
_{\left( k, j\right) }-\widetilde{T}_{\left( k, j\right) }\right\vert ^{p}.
\end{equation*}%
By a truncation argument,  with $\Delta $ from (\ref{Prokhorov dist}) and (%
\ref{Str-Dud}), %
\begin{eqnarray*}
R_{1} &\leq &x^{-p}\Delta ^{p}\sum_{k_{0}\leq k, ~\left( k, j\right) \in
\mathcal{K}_{N}}\alpha _{\left( k, j\right) }^{-p} \\
&&+x^{-p}\sum_{k_{0}\leq k, ~\left( k, j\right) \in \mathcal{K}_{N}}\alpha
_{\left( k, j\right) }^{-p}\mathbb{E}\left\vert \widetilde{S}_{\left(
k, j\right) }-\widetilde{T}_{\left( k, j\right) }\right\vert ^{p}1\left(
\left\vert \widetilde{S}_{\left( k, j\right) }-\widetilde{T}_{\left(
k, j\right) }\right\vert \geq \Delta \right).
\end{eqnarray*}%
Let $\eta \in (0, \delta -\alpha ), $ $p^{\prime }=p+2\eta $ and $\gamma =%
\frac{2\eta }{p+2\eta }\leq \eta.$ Applying H\"{o}lder's inequality one may
write%
\begin{eqnarray*}
&&\left\Vert\left\vert \widetilde{S}_{\left( k, j\right) }-\widetilde{T}_{\left( k, j\right) }\right\vert 1\left(  \left\vert \widetilde{S}_{\left(
k, j\right) }-\widetilde{T}_{\left( k, j\right) }\right\vert \geq \Delta\right) \right\Vert _{L^{p}}\\
&& \qquad\qquad\qquad\qquad\qquad\qquad \leq \left\Vert \widetilde{S}_{\left( k, j\right)}-\widetilde{T}_{\left( k, j\right) }\right\Vert _{L^{p^{\prime }}}
\mathbb{P}\left( \left\vert \widetilde{S}_{\left( k, j\right) }-\widetilde{T}_{\left(k, j\right) }\right\vert >\Delta \right) ^{\frac{\gamma }{p}}.
\end{eqnarray*}%
By Condition \textbf{C2},  for some constant $c>0, $ we get%
\begin{equation*}
\left\Vert \widetilde{S}_{\left( k, j\right) }-\widetilde{T}_{\left(
k, j\right) }\right\Vert _{L^{p^{\prime }}}\leq 2\left\Vert X_{\left(
k, j\right) }\right\Vert _{L^{p^{\prime }}}+2\left\Vert V_{\left( k, j\right)
}^{\prime }\right\Vert _{L^{p^{\prime }}}\leq c\left( 1+\mu _{\delta
}\right) \left\vert I_{k, j}\right\vert ;
\end{equation*}%
consequently,  using (\ref{Str-Dud})%
\begin{eqnarray*}
R_{1} &\leq &x^{-p}\Delta ^{p}\sum_{k_{0}\leq k, ~\left( k, j\right) \in
\mathcal{K}_{N}}\alpha _{\left( k, j\right) }^{-p} \\
&&+c\left( 1+\mu _{\delta }\right) ^{p}x^{-p}\sum_{k_{0}\leq k, ~\left(
k, j\right) \in \mathcal{K}_{N}}\alpha _{\left( k, j\right) }^{-p}\left\vert
I_{k, j}\right\vert ^{p}\left( \mathbf{P}\left( \left\vert \widetilde{S}%
_{\left( k, j\right) }-\widetilde{T}_{\left( k, j\right) }\right\vert \geq
\Delta \right) \right) ^{\gamma } \\
&\leq &x^{-p}\Delta ^{p}\sum_{k_{0}\leq k, ~\left( k, j\right) \in \mathcal{K}%
_{N}}\alpha _{\left( k, j\right) }^{-p} \\
&&+c\left( 1+\mu _{\delta }\right) ^{p}x^{-p}\Delta ^{\gamma
}\sum_{k_{0}\leq k, ~\left( k, j\right) \in \mathcal{K}_{N}}\alpha _{\left(
k, j\right) }^{-p}2^{(k-\left[ \beta k\right] )p} \\
&\leq &c_{\varepsilon, \beta, \lambda _{1}, \lambda _{2}, \eta }\left(
1+\lambda _{0}+\mu _{\delta }\right) ^{p+\gamma }\exp \left( -\frac{\lambda
_{1}}{4}\gamma 2^{\varepsilon ^{2}n/2}\right) x^{-p}\sum_{k_{0}\leq k\leq
n}\sum_{j\leq 2^{\left[ \beta k\right] }}\alpha _{\left( k, j\right)
}^{-p}2^{kp}.
\end{eqnarray*}%
Now,  choosing $\alpha _{\left( k, j\right) }=2^{-k}j^{-2}, $ we obtain%
\begin{eqnarray*}
\sum_{k_{0}\leq k\leq n}\sum_{j\leq 2^{\left[ \beta k\right] }}\alpha
_{\left( k, j\right) }^{-p}2^{kp} &\leq &\sum_{k_{0}\leq k\leq n}\sum_{j\leq
2^{\left[ \beta k\right] }}2^{2kp}j^{2p}\leq 2^{2np}\sum_{k_{0}\leq k\leq
n}\sum_{j\leq 2^{\left[ \beta k\right] }}j^{2p} \\
&\leq &2^{2np}\sum_{k_{0}\leq k\leq n}2^{\left( 2p+1\right) \left[ \beta k%
\right] }\leq 2^{2np}2^{\left( 2p+1\right) \left[ \beta n\right] }n\leq
2^{nc_{\alpha, \beta }},
\end{eqnarray*}%
which implies that%
\begin{equation*}
R_{1}\leq c_{\varepsilon, \beta, \lambda _{1}, \lambda _{2}, \eta }\left(
1+\lambda _{0}+\mu _{\delta }\right) ^{p+\gamma }\exp \left( -\frac{1}{4}%
\gamma \lambda _{1}2^{\varepsilon ^{2}n/2}\right) 2^{nc_{\alpha, \beta
}}x^{-p}.
\end{equation*}%
Since $\gamma =\frac{2\eta }{p+2\eta }\leq \eta \leq p\eta $ and $x\geq 1, $
we conclude that%
\begin{equation}
R_{1}\leq A^{\prime }\exp \left( -\frac{1}{4}\gamma \lambda _{1}\left(
2^{n}\right) ^{\varepsilon ^{2}/2}\right)  \label{R11a}
\end{equation}%
for some $A^{\prime }=c_{\varepsilon, \beta, \lambda _{1}, \lambda
_{2}, \alpha, \alpha ^{\prime }, \eta }^{\prime }\left( 1+\lambda _{0}+\mu
_{\delta }\right) ^{p\left( 1+\eta \right) }.$ Now we give a control for $%
R_{2}.$ Using Doob's inequality,  for any $\lambda >2, $%
\begin{eqnarray}
R_{2} &\leq &2\mathbb{P}\left( \sup_{k_{0}\leq k, ~\left( k, j\right) \in
\mathcal{K}_{N}}\left\vert \sum_{k_{0}\leq l, ~\left( l, i\right) \preceq
\left( k, j\right) }V_{\left( l, i\right) }^{\prime }\right\vert \geq x\right)
\notag \\
&\leq &2x^{-\lambda }\mathbb{E}\left( \sum_{k_{0}\leq l, ~\left( l, i\right)
\in \mathcal{K}_{N}}\left\vert V_{\left( l, i\right) }^{\prime }\right\vert
\right) ^{\lambda }.  \label{R12a}
\end{eqnarray}%
By Rosenthal's inequality
\begin{eqnarray}
\left( \mathbb{E}\left( \sum_{k_{0}\leq l, ~\left( l, i\right) \in \mathcal{K}%
_{N}}\left\vert V_{\left( l, i\right) }^{\prime }\right\vert \right)
^{\lambda }\right) ^{1/\lambda } &\leq &c_{\lambda }\left( \sum_{k_{0}\leq
l, ~\left( l, i\right) \in \mathcal{K}_{N}}\mathbb{E}\left( \left\vert
V_{\left( l, i\right) }^{\prime }\right\vert ^{2}\right) \right) ^{1/2}
\notag \\
&&+c_{\lambda }\left( \sum_{k_{0}\leq l, ~\left( l, i\right) \in \mathcal{K}%
_{N}}\mathbb{E}\left( \left\vert V_{\left( l, i\right) }^{\prime }\right\vert
^{\lambda }\right) \right) ^{1/\lambda }  \notag \\
&\leq &c_{\lambda }^{\prime }\left( 2^{\beta n}\right) ^{1/2}.  \label{R12b}
\end{eqnarray}%
From (\ref{R11a}),  (\ref{R12a}) and (\ref{R12b}) we obtain%
\begin{eqnarray*}
&&\mathbb{P}\left( \sup_{k_{0}\leq k, ~\left( k, j\right) \in \mathcal{K}%
_{N}}\left\vert \sum_{\left( l, i\right) \preceq \left( k, j\right) }\left(
X_{\left( l, i\right) }^{\prime }-Y_{\left( l, i\right) }^{\prime }\right)
\right\vert \geq 2x\right) \\
&\leq &A^{\prime }\exp \left( -\frac{1}{4}\gamma \lambda _{1}\left(
2^{n}\right) ^{\varepsilon ^{2}/2}\right) +c_{\lambda }\left( 2^{\beta
n}\right) ^{\lambda /2}x^{-\lambda }.
\end{eqnarray*}%
Choosing $x=\frac{1}{2}\left( 2^{n}\right) ^{\frac{1}{2}-\rho }, $ we find%
\begin{eqnarray}
&&\mathbb{P}\left( \left( 2^{n}\right) ^{-\frac{1}{2}}\sup_{k_{0}\leq
k, ~\left( k, j\right) \in \mathcal{K}_{N}}\left\vert \sum_{\left( l, i\right)
\preceq \left( k, j\right) }\left( X_{\left( l, i\right) }^{\prime }-Y_{\left(
l, i\right) }^{\prime }\right) \right\vert \geq \left( 2^{n}\right) ^{-\rho
}\right)  \notag \\
&\leq &A^{\prime }\exp \left( -\frac{1}{4}\gamma \lambda _{1}\left(
2^{n}\right) ^{\varepsilon ^{2}/2}\right) +c_{\lambda }\left( 2^{n}\right)
^{-\frac{1}{2}\lambda \left( 1-\beta -2\rho \right) }.  \label{R15a}
\end{eqnarray}%

So far we performed the construction for $k \geq k_0.$ It remains to construct the sequences $X_{\left( k,j\right) }^{\prime }$ and
$Y_{\left( k,j\right) }^{\prime }$ for $\left( k,j\right) \preceq \left(k_{0}-1,m_{k_{0}-1}\right).$
This construction can be performed by any method such that the sequences
$(X_{\left(k,j\right) }^{\prime })$ and $(Y_{\left( k,j\right) }^{\prime})$, where $\left( k,j\right) \preceq \left(k_{0}-1,m_{k_{0}-1}\right),$
are independent and $Y_{\left( k,j\right) }^{\prime }\overset{d}{=}X_{\left(k,j\right) }$ for the same $\left( k,j\right).$
Indeed, let $F_{X|Y_{1},...,Y_{k}}\left(x|y_{1},...,y_{k}\right) $ be the conditional distribution of $X$ given $%
[Y_{1}=y_{1},...,Y_{k}=y_{k}]$ and let
$ (U_{\left( k,j\right) })$ be a sequence of independent r.v.'s uniformly distributed on $\left( 0,1\right) .$
Denote for brevity the constructed part by $\mathbf{X}_{k_{0}}^{\prime }=(X_{\left( k,j\right) }^{\prime }) _{k_{0}\leq k,~\left( k,j\right) \in \mathcal{K}_{N}}.$
Define $X_{\left( k_{0}-1,1\right) }^{\prime }$ as the
conditional quantile transform
$$
X_{\left( k_{0}-1,1\right) }^{\prime}=F_{X_{\left( k_{0}-1,1\right) }|\mathbf{X}_{k_{0}}}^{-1}\left( U_{\left(k_{0}-1,1\right) }|\mathbf{X}_{k_{0}}^{\prime }\right) ,$$
where $\mathbf{X}_{k_{0}}=\left( X_{\left( k,j\right) }\right) _{k_{0}\leq k,~\left(
k,j\right) \in \mathcal{K}_{N}}.$
We continue by setting
$$
X_{\left(
k_{0}-1,j\right) }^{\prime }=F_{X_{\left( k,j\right) }|\mathbf{X}%
_{k_{0}}}^{-1}\left( U_{\left( k_{0}-1,j-1\right) }|X_{\left(
k_{0}-1,1\right) }^{\prime },...,X_{\left( k_{0}-1,j-1\right) }^{\prime }%
\mathbf{X}_{k_{0}}^{\prime }\right) ,$$
for $j=2,...,m_{k_{0}-1}.$
In the same way we extend the construction to all $X_{\left( k,j\right) }^{\prime }$
with $1\leq k<k_{0}-1.$
The construction of the sequence $\left( Y_{\left(k,j\right) }^{\prime }\right),$ for $\left( k,j\right) \preceq \left(k_{0}-1,m_{k_{0}-1}\right)$ is similar.

Since the sequence $\left(X_{k}\right) _{k\geq 1}$ satisfies Condition \textbf{C1},  so does the sequence $(X_{\left( k, j\right) }^{\prime }).$
Using the maximal inequality stated in
Proposition \ref{Prop max Lp bound} and noting that the cardinality of the
set $\left\{ \left( k, j\right) :\left( k, j\right) \preceq \left(
k_{0}-1, m_{k_{0}-1}\right) \right\} $ is less or equal to $2^{\beta
k_{0}}\leq 2^{\varepsilon n}, $ we obtain,  for any $\eta ^{\prime }\in \left(
0, \frac{\delta -\alpha }{\left( 2+\alpha +\delta \right) ^{2}}\right), $%
\begin{equation*}
\mathbb{E}\left( \sup_{\left( k, j\right) \preceq \left(
k_{0}-1, m_{k_{0}-1}\right) }\left\vert \sum_{\left( l, i\right) \preceq
\left( k, j\right) }X_{\left( l, i\right) }^{\prime }\right\vert ^{p}\right)
\leq A^{\prime \prime }\left( 2^{\varepsilon n}\right) ^{\frac{1}{2}p},
\end{equation*}%
for come constant $A^{\prime \prime }=c_{\varepsilon, \beta, \lambda
_{1}, \lambda _{2}, \delta, \alpha, \eta }^{\prime \prime }\left( 1+\lambda
_{0}+\mu _{\delta }\right) ^{p\left( 1+\eta ^{\prime }\right) }.$ By
Chebyshev's inequality,  for any $x>0$ we get%
\begin{eqnarray*}
\mathbb{P}\left( \sup_{\left( k, j\right) \preceq \left(
k_{0}-1, m_{k_{0}-1}\right) }\left\vert \sum_{\left( l, i\right) \preceq
\left( k, j\right) }X_{\left( l, i\right) }^{\prime }\right\vert \geq x\right)
&\leq &x^{-p}\mathbb{E}\sup_{\left( k, j\right) \preceq \left(
k_{0}-1, m_{k_{0}-1}\right) }\left\vert \sum_{\left( l, i\right) \preceq
\left( k, j\right) }X_{\left( l, i\right) }^{\prime }\right\vert ^{p} \\
&\leq &A^{\prime \prime }x^{-p}\left( 2^{\varepsilon n}\right) ^{\frac{1}{2}%
p}.
\end{eqnarray*}%
Substituting $x=\left( 2^{n}\right) ^{\frac{1}{2}-\rho }$ yields%
\begin{equation*}
\mathbb{P}\left( \left( 2^{n}\right) ^{-\frac{1}{2}}\sup_{\left( k, j\right)
\preceq \left( k_{0}-1, m_{k_{0}-1}\right) }\left\vert \sum_{\left(
l, i\right) \preceq \left( k, j\right) }X_{\left( l, i\right) }^{\prime
}\right\vert \geq \left( 2^{n}\right) ^{-\rho }\right) \leq A^{\prime \prime
}\left( 2^{n}\right) ^{-\frac{p}{2}+p\left( \rho +\frac{1}{2}\varepsilon
\right) }.
\end{equation*}%
A similar inequality can be proved with $Y_{\left( l, i\right) }^{\prime }$
instead of $X_{\left( l, i\right) }^{\prime }.$ Combining this with (\ref%
{R15a}) we obtain, %
\begin{eqnarray}
&&\mathbb{P}\left( \left( 2^{n}\right) ^{-\frac{1}{2}}\sup_{\left(
k, j\right) \in \mathcal{K}_{N}}\left\vert \sum_{\left( l, i\right) \preceq
\left( k, j\right) }\left( X_{\left( l, i\right) }^{\prime }-Y_{\left(
l, i\right) }^{\prime }\right) \right\vert \geq 2\left( 2^{n}\right) ^{-\rho
}\right)  \notag \\
&\leq &A^{\prime \prime \prime }\left( \exp \left( -\frac{1}{4}\gamma
\lambda _{1}\left( 2^{n}\right) ^{\varepsilon ^{2}/2}\right) +\left(
2^{n}\right) ^{-\frac{1}{2}\lambda \left( 1-\beta -2\rho \right) }+\left(
2^{n}\right) ^{-\frac{p}{2}+p\left( \rho +\frac{1}{2}\varepsilon \right)
}\right),   \label{R17}
\end{eqnarray}%
for some $A^{\prime \prime }=c_{\varepsilon, \beta, \lambda _{1}, \lambda
_{2}, \delta, \delta ^{\prime }, \eta, \lambda }^{\prime \prime \prime }\left(
1+\lambda _{0}+\mu _{\delta }\right) ^{p\left( 1+\eta +\eta ^{\prime
}\right) }.$ Recall that $p=2+2\alpha, $ $\alpha <\delta, $ $\beta >\frac{1}{%
2}$ and $\rho <\frac{1-\beta }{2}.$ Taking $\lambda =\frac{2+2\alpha }{%
1-\beta -2\rho }>p, $ the right-hand side of (\ref{R17}) does not exceed $%
A^{\prime \prime \prime }\left( 2^{n}\right) ^{-1-\alpha +\left( 2+2\alpha
\right) \left( \rho +\varepsilon \right) }.$ It remains to choose a
sufficiently small $\eta +\eta ^{\prime }$ such that $p\left( 1+\eta +\eta
^{\prime }\right) \leq 2+2\delta, $ which implies%
\begin{equation*}
A^{\prime \prime \prime }\leq c_{\varepsilon, \beta, \lambda _{1}, \lambda
_{2}, \alpha, \rho }^{\prime \prime \prime }\left( 1+\lambda _{0}+\mu
_{\delta }\right) ^{2+2\delta }.
\end{equation*}%
The assertion of the proposition follows.
\end{proof}

\subsection{Coupling with independent normal r.v.'s\label{sec Coupling Norm}}

Assume conditions \textbf{C1},  \textbf{C2} and \textbf{C3}.
Without loss of generality we can consider that $\mu _{i}=0, $ $i\geq 1$ and $\sigma =1.$
The following proposition shows that with high probability the partial sums $\sum_{\left(
l, i\right) \preceq \left( k, j\right) }X_{\left( l, i\right) }$ can be coupled
with the partial sums of some normal r.v.'s.
Note the presence of two terms in the upper bound below.
One of them called {\it dependence error} comes from replacing dependent blocks by independent ones;
the second one called {\it Sakhanenko's error } is due to the use of Sakhanenko's strong approximation result for the independent blocks.

\begin{proposition}
\label{Prop-Normal 001} Let $\alpha
<\delta, $ $\beta >\frac{1}{2}$ and $0<\rho <\frac{1-\beta }{2}.$ Then for
any $N\in \mathbb{N}, $ on some extension of the initial probability space
there exists a sequence of independent standard normal r.v.'s $(W_{\left(
k, j\right) }^{\prime })_{\left( k, j\right) \in \mathcal{K}_{N}}$ and a
version $(X_{\left( k, j\right) }^{\prime })_{\left( k, j\right) \in \mathcal{K%
}_{N}}$ of the sequence $(X_{\left( k, j\right) })_{\left( k, j\right) \in
\mathcal{K}_{N}}$ such that
\begin{eqnarray}
&&\mathbb{P}\left( \left( 2^{n}\right) ^{-\frac{1}{2}}\sup_{\left(
k, j\right) \in \mathcal{K}_{N}}\left\vert \sum_{\left( l, i\right) \preceq
\left( k, j\right) }\left( X_{\left( l, i\right) }^{\prime }-\sigma
_{l, i}W_{\left( l, i\right) }^{\prime }\right) \right\vert \geq 2\left(
2^{n}\right) ^{-\rho }\right)  \notag \\
&\leq &C_{2}\left( 2^{n}\right) ^{-1-\alpha +\left( \varepsilon +\rho
\right) \left( 2+2\alpha \right) }\quad \quad \quad \text{(dependence error)}
\notag \\
&&+C_{2}\left( 2^{n}\right) ^{-\beta \alpha +\rho \left( 2+2\alpha \right)
}\quad \quad \quad \text{(Sakhanenko's error)}  \label{propnorm001}
\end{eqnarray}%
where $\sigma _{l, i}^{2}=\mathrm{Var}\left( X_{\left( l, i\right) }\right)$ and $C_{2}=c_{\varepsilon, \beta, \lambda _{1}, \lambda _{2}, \alpha, \rho
}\left( 1+\lambda _{0}+\mu _{\delta }\right) ^{2+2\delta }.$
\end{proposition}

\begin{proof}
Let $p=2+2\alpha.$ Since $\left\vert I_{k, j}\right\vert \leq 2^{k-\left[
\beta k\right] }, $ we obtain,  using Proposition \ref{Prop Lp bound}
\begin{equation*}
\mathbb{E}\left\vert X_{\left( k, j\right) }\right\vert ^{p}\leq A\left\vert
I_{k, j}\right\vert ^{\frac{p}{2}}\leq A\left( 2^{k-\left[ \beta k\right]
}\right) ^{\frac{p}{2}},
\end{equation*}%
where $A=c_{\lambda _{1}, \lambda _{2}, \delta ^{\prime }, \eta }\left(
1+\lambda _{0}+\mu _{\delta }\right) ^{p\left( 1+\eta \right) }$ and $\eta
>0 $ is arbitrary. Taking into account that $m_{k}=2^{\left[ \beta k\right]
}\leq 2^{\beta k}, $ we have%
\begin{eqnarray}
\sum_{\left( k, j\right) \in \mathcal{K}_{N}}\mathbb{E}\left\vert X_{\left(
k, j\right) }\right\vert ^{p} &=&\sum_{k=1}^{n}\sum_{j=1}^{m_{k}}\mathbb{E}%
\left\vert X_{\left( k, j\right) }\right\vert ^{p}  \notag \\
&\leq &\sum_{k=1}^{n}m_{k}A\left( 2^{k-\left[ \beta k\right] }\right) ^{%
\frac{p}{2}}  \notag \\
&\leq &A2^{\frac{p}{2}}\sum_{k=1}^{n}2^{k\left( \beta +\frac{p}{2}\left(
1-\beta \right) \right) }  \notag \\
&\leq &c_{\alpha, \beta }A\left( 2^{n}\right) ^{\beta +\frac{p}{2}\left(
1-\beta \right) }.  \label{coupling norm002}
\end{eqnarray}%
By (\ref{SakhLp}) in Section \ref{subsectionSakh},  on some probability space
$\left( \Omega ^{\prime \prime }, \mathcal{F}^{\prime \prime }, \mathbb{P}%
^{\prime \prime }\right), $ there exist a version $(Y_{\left( k, j\right)
}^{\prime \prime })_{\left( k, j\right) \in \mathcal{K}_{N}}$ of the sequence
$(Y_{\left( k, j\right) })_{\left( k, j\right) \in \mathcal{K}_{N}}$ and
independent standard normal r.v.'s $(W_{\left( k, j\right) }^{\prime
})_{\left( k, j\right) \in \mathcal{K}_{N}}$ such that%
\begin{equation*}
\mathbb{P}^{\prime \prime }\left( \sup_{\left( k, j\right) \in \mathcal{K}%
_{N}^{0}}\left\vert \sum_{\left( l, i\right) \preceq \left( k, j\right)
}\left( Y_{\left( l, i\right) }^{\prime \prime }-\sigma _{l, i}W_{\left(
l, i\right) }^{\prime }\right) \right\vert \geq a\right) \leq \frac{c_{p}}{%
a^{p}}\sum_{\left( k, j\right) \in \mathcal{K}_{N}^{0}}\mathbb{E}\left\vert
X_{\left( k, j\right) }\right\vert ^{p}.
\end{equation*}

Choosing $a=\left( 2^{\beta n}\right) ^{\frac{1}{2}-\rho }$ and taking into
account (\ref{coupling norm002}) we obtain%
\begin{eqnarray*}
&&\mathbb{P}^{\prime \prime }\left( \left( 2^{n}\right) ^{-\frac{1}{2}%
}\sup_{\left( k, j\right) \in \mathcal{K}_{N}^{0}}\left\vert \sum_{\left(
l, i\right) \preceq \left( k, j\right) }\left( Y_{\left( l, i\right) }^{\prime
\prime }-\sigma _{l, i}W_{\left( l, i\right) }^{\prime }\right) \right\vert
\geq \left( 2^{n}\right) ^{-\rho }\right) \\
&\leq &c_{p}\left( 2^{n}\right) ^{-\frac{p}{2}+\rho p}c_{\alpha, \beta
}A\left( 2^{n}\right) ^{\beta +\frac{p}{2}\left( 1-\beta \right) } \\
&\leq &c_{\alpha, \beta }^{\prime }A\left( 2^{n}\right) ^{-\beta \alpha
+\rho \left( 2+2\alpha \right) }.
\end{eqnarray*}%
By Berkes-Philip's lemma ( Lemma 2.1 of \cite{BerkPhilip}) we can reconstruct on some new probability space the sequences $%
(X_{\left( k, j\right) }^{\prime })_{\left( k, j\right) \in \mathcal{K}_{N}}$,
$(Y_{\left( k, j\right) }^{\prime })_{\left( k, j\right) \in \mathcal{K}_{N}}$%
,  $(Y_{\left( k, j\right) }^{\prime \prime })_{\left( k, j\right) \in \mathcal{%
K}_{N}}$ and $(W_{\left( k, j\right) }^{\prime })_{\left( k, j\right) \in
\mathcal{K}_{N}}$  in such a way that
$
Y_{\left( k, j\right) }^{\prime }=Y_{\left( k, j\right) }^{\prime \prime}$ a.s. for any $\left( k, j\right) \in \mathcal{K}_{N}.$ %
Without loss of generality we shall consider this new probability space as
an extension of the initial probability space. Using Proposition \ref%
{Prop-Indep 001} we obtain%
\begin{eqnarray*}
&&\mathbb{P}\left( 2^{-\frac{n}{2}}\sup_{\left( k, j\right) \in \mathcal{K}%
_{N}}\left\vert \sum_{\left( l, i\right) \preceq \left( k, j\right) }\left(
X_{\left( l, i\right) }^{\prime }-\sigma _{l, i}W_{\left( l, i\right) }^{\prime
}\right) \right\vert \geq 2\left( 2^{n}\right) ^{-\rho }\right) \\
&\leq &C_{1}\left( 2^{n}\right) ^{-1-\alpha +\left( \varepsilon +\rho
\right) \left( 2+2\alpha \right) }\quad \;\text{(dependence error)} \\
&&+c_{\delta, \beta }^{\prime }A\left( 2^{n}\right) ^{-\beta \alpha +\rho
\left( 2+2\alpha \right) }\quad \;\text{(Sakhanenko's error)}
\end{eqnarray*}%
with $C_{1}$ defined by Proposition \ref{Prop-Indep 001}. Taking into
account that $p=2+2\alpha, $ $\alpha <\delta $ and choosing $\eta $
sufficiently small we get $p\left( 1+\eta \right) \leq 2+2\delta, $ which
implies $c_{\alpha, \beta }^{\prime }A\leq c_{\lambda _{1}, \lambda
_{2}, \alpha, \eta }^{\prime }\left( 1+\lambda _{0}+\mu _{\delta }\right)
^{2+2\delta }.$ This concludes the proof.
\end{proof}

\subsection{Construction of the sequences $(\widetilde{X}_{i})_{1\leq i\leq
N}$ and $\left( W_{i}\right) _{1\leq i\leq N}$\label{sec Construction}}

As before,  we suppose that $\mu _{i}=0, $ $i\geq 1$ and $\sigma ^{2}=1.$ Let $%
(X_{\left( k, j\right) }^{\prime })_{\left( k, j\right) \in \mathcal{K}_{N}}$
and $(W_{\left( k, j\right) }^{\prime })_{\left( k, j\right) \in \mathcal{K}%
_{N}}$ be the sequences which appear in Proposition \ref{Prop-Normal 001}.

First we shall construct the sequence $\left( W_{i}\right) _{1\leq i\leq N}.$
Note that,  by condition \textbf{C3},  the variances $\sigma _{k, i}^{2}=%
\mathrm{Var}\left( X_{\left( k, j\right) }\right) $ can be approximated by $%
\sigma ^{2}\left\vert I_{k, j}\right\vert =\left\vert I_{k, j}\right\vert, $
but in general do not coincide with $\left\vert I_{k, j}\right\vert.$
Therefore to perform our construction we have to replace each of the non
identically distributed normal random variables $\sigma _{k, j}^{2}W_{\left(
k, j\right) }^{\prime }$ by some sums of independent identically distributed
standard normal random variables.\ Let $\left( W_{i}\right) _{1\leq i\leq N}$
be a sequence of independent standard normal r.v.'s; let $\xi _{k, j}$ be an
extra standard normal random variable. Set $I_{k, j}:=\{i_{1}, \cdots
, i_{\left\vert I_{k, j}\right\vert }\}$ with $i_{1}\preceq...\preceq
i_{\left\vert I_{k, j}\right\vert }$ and let $i_{k, j}^{\ast }$ be the maximal
index $j\in \left\{ i_{1}, ..., i_{\left\vert I_{k, j}\right\vert }\right\} $
for which the variance of the partial sum $\sum_{i=i_{1}}^{j}W_{i}$ does not
exceed $\sigma _{k, j}^{2}, $ i.e. $i_{k, j}^{\ast }=i_{m_{k, j}^{\ast }}, $
where $m_{k, j}^{\ast }=\min \left\{ \left\vert I_{k, j}\right\vert, \left[
\sigma _{k, j}^{2}\right] \right\}.$

It is easy to check that   $W_{\left( k, j\right) }^{\prime
\prime }:=\sum_{i=i_{1}}^{i_{k, j}^{\ast }}W_{i}+\xi _{k, j}f_{k, j}$ where $%
f_{k, j}^{2}=\left\vert \sigma _{k, j}^{2}-i_{k, j}^{\ast }\right\vert$ is a normal random variable
 with  mean $0$ and variance $\sigma _{k, j}^{2}$; furthermore, we may consider $W_{\left( k, j\right)
}^{\prime \prime }$ as a new version of $\sigma _{k, j}W_{\left( k, j\right)
}^{\prime }.$  The random variable $%
\sum_{k\in I_{k, j}}W_{k}-W_{\left( k, j\right) }^{\prime \prime
}$, which is equal to $\sum_{i_{k, j}^{\ast }+1\leq k\leq \left\vert I_{k, j}\right\vert }W_{k}-\xi
_{k, j}f_{k, j}$, has also a normal random variable of mean $0$ and variance $\left(
\left\vert I_{k, j}\right\vert -i_{k, j}^{\ast }\right)
+f_{k, j}^{2}=\left\vert \sigma _{k, j}^{2}-\left\vert I_{k, j}\right\vert
\right\vert.$  By Berkes-Philip's lemma ( Lemma 2.1 of \cite{BerkPhilip}),  without loss of
generality,  we can reconstruct the sequences $\left( X_{\left( k, j\right)
}^{\prime }\right) _{\left( k, j\right) \in \mathcal{K}_{N}}, $ $\left( \sigma
_{k, j}W_{\left( k, j\right) }^{\prime }\right) _{\left( k, j\right) \in
\mathcal{K}_{N}}, $ $\left( W_{\left( k, j\right) }^{\prime \prime }\right)
_{\left( k, j\right) \in \mathcal{K}_{N}}$ and $\left( W_{\left( k, j\right)
}^{\prime \prime }\right) _{\left( k, j\right) \in \mathcal{K}_{N}}$ on the
same probability space in such way that a.s. $\left( \sigma _{k, j}W_{\left(
k, j\right) }^{\prime }\right) _{\left( k, j\right) \in \mathcal{K}%
_{N}}=\left( W_{\left( k, j\right) }^{\prime \prime }\right) _{\left(
k, j\right) \in \mathcal{K}_{N}}.$ We shall consider that this probability
space is an extension of the initial probability space. Thus we have
constructed the $W_{i}$'s when $i$ belongs to the union of all islands $%
I=\cup _{\left( k, j\right) \in \mathcal{K}_{N}}I_{k, j}$ with the property
that the $\eta _{k, j}=\sum_{i\in I_{k, j}}W_{i}-\sigma _{k, j}W_{\left(
k, j\right) }^{\prime }$ are independent normal and centered random variables
with variances $v_{k, j}^{2}=\left\vert \sigma _{k, j}^{2}-\left\vert
I_{k, j}\right\vert \right\vert \leq \tau \left\vert I_{k, j}\right\vert
^{\gamma }, $ for any $\gamma >0.$ Therefore the sum $\sum_{\left( l, i\right)
\leq \left( n, m\right) }\eta _{l, i}$ is normal with mean $0$ and variance $%
\sum_{\left( l, i\right) \leq \left( n, m\right) }v_{k, i}^{2}\leq c_{\beta
}\tau 2^{\left( \beta +\gamma \right) n};$ by Doob's inequality,  with $%
p=2+2\alpha, $ it follows%
\begin{eqnarray}
\mathbb{P}\left( \left( 2^{n}\right) ^{-\frac{1}{2}}\sup_{\left( k, j\right)
\in \mathcal{K}_{N}}\left\vert \sum_{\left( l, i\right) \leq \left(
k, j\right) }\eta _{l, i}\right\vert \geq \left( 2^{n}\right) ^{-\rho }\right)
&\leq &\left( 2^{n}\right) ^{-p/2+\rho p}\mathbb{E}\left( \left\vert
\sum_{\left( l, i\right) \leq \left( n, m\right) }\eta _{l, i}\right\vert
^{p}\right)  \notag \\
&\leq &c_{\alpha, \beta }\left( 2^{n}\right) ^{-p/2+\rho p}\left( \tau
2^{\left( \beta +\gamma \right) n}\right) ^{p/2}  \notag \\
&=&c_{\alpha, \beta }\tau ^{1+\alpha }\left( 2^{n}\right) ^{-\left( 1-\beta
\right) \left( 1+\alpha \right) +\left( \rho +\gamma /2\right) \left(
2+2\alpha \right) } \notag
\\
&\ &  \label{constr005}
\end{eqnarray}%
where $\gamma >0$ is arbitrary. When $i$ belongs to the union of gaps $%
J=\cup _{\left( k, j\right) \in \mathcal{K}_{N}}J_{k, j}$ the variables $W_{i}$
can be taken as any independent standard normal random variables independent
of the sequence $\left( W_{k}\right) _{k\in I}.$

So far we have constructed the variables $(X_{\left( k, j\right) }^{\prime})_{\left( k, j\right) \in \mathcal{K}_{N}}$ corresponding to sums over the islands. Now we proceed to construct the components of the sequence $( \widetilde{X}_{i} ) _{1\leq i\leq N}.$
First, we proceed with the components belonging to all islands. For each $\left( k, j\right) \in \mathcal{K}_{N}, $ we construct a sequence $( \widetilde{X}_{i} ) _{i\in I_{k, j}}$ such that $\sum_{i\in I_{k, j}}\widetilde{X}_{i}=X_{\left( k, j\right) }^{\prime }$ and $( \widetilde{X}_{i} )_{i=1, ..., N}\overset{d}{=}\left( X_{i}\right) _{i=1, ..., N}.$
Denote by $F_{X|Y_{1}, ..., Y_{k}}\left( x|y_{1}, ..., y_{k}\right) $ the conditional
distribution of $X$ given $[Y_{1}=y_{1}, ..., Y_{k}=y_{k}].$ Without loss of
generality,  on the initial probability space,  there exists a sequence $%
U_{1}, ..., U_{N}$ of independent r.v.'s uniformly distributed on $\left(
0, 1\right)$. Let $i_{1}, ..., i_{\left\vert I_{k, j}\right\vert }$ be the
indices in the set $I_{k, j}.$ The required construction is performed in the
standard way by defining first $\widetilde{X}_{i_{1}}$ as the conditional
quantile transform
$$
F_{X_{i_{1}}|X_{\left( k, j\right) }}^{-1}\left(U_{k, i_{1}}|X_{\left( k, j\right) }^{\prime }\right)
$$
and then by setting, for $l=2, ..., \left\vert I_{k, j}\right\vert,$
$$
\widetilde{X}_{i_{l}}=F_{X_{i_{l}}|X_{i_{1}}, ..., X_{i_{l-1}}, X_{\left(
k, j\right) }}^{-1}\left( U_{i_{l}}|\widetilde{X}_{i_{1}}, ..., \widetilde{X}%
_{i_{l-1}}, X_{\left( k, j\right) }^{\prime }\right).
$$
Thus we have constructed the vector $\widetilde{X}%
_{I}=( \widetilde{X}_{i} ) _{i\in I}, $ where $I=\cup _{\left(
k, j\right) \in \mathcal{K}_{N}}I_{k, j}$ is the union of all islands, such
that $\widetilde{X}_{I}\overset{d}{=}X_{I}=\left( X_{i}\right) _{i\in I}.$
In the same way we construct the $\widetilde{X}_{i}$ when $i$ belongs to the
union of gaps $J=\cup _{\left( k, j\right) \in \mathcal{K}_{N}}J_{k, j}:=%
\{j_{1}, ..., j_{\left\vert J\right\vert} \}$ : set $%
X_{j_{1}}=F_{X_{j_{1}}|X_{I}}^{-1}\left( U_{k, j_{1}}|\widetilde{X}%
_{I}\right) $ and subsequently $\widetilde{X}%
_{j_{l}}=F_{X_{j_{l}}|X_{j_{1}}, ..., X_{j_{l-1}}, X_{I}}^{-1}\left( U_{j_{l}}|%
\widetilde{X}_{j_{1}}, ..., \widetilde{X}_{j_{l-1}}, \widetilde{X}_{I}\right), $
for $l=2, ..., \left\vert J\right\vert.$

\subsection{Putting together the bounds\label{sec Putting Together}}

Denote by $r_{k, j}$ the right end of the island $I_{k, j}$ and let $\mathcal{L%
}_{N}=\left\{ 1\right\} \cup \left\{ r_{k, j}:\left( k, j\right) \in \mathcal{K%
}_{N}\right\} $ be the set of $r_{k, j},  $ ordered with respect to
lexicographical order $\preceq $. For any $r=r_{k, j}\in \mathcal{L}_{N}$
denote by $r^{next}$ be the next element in the set $\mathcal{L}_{N}, $ i.e. $%
r^{next}=\inf \left\{ r^{\prime }:r^{\prime }\in \mathcal{L}_{N}, \ r\preceq
r^{\prime }\right\}.$

Let $(\widetilde{X}_{i})_{1\leq i\leq N}$ and $\left( W_{i}\right) _{1\leq
i\leq N}$ be the sequences constructed in Section \ref{sec Construction}.
Recall that by construction,  for any $r=r_{k, j}\in \mathcal{L}_{N}, $ we have
$\left\{ 1, ..., r\right\} =\sum_{\left( l, i\right) \preceq \left( k, j\right)
}J_{l, j}\cup I_{k, j}.$ First we replace in the statement of Theorem \ref{Th
main res 1} the $\sup $ over the set $\left\{ j:1\leq j\leq N\right\} $ by
the $\sup $ over the grid $\mathcal{L}_{N}$ and the $\sup $ of the
oscillation term: in other words,  the random variable  $\displaystyle \sup_{1\leq j\leq
N}\left\vert \sum_{i\leq j}\left( \widetilde{X}_{i}-W_{i}\right) \right\vert
$ is bounded by
\begin{equation}
\underbrace{\sup_{r\in \mathcal{L}_{N}}\left\vert \sum_{i\leq r}\left(
\widetilde{X}_{i}-W_{i}\right) \right\vert }_{\text{(sup over the grid }%
\mathcal{L}_{N}\text{)}}\quad +\quad \underbrace{\sup_{r\in \mathcal{L}%
_{N}}\sup_{r\leq r^{\prime }\leq \min \left\{ r^{next}-1, N\right\}
}\left\vert \sum_{r\leq i\leq r^{\prime }}\left( \widetilde{X}%
_{i}-W_{i}\right) \right\vert }_{\text{(oscillation term)}}.  \label{put000}
\end{equation}%
%
%
%
%
%
%
%
%
%
%
%
%
%
%
%
%
For any $r=r_{k, j}, $ we have%
\begin{equation}
\sum_{1\leq i\leq r}\widetilde{X}_{i}=\sum_{\left( l, h\right) \preceq \left(
k, j\right) }\left( \sum_{i\in I_{l, h}}\widetilde{X}_{i}+\sum_{i\in J_{l, h}}%
\widetilde{X}_{i}\right) =\sum_{\left( l, h\right) \preceq \left( k, j\right)
}\left( \underline{X}_{\left( l, h\right) }+\overline{X}_{\left( l, h\right)
}\right),   \label{put00a}
\end{equation}%
where $\underline{X}_{\left( l, h\right) }=\sum_{i\in I_{l, h}}\widetilde{X}%
_{i}$ and $\overline{X}_{\left( l, h\right) }=\sum_{i\in J_{l, h}}\widetilde{X}%
_{i}.$ In the same way%
\begin{equation}
\sum_{1\leq i\leq r}W_{i}=\sum_{\left( l, h\right) \preceq \left( k, j\right)
}\left( \sum_{i\in I_{l, h}}W_{i}+\sum_{i\in J_{l, h}}W_{i}\right)
=\sum_{\left( l, h\right) \preceq \left( k, j\right) }\left( \underline{W}%
_{\left( l, h\right) }+\overline{W}_{\left( l, h\right) }\right)
\label{put00b}
\end{equation}%
where $\underline{W}_{\left( l, h\right) }=\sum_{i\in I_{l, h}}W_{i}$ and $%
\overline{W}_{\left( l, h\right) }=\sum_{i\in J_{l, h}}W_{i}.$ From (\ref%
{put000}),  (\ref{put00a}) and (\ref{put00b}) we obtain%
\begin{eqnarray}
&&\sup_{1\leq j\leq N}\left\vert \sum_{i\leq j}\left( \widetilde{X}%
_{i}-W_{i}\right) \right\vert  \label{put001} \\
&\leq &\sup_{\left( k, j\right) \in \mathcal{K}_{N}}\left\vert \sum_{\left(
l, h\right) \leq \left( k, j\right) }\underline{X}_{\left( l, h\right) }-%
\underline{W}_{\left( l, h\right) }\right\vert \quad \;\text{(sup over
islands)}  \notag \\
&&+\sup_{\left( k, j\right) \in \mathcal{K}_{N}}\left\vert \sum_{\left(
l, h\right) \leq \left( k, j\right) }\left( \overline{X}_{\left( l, h\right) }-%
\overline{W}_{\left( l, h\right) }\right) \right\vert \quad \;\text{(sup over
gaps)}  \notag \\
&&+\sup_{r\in \mathcal{L}_{N}}\sup_{r\leq r^{\prime }\leq \min \left\{
r^{next}-1, N\right\} }\left\vert \sum_{l\leq i\leq r^{\prime }}\left(
\widetilde{X}_{i}-W_{i}\right) \right\vert \quad \;\text{(oscillation term)}
\notag
\end{eqnarray}%
where the term "sup over islands" is bounded by the sum%
\begin{eqnarray}
&&\sup_{\left( k, j\right) \in \mathcal{K}_{N}}\left\vert \sum_{\left(
l, h\right) \leq \left( k, j\right) }\underline{X}_{\left( l, h\right) }-\sigma
_{l, h}W_{\left( l, h\right) }^{\prime }\right\vert \quad \;\text{(normal
approximation)}  \label{put002} \\
&&+\sup_{\left( k, j\right) \in \mathcal{K}_{N}}\left\vert \sum_{\left(
l, h\right) \leq \left( k, j\right) }\left( \sigma _{l, h}W_{\left( l, h\right)
}^{\prime }-\underline{W}_{\left( l, h\right) }\right) \right\vert.\quad
\quad \text{(variance homogenization)}  \notag
\end{eqnarray}%
The control of the term "normal approximation" has been yet given in
Proposition \ref{Prop-Normal 001} where it is bounded by two terms
"dependence error" and "Sakhanenko's error". The "variance homogenization"
term is controlled by (\ref{constr005}). As to terms "sup over gaps" and
"oscillation term" they will be considered in (\ref{gaps006}) and (\ref%
{Oscil005}) below.

\subsection{Bound for the partial sums over gaps\label{bound for the partial
sums of gaps}}

Let $p=2+2\alpha, $ where $\alpha <\delta.$ Since the blocks are indexed by
$l=k_{0}, ..., n$ and the total length of the gaps in the block $l$ is less
than $\left( 2+[\beta l]\right) 2^{[\beta l]+\left[ \varepsilon l\right] -1}$
the total length $L^{gap}$ of all gaps satisfies
\begin{equation*}
L^{gap}=\sum_{\left( l, i\right) \preceq \left( n, m\right) } 
\left\vert J_{l, i}\right\vert 
\le  \sum_{k_{0}\leq l\leq n}\left( 2+[\beta l]\right) 2^{[\beta l]+
\left[ \varepsilon l\right] -1}\leq c_{\varepsilon, \beta }2^{\left( \beta
+\varepsilon \right) k}.
\end{equation*}%
By Proposition \ref{Prop max Lp bound},  we have,  for any $\eta >0, $%
\begin{eqnarray*}
\left\Vert \sup_{\left( k, j\right) \preceq \left( n, m\right) }\left\vert
\sum_{\left( l, i\right) \leq \left( k, j\right) }\overline{X}_{\left(l, i\right) }\right\vert \right\Vert _{L^{p}} 
&\leq &c_{\lambda _{1}, \lambda_{2}, \alpha, \delta, \eta }\left( 1+\lambda _{0}+\mu _{\delta }\right)
^{1+\eta }  
\left(      L^{gap}    \right)^{\frac{1}{2}} \\
&\leq &c_{\varepsilon, \beta, \lambda _{1}, \lambda _{2}, \alpha, \delta, \eta
}\left( 1+\lambda _{0}+\mu _{\delta }\right) ^{1+\eta }\left( 2^{\left(\beta +\varepsilon \right) n}\right) ^{\frac{1}{2}}.
\end{eqnarray*}%
Using Chebyshev's inequality
with $x=\left( 2^{n}\right) ^{\frac{1}{2}-\rho }$,  we get
\begin{eqnarray}
\mathbb{P}\left( \sup_{\left( k, j\right) \in \mathcal{K}_{N}}\left\vert
\sum_{\left( l, i\right) \leq \left( k, j\right) }\overline{X}_{\left(
l, i\right) }\right\vert \geq x\right) &\leq &\frac{1}{x^{p}}\mathbb{E}\left(
\sup_{\left( k, j\right) \preceq \left( n, m\right) }\left\vert \sum_{\left(
l, i\right) \leq \left( k, j\right) }\overline{X}_{\left( l, i\right)
}\right\vert \right) ^{p}  \notag \\
&\leq &A\left( 2^{n}\right) ^{-p/2+\rho p}\left( 2^{\left( \beta
+\varepsilon \right) n}\right) ^{p/2}  \notag \\
&\leq &A\left( 2^{n}\right) ^{-\left( 1-\beta -\varepsilon \right) p/2+\rho
p},   \label{gaps005a}
\end{eqnarray}%
where $A=c_{\varepsilon, \beta, \lambda _{1}, \lambda _{2}, \alpha, \delta
, \eta }\left( 1+\lambda _{0}+\mu _{\delta }\right) ^{p\left( 1+\eta \right)
}.$ 
A similar bound can be established with $\overline{W}_{\left( l, i\right) }$
instead of $\overline{X}_{\left( l, i\right) }.$ Using this bound and (\ref%
{gaps005a}) it follows that,  for any $0<\rho <\frac{1}{4}, $%
\begin{eqnarray}
&&\mathbb{P}\left( \left( 2^{n}\right) ^{-\frac{1}{2}}\sup_{\left(
k, j\right) \in \mathcal{K}_{N}}\left\vert \sum_{\left( l, i\right) \leq
\left( k, j\right) }\left( \overline{X}_{\left( l, i\right) }-\overline{W}%
_{\left( l, i\right) }\right) \right\vert \geq \left( 2^{n}\right) ^{-\rho
}\right)  \notag \\
&\leq &A^{\prime }\left( 2^{n}\right) ^{-\left( 1-\beta -\varepsilon \right)
p/2+\rho p}  \notag \\
&\leq &A^{\prime }\left( 2^{n}\right) ^{-\left( 1-\beta \right) \left(
1+\alpha \right) +\left( \rho +\varepsilon /2\right) \left( 2+2\alpha
\right) }  \label{gaps006}
\end{eqnarray}%
where $A^{\prime }=c_{\varepsilon, \beta, \lambda _{1}, \lambda _{2}, \alpha
, \delta }^{\prime }\left( 1+\lambda _{0}+\mu _{\delta }\right) ^{2+2\delta
}. $

\subsection{Bound for the oscillation term\label{Bound for the oscillation
term}}

Denote for brevity $r^{+}=\min \left\{ r^{next}-1, N\right\}.$ First note
that
\begin{equation*}
r^{+}-r\leq \max_{\left( k, j\right) \in \mathcal{K}_{N}}(\left\vert
I_{k, j}\right\vert +\left\vert J_{k, j}\right\vert )\leq c_{\varepsilon
, \beta }\left( 2^{\left( \beta +\varepsilon \right) n}+2^{\left( 1-\beta
\right) n}\right).
\end{equation*}%
Let $p=2+2\alpha $ where $\alpha <\delta.$ By Proposition \ref{Prop max Lp
bound} of Section \ref{sec Lp bounds},  for any $\eta >0, $ we have%
\begin{equation*}
\left\Vert \sup_{r\in \mathcal{L}_{N}}\sup_{r\leq l\leq r^{+}}\left\vert
\sum_{r\leq i\leq l}\widetilde{X}_{i}\right\vert \right\Vert
_{L^{p}}^{p}\leq A\left( \sup_{r\in \mathcal{L}_{N}}\left( r^{+}-r\right)
\right) ^{p/2}\leq c_{\varepsilon, \beta }A\left( 2^{\left( \beta
+\varepsilon \right) n}+2^{\left( 1-\beta \right) n}\right) ^{p/2},
\end{equation*}%
where $A=c_{\lambda _{1}, \lambda _{2}, \alpha, \delta, \eta }\left( 1+\lambda
_{0}+\mu _{\delta }\right) ^{p\left( 1+\eta \right) }.$ Therefore,  by
Chebyshev's inequality,  with $x=\frac{1}{2}\left( 2^{n}\right) ^{\frac{1}{2}%
-\rho }$ and $\rho >0, $%
\begin{eqnarray*}
\mathbb{P}\left( \sup_{r\in \mathcal{L}_{N}}\sup_{r\leq l\leq
r^{+}}\left\vert \sum_{r\leq i\leq l}\widetilde{X}_{i}\right\vert \geq
x\right) &\leq &x^{-p}\mathbb{E}\left( \sup_{r\in \mathcal{L}%
_{N}}\sup_{r\leq l\leq r^{+}}\left\vert \sum_{r\leq i\leq l}\widetilde{X}%
_{i}\right\vert \right) ^{p} \\
&\leq &c_{\varepsilon, \beta }Ax^{-p}\left( 2^{\left( \beta +\varepsilon
\right) n}+2^{\left( 1-\beta \right) n}\right) ^{p/2} \\
&\leq &c_{\varepsilon, \beta }^{\prime }A2^{p+np\rho }\left( \left(
2^{n}\right) ^{-\left( 1+\alpha \right) \left( 1-\beta -\varepsilon \right)
}+\left( 2^{n}\right) ^{-\left( 1+\alpha \right) \beta }\right).
\end{eqnarray*}%
Choosing $\eta $ small enough we have $p\left( 1+\eta \right) \leq 2+2\delta
$ and therefore
\begin{equation*}
c_{\varepsilon, \beta }2^{p}A\leq A^{\prime }=c_{\varepsilon, \beta, \lambda
_{1}, \lambda _{2}, \alpha, \delta }^{\prime }\left( 1+\lambda _{0}+\mu
_{\delta }\right) ^{2+2\delta }.
\end{equation*}%
Since a similar bound can be established with $W_{i}$ instead of $X_{i}, $ we
obtain the following bound for the oscillation term:%
\begin{eqnarray}
&&\mathbb{P}\left( \left( 2^{n}\right) ^{-\frac{1}{2}}\sup_{r\in \mathcal{L}%
_{N}}\sup_{r\leq l\leq r^{+}}\left\vert \sum_{r\leq i\leq l}\left(
\widetilde{X}_{i}-W_{i}\right) \right\vert \geq 2\left( 2^{n}\right) ^{-\rho
}\right)  \notag  \label{Oscil005a} \\
&\leq &2A^{\prime }\left( 2^{n}\right) ^{\left( 2+2\alpha \right) \left(
\rho +\varepsilon /2\right) }\left( \left( 2^{n}\right) ^{-\left( 1+\alpha
\right) \left( 1-\beta \right) }+\left( 2^{n}\right) ^{-\left( 1+\alpha
\right) \beta }\right).  \label{Oscil005}
\end{eqnarray}

\subsection{Optimizing the bounds \label{Optimizing the bounds}}

Let $\alpha <\delta, $ $\beta >\frac{1}{2}$ and $0<\rho <\frac{1-\beta }{2}.$
Using (\ref{put001}),  we may decompose the quantity $\displaystyle\mathbb{P}\left(
\left( 2^{n}\right) ^{-\frac{1}{2}}\sup_{1\leq j\leq N}\left\vert
\sum_{i=1}^{j}\left( \widetilde{X}_{i}-W_{i}\right) \right\vert \geq 6\left(
2^{n}\right) ^{-\rho }\right) $ in 3 terms,  the first one \textquotedblleft
sup over islands\textquotedblright\ being itself decomposed in two terms
(see (\ref{put002})); consequently,  this quantity is decomposed in four
terms listed below:

\begin{itemize}
\item the first term \textquotedblleft normal
approximation\textquotedblright\ is controlled with Proposition \ref%
{Prop-Normal 001},  it is bounded by two terms named \textquotedblleft
dependence error\textquotedblright\ and \textquotedblright Sakhanenko's
error\textquotedblright,

\item the second term ``variance homogenization'' is controlled in \ref%
{constr005} with $\gamma =\varepsilon$

\item the term ``sup over gaps'' is controlled in (\ref{gaps006}),

\item the term ``oscillation term'' is controlled in (\ref{Oscil005}).
\end{itemize}
Putting these bounds altogether,  we obtain%
\begin{eqnarray*}
&&\mathbb{P}\left( \left( 2^{n}\right) ^{-\frac{1}{2}}\sup_{1\leq j\leq
N}\left\vert \sum_{i=1}^{j}\left( X_{i}-W_{i}\right) \right\vert \geq
6\left( 2^{n}\right) ^{-\rho }\right) \\
&\leq &A\left( 2^{n}\right) ^{-\left( 1+\alpha \right) +\left( \rho
+\varepsilon \right) \left( 2+2\alpha \right) }\quad \text{(dependence error)%
} \\
&&+A\left( 2^{n}\right) ^{-\beta \alpha +\rho \left( 2+2\alpha \right)
}\quad \text{(Sakhanenko's error)} \\
&&+A\left( 2^{n}\right) ^{-\left( 1-\beta \right) \left( 1+\alpha \right)
+\left( \rho +\varepsilon /2\right) \left( 2+2\alpha \right) }\quad \text{%
(variance homogenization error)} \\
&&+A\left( 2^{n}\right) ^{-\left( 1-\beta \right) \left( 1+\alpha \right)
+\left( \rho +\varepsilon /2\right) \left( 2+2\alpha \right) }\quad \text{%
(gaps error)} \\
&&+A\left( \left( 2^{n}\right) ^{-\left( 1-\beta \right) \left( 1+\alpha
\right) }+\left( 2^{n}\right) ^{-\beta \left( 1+\alpha \right) }\right)
\left( 2^{n}\right) ^{\left( \rho +\varepsilon /2\right) \left( 2+2\alpha
\right) }\quad \text{(oscillation error)}
\end{eqnarray*}%
where $A=c_{\varepsilon, \beta, \lambda _{1}, \lambda _{2}, \alpha }\left(
1+\tau ^{1+\alpha }+\left( 1+\lambda _{0}+\mu _{\delta }\right) ^{2+2\alpha
}\right). $

For the
moment let us ignore the factors containing $\varepsilon $ which have a
small contribution to the bound.
The term (dependence error) is negligible with respect to all other terms;
equating the powers of the term (Sakhanenko's error) and the term (gaps
error) (or equivalently variance homogenization error) we get $\beta \alpha
=\left( 1-\beta \right) \left( 1+\alpha \right) $ i.e. $\beta =\frac{%
1+\alpha }{1+2\alpha }.$ Implementing $\beta =\frac{1+\alpha }{1+2\alpha }$
in the above inequality yields%
\begin{eqnarray*}
&&\mathbb{P}\left( \left( 2^{n}\right) ^{-\frac{1}{2}}\sup_{1\leq l\leq
N}\left\vert \sum_{i=1}^{l}\left( \widetilde{X}_{i}-W_{i}\right) \right\vert
\geq 6\left( 2^{n}\right) ^{-\rho }\right) \\
&\leq &A\left( 2^{n}\right) ^{-1-\delta +\left( \rho +\varepsilon \right)
\left( 2+2\alpha \right) }\quad \text{(dependence error)} \\
&&+6A\left( 2^{n}\right) ^{-\frac{\alpha \left( 1+\alpha \right) }{1+2\alpha
}+\left( \rho +\varepsilon \right) \left( 2+2\alpha \right) }%
\begin{array}{c}
\text{(Sakhanenko's error} \\
\text{+ variance homogenization error} \\
\text{+ gaps error + oscillation error).}%
\end{array}%
\end{eqnarray*}%
Taking into account that $\alpha <\delta $ and $2^{n}\leq N<2^{n+1}$ we
obtain
\begin{equation}
\mathbb{P}\left( N^{-\frac{1}{2}}\sup_{1\leq l\leq N}\left\vert
\sum_{i=1}^{l}\left( \widetilde{X}_{i}-W_{i}\right) \right\vert \geq
6N^{-\rho }\right) \leq A^{\prime }N^{-\frac{\alpha \left( 1+\alpha \right)
}{1+2\alpha }+\left( \rho +\varepsilon \right) \left( 2+2\alpha \right) }
\label{optimiz003}
\end{equation}%
where $A^{\prime }=c_{\varepsilon, \lambda _{1}, \lambda _{2}, \alpha, \delta
}\left( 1+\lambda _{0}+\mu _{\delta }+\sqrt{\tau }\right) ^{2+2\delta }$ and
$\rho $ satisfies $0<\rho <\frac{1-\beta }{2}=\frac{\alpha }{2\left(
1+2\alpha \right) }.$

Note that the function $g\left( \alpha \right) =\frac{\alpha \left( 1+\alpha
\right) }{1+2\alpha }-\rho \left( 2+2\alpha \right)$ is strictly increasing
on $\mathbb{R}^+$ when $\rho <\frac{1}{4}.$ Therefore we can get rid of the
constant $\varepsilon $ in the bound by choosing $\alpha ^{\prime }<\alpha.$
If we let $\Delta =g\left( \alpha \right) -g\left( \alpha ^{\prime }\right)
>0$ and choose $\varepsilon $ sufficiently small,  we obtain $\left(
2^{n}\right) ^{-g\left( \alpha \right) +\varepsilon \left( 2+2\alpha \right)
}=\left( 2^{n}\right) ^{-g\left( \alpha ^{\prime }\right) -\Delta
+\varepsilon \left( 2+2\alpha \right) }\leq \left( 2^{n}\right) ^{-g\left(
\alpha ^{\prime }\right) }.$ Since $\alpha $ and $\alpha ^{\prime }$ are
arbitrary satisfying $\alpha ^{\prime }<\alpha <\delta, $ the assertion of
Theorem \ref{Th main res 1} follows with $\alpha ^{\prime }$ replacing $%
\alpha.$

We have performed a construction of the sequences $\widetilde{X}^{\left(
N\right) }=(\widetilde{X})_{1\leq i\leq N}$ and $W^{\left( N\right)
}=(W)_{1\leq i\leq N}$ for each fixed $N\geq 1, $ where for each $N$ the
constructed sequences in general are different. Below we show how to obtain
a construction of the entire sequences $(\widetilde{X})_{i\geq 1}$ and $%
(W)_{i\geq 1}.$

Let $\Omega ^{\left( N\right) }=\mathbb{R}^{N+1}\times \mathbb{R}^{N+1}.$
Without loss of generality,  for any $\omega =\left( \omega _{1}, \omega
_{2}\right) \in \Omega ^{\left( N\right) }, $ the sequences $\widetilde{X}%
^{\left( N\right) }$ and $W^{\left( N\right) }$ can be reconstructed on $%
\Omega ^{\left( N\right) }$ so that $\widetilde{X}_{i}=\omega _{1, i}, $ $%
W_{i}=\omega _{2, i}$ and their joint distribution,  say $\mathbb{P}^{\left(
N\right) }, $ is preserved. Each measure $\mathbb{P}^{\left( N\right) }$ can
be extended (arbitrarily) on the space $\mathbb{R}^{\infty }\times \mathbb{R}%
^{\infty }.$ From the bound (\ref{optimiz003}) it follows that the sequence
of measures $\mathbb{P}^{\left( N\right) }$ is tight. Therefore there is a
weak limit which satisfies (\ref{optimiz003}) and thus provides the desired
construction.

\section{ Proof of the results of Section \ref{sec Applic} \label{sec Proofs res Sec3}}
Throughout this section we assume that the Markov chain $(X_{n})_{n \geq 0}$ and the function $f$ satisfy the hypotheses \textbf{M1,  M2, }
\textbf{\ M3} and \textbf{M4}.

\subsection{Proof of Proposition \protect\ref{Proposition MC001} }
First,  we establish the following bound for the characteristic functions $%
\phi _{x, 1}, $ $\phi _{x, 2}$ and $\phi _{x}$ involved in Proposition \ref%
{Proposition MC001}.

\begin{lemma}
\label{Lemma Cond Dependence for MC} For any $k_{gap}, M_{1}, M_{2}\in \mathbb{N%
}, $ any sequence $j_{0}<...<j_{M_{1}+M_{2}}$ and any $t=(t_i)_i\in \mathbb{R}%
^{M_{1}}, s=(s_i)_i\in \mathbb{R}^{M_{2}}$ satisfying $\left\Vert \left(
t, s\right) \right\Vert _{\infty }\leq \varepsilon _{0}, $
\begin{equation*}
\left\vert \phi _{x}\left( t, s\right) -\phi _{x, 1}\left( t\right) \phi
_{x, 2}\left( s\right) \right\vert \leq 2C_{Q}C_{\mathbf{P}}^{M_{1}+M_{2}}%
\Bigl( \left\Vert \nu \right\Vert _{\mathcal{B}^{\prime }}+\left\Vert
\delta_{x}\right\Vert _{\mathcal{B}^{\prime }}\Bigr) \left\Vert e%
\right\Vert _{\mathcal{B}}\kappa ^{k_{gap}}.
\end{equation*}
\end{lemma}

\begin{proof}
Set for brevity $\phi _{1}=\phi _{x, 1}, $ $\phi _{2}=\phi _{x, 2}$ and $\phi
=\phi _{x}.$ The characteristic function $\phi $ can be rewritten in the
following form:
\begin{equation*}
\phi \left( t, s\right) =\left( \mathbf{P}^{j_{0}}\mathbf{P}%
_{t_{1}}^{\left\vert J_{1}\right\vert }...\mathbf{P}_{t_{M_{1}}}^{\left\vert
J_{M_{1}}\right\vert }\mathbf{P}^{k_{gap}}\mathbf{P}_{s_{M_{1}+1}}^{\left%
\vert J_{M_{1}+1}\right\vert }...\mathbf{P}_{s_{M_{1}+M_{2}}}^{\left\vert
J_{M_{1}+M_{2}}\right\vert }e\right) \left( x\right).
\end{equation*}
Since $\mathbf{P}=\Pi +Q$ we get $\mathbf{P}^{k}=\Pi +Q^{k}, $ and thus%
\begin{equation}  \label{decompose}
\phi \left( t, s\right)= \phi_{\Pi}\left( t, s\right)+\phi_Q \left( t, s\right)
\end{equation}
with
\begin{equation}  \label{termePi}
\phi_{\Pi}\left( t, s\right):=\left(\mathbf{P}^{j_{0}} \mathbf{P}%
_{t_{1}}^{\left\vert J_{1}\right\vert }...\mathbf{P}_{t_{M_{1}}}^{\left\vert
J_{M_{1}}\right\vert }\Pi ^{k_{gap}}\mathbf{P}_{s_{M_{1}+1}}^{\left\vert
J_{M_{1}+1}\right\vert }...\mathbf{P}_{s_{M_{1}+M_{2}}}^{\left\vert
J_{M_{1}+M_{2}}\right\vert }e\right) \left( x\right)
\end{equation}
and
\begin{equation}  \label{termeQ}
\phi_Q \left( t, s\right):=\left(\mathbf{P}^{j_{0}} \mathbf{P}%
_{t_{1}}^{\left\vert J_{1}\right\vert }...\mathbf{P}_{t_{M_{1}}}^{\left\vert
J_{M_{1}}\right\vert }Q^{k_{gap}}\mathbf{P}_{s_{M_{1}+1}}^{\left\vert
J_{M_{1}+1}\right\vert }...\mathbf{P}_{s_{M_{1}+M_{2}}}^{\left\vert
J_{M_{1}+M_{2}}\right\vert }e\right)\left( x\right).
\end{equation}
First,  since $\Pi ^{k_{gap}} \mathbf{P}_{s_{M_{1}+1}}^{\left\vert
J_{M_{1}+1}\right\vert }...\mathbf{P}_{s_{M_{1}+M_{2}}}^{\left\vert
J_{M_{1}+M_{2}}\right\vert }e = \nu \left( \mathbf{P}%
_{s_{M_{1}+1}}^{\left\vert J_{M_{1}+1}\right\vert }...\mathbf{P}%
_{s_{M_{1}+M_{2}}}^{\left\vert J_{M_{1}+M_{2}}\right\vert }e\right)e$,  we may write,  setting $\psi_2(s) := \nu \left( \mathbf{P}%
_{s_{M_{1}+1}}^{\left\vert J_{M_{1}+1}\right\vert }...\mathbf{P}%
_{s_{M_{1}+M_{2}}}^{\left\vert J_{M_{1}+M_{2}}\right\vert }e\right)$
\begin{eqnarray*}
\phi_{\Pi}\left( t, s\right) &=& \psi_2(s) \left( \mathbf{P}^{j_{0}} \mathbf{P%
}_{t_{1}}^{\left\vert J_{1}\right\vert }...\mathbf{P}_{t_{M_{1}}}^{\left%
\vert J_{M_{1}}\right\vert }e \right)(x) \\
&=& \psi_2(s) \phi _{1}\left( t\right).
\end{eqnarray*}%
Notice that $\phi _{2}\left(s\right)= \left( \mathbf{P}^{k_{gap}+j_{M_1} }%
\mathbf{P}_{s_{M_{1}+1}}^{\left\vert J_{M_{1}+1}\right\vert }...\mathbf{P}%
_{s_{M_{1}+M_{2}}}^{\left\vert J_{M_{1}+M_{2}}\right\vert }e%
\right)(x)$; using the equality $\nu \mathbf{P}=\nu$,  one gets $%
\psi_2(s)=\nu\left( \mathbf{P}^{k_{gap}+j_{M_1} }\mathbf{P}%
_{s_{M_{1}+1}}^{\left\vert J_{M_{1}+1}\right\vert }...\mathbf{P}%
_{s_{M_{1}+M_{2}}}^{\left\vert J_{M_{1}+M_{2}}\right\vert }e\right)$
which allows us to control the difference between $\psi_2$ and $\phi_2$,
namely 
\begin{eqnarray*}
\psi_2(s)-\phi _{2}\left( s\right) &=&\left( \nu -\delta_{x}\right) \left(
\mathbf{P}^{k_{gap}+j_{M_1}}\mathbf{P}_{s_{M_{1}+1}}^{\left\vert
J_{M_{1}+1}\right\vert }...\mathbf{P}_{s_{M_{1}+M_{2}}}^{\left\vert
J_{M_{1}+M_{2}}\right\vert }e\right) \\
&=&\left( \nu -\delta_{x}\right) \left( \Pi \mathbf{P}_{s_{M_{1}+1}}^{\left%
\vert J_{M_{1}+1}\right\vert }...\mathbf{P}_{s_{M_{1}+M_{2}}}^{\left\vert
J_{M_{1}+M_{2}}\right\vert }e\right) \\
&&+\left( \nu -\delta_{x}\right) \left( Q^{k_{gap}}\mathbf{P}^{j_{M_1}}%
\mathbf{P}_{s_{M_{1}+1}}^{\left\vert J_{M_{1}+1}\right\vert }...\mathbf{P}%
_{s_{M_{1}+M_{2}}}^{\left\vert J_{M_{1}+M_{2}}\right\vert }e\right)
\\
&=&\left( \nu -\delta_{x}\right) \left( e\right)\nu \left( \mathbf{P%
}_{s_{M_{1}+1}}^{\left\vert J_{M_{1}+1}\right\vert }...\mathbf{P}%
_{s_{M_{1}+M_{2}}}^{\left\vert J_{M_{1}+M_{2}}\right\vert }e\right)
\\
&&+\left( \nu -\delta_{x}\right) \left( Q^{k_{gap} }\mathbf{P}^{j_{M_1}}%
\mathbf{P}_{s_{M_{1}+1}}^{\left\vert J_{M_{1}+1}\right\vert }...\mathbf{P}%
_{s_{M_{1}+M_{2}}}^{\left\vert J_{M_{1}+M_{2}}\right\vert }e\right)
\end{eqnarray*}
with $\left( \nu -\delta_{x}\right) \left( e\right)=0$;
consequently
\begin{eqnarray*}  \label{psi2-phi2}
\Big\vert \psi_2(s)-\phi _{2}\left( s\right)\Big\vert &=&\Big\vert \left(
\nu -\delta_{x}\right) \left( Q^{k_{gap}+j_{M_1}}\mathbf{P}%
_{s_{M_{1}+1}}^{\left\vert J_{M_{1}+1}\right\vert }...\mathbf{P}%
_{s_{M_{1}+M_{2}}}^{\left\vert J_{M_{1}+M_{2}}\right\vert }e\right)%
\Big\vert \\
&\leq& C_Q C_{\mathbf{P}}^{1+M_2} \kappa^{k_{gap}} (\Vert \nu\Vert_{\mathcal{%
B}^{\prime }}+\Vert \delta_x\Vert_{\mathcal{B}^{\prime }})\Vert e%
\Vert_{\mathcal{B}}.
\end{eqnarray*}
On the other hand,  one easily gets%
\begin{equation}
\Big\vert \phi_Q \left( t, s\right)\Big\vert \leq C_{Q} C_{\mathbf{P}%
}^{1+M_{1}+M_{2}}\kappa ^{k_{gap}}\left\Vert e\right\Vert _{%
\mathcal{B}}\left\Vert \delta_{x}\right\Vert _{\mathcal{B}^{\prime }}.
\end{equation}%
Writing $\phi(t,
s)=\phi_1(t)\phi_2(s)+\phi_1(t)(\psi_2(s)-\phi_2(s))+\phi_Q(t, s)$ and using
the previous inequalities,  one finally gets
\begin{equation*}
\Big\vert \phi \left( t, s\right)-\phi_1(t)\phi_2(s)\Big\vert \leq 2 C_Q C_{%
\mathbf{P}}^{1+M_1+M_2} \Bigl(\Vert \nu\Vert_{\mathcal{B}^{\prime }}+\Vert
\delta_x\Vert_{\mathcal{B}^{\prime }}\Bigr)\kappa ^{k_{gap}}\left\Vert
e\right\Vert _{\mathcal{B}}.
\end{equation*}
Lemma is proved.
\end{proof}

To prove Proposition \ref{Proposition MC001},  set $k_{0}=\max \left\{ 1, \log
_{2}C_{\mathbf{P}}\right\} $ so that $C_{\mathbf{P}}\leq 2^{k_{0}}.$ \newline
Since $\displaystyle \max_{m=1, ..., M_{1}+M_{2}}card\left( J_{m}\right) \geq
1, $ one gets
\begin{equation*}
C_{\mathbf{P}}^{M_{1}+M_{2}}\leq 2^{k_{0}\left( M_{1}+M_{2}\right) }\leq
\left( 1+\max_{m=1, ..., M_{1}+M_{2}}card\left( J_{m}\right) \right)
^{k_{0}\left( M_{1}+M_{2}\right) }.
\end{equation*}%
Now,  Proposition \ref{Proposition MC001} follows from Lemma \ref{Lemma Cond
Dependence for MC}.

\subsection{Proof of Proposition \protect\ref{Proposition MC002} }

We need two auxiliary lemmas.

\begin{lemma}
\label{Lemma Cov MC} For any $l, k=0, 1, ...$
\begin{equation}
\left\vert \text{\textrm{Cov}}_{\mathbb{P}_{x}}\left( f\left( X_{l}\right),
f\left( X_{l+k}\right) \right) \right\vert \leq A\left( x\right) \kappa
^{k\gamma /4},   \label{covT000}
\end{equation}%
for any positive constant $\gamma $ satisfying $0<\gamma \leq \min \left\{
1, 2\delta \right\}, $ where%
\begin{equation*}
A\left( x\right) =c_{\delta }\Bigl( 1+C_{Q}C_{\mathbf{P}}^{2}\left(
\left\Vert \nu \right\Vert _{\mathcal{B}^{\prime }}+\left\Vert
\delta_{x}\right\Vert _{\mathcal{B}^{\prime }}\right) \left\Vert e%
\right\Vert _{\mathcal{B}}+\mu _{\delta }^{2+\gamma }\left( x\right) \Bigr).
\end{equation*}%
\end{lemma}

\begin{proof}
We give a proof involving Lemma \ref{Lemma Cond Dependence for MC}. Let $V$
and $V^{\prime }$ be two independent identically distributed r.v.'s of mean $%
0, $ independent of $X_{l}$ and $X_{l+m}$ and whose common characteristic
function is supported in the interval $[-\varepsilon _{0}, \varepsilon _{0}], $
for some $\varepsilon _{0}>0.$ Set $Y_{l}=f\left( X_{l}\right) +V$ and $%
Y_{l+k}^{\prime }=f\left( X_{l+k}\right) +V^{\prime }.$

Let $\widetilde{\phi }_{1}$ (resp. $\widetilde{\phi }_{2},$ $\widetilde{\phi }\left( t, u\right) $)
be the characteristic function of $Y_{l} $  (resp. $Y_{l+k}^{\prime },$ $\left( Y_{l}, Y_{l+k}^{\prime }\right)$  ).
Set $g_{T}\left(
x\right) =x1_{\left( \left\vert x\right\vert \leq T\right) }$ and $%
h_{T}\left( x, y\right) =g_{T}\left( x\right) g_{T}\left( y\right) $ for $%
x, y\in \mathbb{R}.$ Let $\widehat{g}_{T}$ (resp. $\widehat{h}_{T}$) be the
Fourier transform of the function $g_{T}$ (resp. $h_{T})$ defined by
\begin{equation*}
\widehat{g}_{T}\left( t\right) =\int e^{itx}g_{T}\left( x\right) dx,
\end{equation*}%
(resp. $\displaystyle\widehat{h}_{T}\left( t, u\right) =\int \int e^{i\left(
tx+uy\right) }h_{T}\left( x, y\right) dxdy=\widehat{g}_{T}\left( t\right)
\widehat{g}_{T}\left( u\right)$).

For any $T>0$ and $l\geq 1, \ k\geq 0$,  one gets%
\begin{equation}
\mathbb{E}_{x}f\left( X_{l}\right) f\left( X_{l+k}\right) =\mathbb{E}%
_{x}Y_{l}Y_{l+k}^{\prime }=\mathbb{E}_{x}h_{T}(Y_{l}, Y_{l+k}^{\prime })+R_{0}
\label{cov000}
\end{equation}%
with
\begin{equation}
\left| R_{0} \right| \leq \mathbb{E}_{x} \left|Y_{l}Y_{l+k}^{\prime } \right|1_{\left( \left\vert
Y_{l}\right\vert >T\right) }+\mathbb{E}_{x} \left|Y_{l}Y_{l+k}^{\prime } \right| 1_{\left(\left\vert Y_{l+k}^{\prime }\right\vert >T\right) }.  \label{cov000a}
\end{equation}%
By the inverse Fourier transform,  one may write 
\begin{equation*}
\mathbb{E}_{x}f(X_{l})f\left( X_{l+k}\right) =\frac{1}{\left( 2\pi \right)
^{2}}\int \int \overline{\widehat{h}_{T}\left( t, u\right) }\widetilde{\phi }%
\left( t, u\right) dtdu+R_{0}.
\end{equation*}%
Analogously
\begin{equation}
\mathbb{E}_{x}f(X_{l})=\mathbb{E}_{x}Y_{l}=\mathbb{E}_{x}g_{T}(Y_{l})+R_{1}=%
\frac{1}{2\pi }\int \overline{\widehat{g}_{T}\left( t\right) }\widetilde{%
\phi }_{1}\left( t\right) dt+R_{1}  \label{cov-expect}
\end{equation}%
and%
\begin{equation*}
\mathbb{E}_{x}f\left( X_{l+k}\right) =\mathbb{E}_{x}Y_{l+k}^{\prime }=%
\mathbb{E}_{x}g_{T}\left( Y_{l+k}^{\prime }\right) +R_{2}=\frac{1}{2\pi }%
\int \overline{\widehat{g}_{T}(u)}\widetilde{\phi }_{2}(u)du+R_{2},
\end{equation*}%
where%
\begin{equation}
R_{1}:=\mathbb{E}_{x}Y_{l}1_{\left( \left\vert Y_{l}\right\vert >T\right)
}\;\;\text{and\ \ }R_{2}:=\mathbb{E}_{x}Y_{l+k}^{\prime }1_{\left(
\left\vert Y_{l+k}^{\prime }\right\vert >T\right) }.  \label{cov-R1 bound}
\end{equation}%
This gives%
\begin{eqnarray}
\text{\textrm{Cov}}_{\mathbb{P}_{x}}(f(X_{l}), f(X_{l+k})) &=&\mathbb{E}%
_{x}f(X_{l})f(X_{l+k})-\mathbb{E}_{x}f(X_{l})\mathbb{E}_{x}f(X_{l+k})  \notag
\\
&=&\frac{1}{\left( 2\pi \right) ^{2}}\int \int \overline{\widehat{h}%
_{T}\left( t, u\right) }\left( \widetilde{\phi }\left( t, u\right) -\widetilde{%
\phi }_{1}\left( t\right) \widetilde{\phi }_{2}\left( u\right) \right)
dtdu+R,   \notag
\\
& &  \label{cov00-AA}
\end{eqnarray}%
where
\begin{equation}
R=R_{0}+R_{1}\mathbb{E}_{x}g_{T}(Y_{l+k}^{\prime })+R_{2}\mathbb{E}%
_{x}g_{T}(Y_{l})+R_{1}R_{2}.  \label{cov00-BB}
\end{equation}%
Note that%
\begin{equation*}
\left\vert \int \int \overline{\widehat{h}_{T}\left( t, u\right) }\left(
\widetilde{\phi }\left( t, u\right) -\widetilde{\phi }_{1}\left( t\right)
\widetilde{\phi }_{2}\left( u\right) \right) dtdu\right\vert \leq \left\Vert
\widehat{h}_{T}\right\Vert _{L^{2}}\left\Vert \widetilde{\phi }-\widetilde{%
\phi }_{1}\widetilde{\phi }_{2}\right\Vert _{L^{2}}.
\end{equation*}%
Since $V, V^{\prime }$ are independent of $X_{l}, X_{l+k}, $ we have $%
\widetilde{\phi }\left( t, u\right) =\phi \left( t, u\right) \mathbb{E}%
_{x}e^{itV}\mathbb{E}_{x}e^{iuV^{\prime }}$ and $\widetilde{\phi }_{1}\left(
t\right) =\phi _{1}\left( t\right) \mathbb{E}_{x}e^{itV}, $ $\widetilde{\phi }%
_{2}\left( u\right) =\phi _{2}\left( u\right) \mathbb{E}_{x}e^{iuV^{\prime
}}, $ where
\begin{equation*}
\phi \left( t, u\right) :=\mathbb{E}_{x}e^{itf(X_{l})+iuf(X_{l+k})}\newline
=\left( \mathbf{P}^{l-1}\mathbf{P}_{t}\mathbf{P}^{k-1}\mathbf{P}_{u}e \right) \left( x\right)
\end{equation*}%
and $\displaystyle\phi _{1}\left( t\right) =\mathbb{E}_{x}e^{itf(X_{l})}=%
\left( \mathbf{P}^{l-1}\mathbf{P}_{t}e\right) \left( x\right), %
\displaystyle\phi _{2}(u)=\mathbb{E}_{x}e^{iuf(X_{l+k})}=\left( \mathbf{P}%
^{l-1}\mathbf{P}_{t}e\right) \left( x\right) $. Since the support
of the characteristic functions of $V$ and $V^{\prime }$ is the interval$%
[-\varepsilon _{0}, \varepsilon _{0}]$ the function $\widetilde{\phi }-%
\widetilde{\phi }_{1}\widetilde{\phi }_{2}$ vanishes outside the square $%
[-\varepsilon _{0}, \varepsilon _{0}]^{2}.$ Then,  by Lemma \ref{Lemma Cond
Dependence for MC},
\begin{eqnarray}
\Vert \widetilde{\phi }-\widetilde{\phi }_{1}\widetilde{\phi }_{2}\Vert
_{L^{2}} &\leq &2\varepsilon _{0}\sup_{\left\vert t\right\vert \leq
\varepsilon _{0}, \left\vert u\right\vert \leq \varepsilon _{0}}\left\vert
\phi \left( t, u\right) -\phi _{1}\left( t\right) \phi _{2}\left( u\right)
\right\vert  \notag \\
&\leq &4\epsilon _{0}C_{Q}C_{\mathbf{P}}^{3}\kappa ^{k}\Bigl(\left\Vert \nu
\right\Vert _{\mathcal{B}^{\prime }}+\left\Vert \delta _{x}\right\Vert _{%
\mathcal{B}^{\prime }}\Bigr)\left\Vert e\right\Vert _{\mathcal{B}}.
\label{cov005-a}
\end{eqnarray}%
Using the inequality
\begin{equation}
\left\Vert \widehat{h}_{T}\right\Vert _{L^{2}}^{2}=\int \int h_{T}^{2}\left(
x, y\right) dxdy=\left( \int g_{T}^{2}\left( x\right) dx\right) ^{2}\leq
\frac{4}{9}T^{6},   \label{cov005}
\end{equation}%
one obtains
\begin{equation}
\left\vert \text{\textrm{Cov}}_{\mathbb{P}_{x}}\left( f\left( X_{l}\right)
, f\left( X_{l+k}\right) \right) \right\vert \leq \frac{2}{3\pi ^{2}}%
T^{3}\epsilon _{0}C_{Q}C_{\mathbf{P}}^{3}\kappa ^{k}\Bigr(\left\Vert \nu
\right\Vert _{\mathcal{B}^{\prime }}+\left\Vert \delta _{x}\right\Vert _{%
\mathcal{B}^{\prime }}\Bigl)\left\Vert e\right\Vert _{\mathcal{B}%
}+\left\vert R\right\vert.  \label{Cov bound}
\end{equation}

Now we shall give a bound for $\left\vert R\right\vert.$ By H\"{o}lder's
inequality,  with $q_{\delta } = \frac{1+\delta}{\delta} > 1, $
\begin{equation*}
\mathbb{E}_{x}\left\vert Y_{l}\right\vert \left\vert Y_{l+k}^{\prime
}\right\vert 1_{\left( \left\vert Y_{l}\right\vert >T\right) }\leq \left(
\mathbb{E}_{x}\left\vert Y_{l}\right\vert ^{2+2\delta }\right) ^{\frac{1}{%
2+2\delta }}\left( \mathbb{E}_{x}\left\vert Y_{l+k}^{\prime} \right\vert ^{2+2\delta
}\right) ^{\frac{1}{2+2\delta }}\mathbb{P}_{x}\left( \left\vert
Y_{l}\right\vert >T\right) ^{\frac{1}{q_{\delta }}}.
\end{equation*}%
Using hypothesis \textbf{M4},  we have%
\begin{equation*}
\left( \mathbb{E}_{x}\left\vert Y_{l}\right\vert ^{2+2\delta }\right) ^{%
\frac{1}{2+2\delta }}\leq  \left( \mathbb{E}_{x}\left\vert
f(X_{l})\right\vert ^{2+2\delta }\right) ^{\frac{1}{2+2\delta }}+\left(
\mathbb{E}_{x}\left\vert V\right\vert ^{2+2\delta }\right) ^{^{\frac{1}{%
2+2\delta }}} \leq c_\delta A_{0}\left( x\right),
\end{equation*}%
with $A_{0}\left( x\right) = \mu _{\delta }\left( x\right) + 1.$
Similarly $\displaystyle\left( \mathbb{E}_{x}\left\vert
Y_{l+k}^{\prime }\right\vert ^{2+2\delta }\right) ^{\frac{1}{2+2\delta }}\leq c_\delta A_{0}\left( x\right).$
On the other hand,  for any $\gamma \in (0, 2\delta ], $ one gets
\begin{equation*}
\mathbb{P}_{x}\left( \left\vert Y_{l}\right\vert >T\right) \leq \frac{1}{%
T^{\gamma q_{\delta }}}\mathbb{E}_{x}|Y_{l}|^{\gamma q_{\delta }}\leq \frac{c_\delta}{T^{\gamma q_{\delta }}}A_{0}^{\gamma q_{\delta }}\left(
x\right).
\end{equation*}
Putting together these bounds gives
\begin{equation}
\mathbb{E}_{x}\left\vert Y_{l}\right\vert \left\vert Y_{l+k}^{\prime
}\right\vert 1_{\left( \left\vert Y_{l}\right\vert >T\right)} \leq c_\delta T^{-\gamma }A_{0}^{2+\gamma }\left( x\right).  \label{cov001a}
\end{equation}
In the same way we obtain, for any $\gamma \in
(0, 2\delta ], $%
\begin{equation}
\mathbb{E}_{x}\left\vert Y_{l}\right\vert \left\vert Y_{l+k}^{\prime
}\right\vert 1\left( \left\vert Y_{l+k}^{\prime }\right\vert >T\right) \leq
c_\delta T^{-\gamma }A_{0}^{2+\gamma }\left( x\right).  \label{cov002a}
\end{equation}
From (\ref{cov000a}),  (\ref{cov001a}),  (\ref{cov002a}),  it follows that%
\begin{equation}
\left\vert R_{0}\right\vert \leq c_\delta T^{-\gamma }A_{0}^{2+\gamma}\left( x\right).  \label{cov003aa}
\end{equation}
From (\ref{cov002a}),  taking $k=0$ we get,  for any $%
\gamma \in (0, 2\delta ], $%
\begin{equation}
\max \left\{ R_{1}, R_{2}\right\} \leq \sup_{l\geq 0}\left( \mathbb{E}%
_{x}Y_{l}^{2}1\left( \left\vert Y_{l}\right\vert >T\right) \right) ^{\frac{1%
}{2}}\leq c_\delta T^{-\gamma/2}A_{0}^{1+\gamma/2 }\left( x\right).
\label{cov003cc}
\end{equation}
Since%
\begin{equation*}
\left\vert \mathbb{E}_{x}g_{T}\left( Y_{l}\right) \right\vert \leq \left(
\mathbb{E}_{x}\left( \left\vert Y_{l}\right\vert ^{2+2\delta }\right)
\right) ^{\frac{1}{2+2\delta }}\leq c_\delta A_{0}\left( x\right)
\end{equation*}%
and
\begin{equation*}
\left\vert \mathbb{E}_{x}g_{T}\left( Y_{l+k}^{\prime }\right) \right\vert
\leq c_\delta A_{0}\left( x\right),
\end{equation*}%
from (\ref{cov003aa}),  (\ref{cov003cc}),  it follows that,
\begin{equation}
\left\vert R\right\vert \leq c_\delta T^{-\gamma /2}A_{0}^{2+\gamma}\left( x\right),  \label{R bound}
\end{equation}%
for any $\gamma \in (0, 2\delta ], $ where we assume without loss of
generality that $A_{0}(x)\geq 1.$
The inequalities (\ref{Cov bound}) and (\ref{R bound}) yield,  for any $\gamma \in (0, 2\delta ], $%
\begin{eqnarray*}
\left\vert \text{\textrm{Cov}}_{\mathbb{P}_{x}}\left( f\left( X_{l}\right)
, f\left( X_{l+k}\right) \right) \right\vert &\leq &\frac{2}{3\pi ^{2}}%
T^{3}\varepsilon _{0}C_{Q}C_{\mathbf{P}}^{3}\kappa ^{k}\left( \left\Vert \nu
\right\Vert _{\mathcal{B}^{\prime }}+\left\Vert \boldsymbol{\delta }%
_{x}\right\Vert _{\mathcal{B}^{\prime }}\right) \left\Vert e\right\Vert _{%
\mathcal{B}} \\
&&+c_\delta T^{-\gamma /2} A_{0}^{2+\gamma }\left( x\right).
\end{eqnarray*}
Choosing $T=\kappa ^{-k/4} $ and taking into account that $A_{0}^{2+\gamma }\left( x\right) \leq c_\delta (1 + \mu_{\delta}^{2+\gamma }\left( x\right) ),$ it follows that%
\begin{equation*}
\left\vert \text{\textrm{Cov}}_{\mathbb{P}_{x}}\left( f\left( X_{l}\right)
, f\left( X_{l+k}\right) \right) \right\vert \leq A\left( x\right) c_{\delta
}\kappa ^{k\min \left\{ 1, \gamma/2 \right\} /4},
\end{equation*}%
which finishes the proof Lemma \ref{Lemma Cov MC}.
\end{proof}

\begin{lemma}
\label{Lemma Cov erg MC} Let $0<\gamma \leq \min \left\{ 1, 2\delta
\right\}.$ Then:

a) There exists a real number $\mu $ not depending on $x$ such that,  for any
$k\geq 1, $%
\begin{equation*}
\left\vert \mathbb{E}_{x}f\left( X_{k}\right) -\mu \right\vert \leq
c_\delta A_{1}\left( x\right) \kappa ^{k\gamma /4-1},
\end{equation*}
where $A_{1}\left( x\right) = 1 + \mu _{\delta }\left( x\right) ^{1+\gamma }
+ \left\Vert \boldsymbol{\delta }_{x}\right\Vert_{\mathcal{B}^{\prime }}\left\Vert e\right\Vert _{\mathcal{B}}C_{\mathbf{P}}C_{Q}.$
Moreover
\begin{equation*}
\sum_{k=0}^{\infty }\left\vert \mathbb{E}_{x}f\left( X_{k}\right) -\mu
\right\vert \leq \overline{\mu }\left( x\right) = c_{\gamma, \kappa, \delta} A_{1}\left( x\right).
\end{equation*}%

b) There exists a sequence of (possibly complex) numbers $\left(
s_{k}\right) _{k\geq 0}$ not depending on $x$ such that
\begin{equation}
\left\vert \text{\textrm{Cov}}_{\mathbb{P}_{x}}\left( f\left( X_{l}\right)
, f\left( X_{l+k}\right) \right) -s_{k}\right\vert \leq c_\delta A_{2}\left( x\right)
\kappa ^{l\gamma /4-1},  \label{coverg000}
\end{equation}%
where
\begin{eqnarray*}
A_{2}\left( x\right) &=&  1+\mu _{\delta }\left( x\right)
^{2+\gamma } \\
&& + \left\Vert \boldsymbol{\delta }_{x}\right\Vert _{\mathcal{B}^{\prime
}}\left\Vert e\right\Vert _{\mathcal{B}}\left( C_{\mathbf{P}}^{2}C_{Q}\left(
\left\Vert \nu \right\Vert _{\mathcal{B}^{\prime }}\left\Vert e\right\Vert _{%
\mathcal{B}}+C_{Q}\right) +C_{\mathbf{P}}C_{Q}\left( 1+\left\Vert \nu
\right\Vert _{\mathcal{B}^{\prime }}C_{\mathbf{P}}\right) \right).
\end{eqnarray*}%
Moreover,  for
$k\geq 0, $%
\begin{equation*}
\left\vert s_{k}\right\vert \leq A_{2}\left( x\right) \kappa ^{k\gamma /4-1}
\end{equation*}%
and%
\begin{equation*}
\left\vert s_{0}\right\vert +2\sum_{k=1}^{\infty }\left\vert
s_{k}\right\vert \leq c_{\gamma, \kappa, \delta } A_{2}\left( x\right) .
\end{equation*}%
\end{lemma}

\begin{proof}
To avoid repetitions we prove first the part b).  We keep the notations from the proof of Lemma \ref{Lemma Cov MC}.
Denote $\widetilde{\phi }_{0}\left( t, u\right) =%
\widetilde{\phi }\left( t, u\right) -\widetilde{\phi }_{1}\left( t\right)
\widetilde{\phi }_{2}\left( u\right).$ By (\ref{cov00-AA}),  for any $%
l=0, 1, ...$%
\begin{equation*}
\text{\textrm{Cov}}_{\mathbb{P}_{x}}\left( f\left( X_{l}\right), f\left(
X_{l+k}\right) \right) =\frac{1}{\left( 2\pi \right) ^{2}}\int \int
\overline{\widehat{h}_{T}\left( t, u\right) }\widetilde{\phi }_{0}\left(
t, u\right) dtdu+R,
\end{equation*}%
with $R$ defined by (\ref{cov00-BB}). Since $V, V^{\prime }$ are independent
of $X_{l}, X_{l+k}, $
\begin{eqnarray}
\widetilde{\phi }\left( t, u\right) &=&\left( \mathbb{E}_{x}e^{i0%
\sum_{j=1}^{l-1}X_{j}+itX_{l}+i0\sum_{j=l+1}^{l+k-1}X_{j}+iuX_{l+k}}\right)
\mathbb{E}_{x}e^{itV}\mathbb{E}_{x}e^{iuV^{\prime }}  \notag \\
&=&\left( \mathbf{P}^{l-1}\mathbf{P}_{t}\mathbf{P}^{k-1}\mathbf{P}%
_{u}e\right) \left( x\right) \mathbb{E}_{x}e^{itV}\mathbb{E}%
_{x}e^{iuV^{\prime }}.  \label{coverg002}
\end{eqnarray}

Note that,  for $k, l\geq 2, $%
\begin{eqnarray*}
\left( \mathbf{P}^{l-1}\mathbf{P}_{t}\mathbf{P}^{k-1}\mathbf{P}_{u}e\right)
\left( x\right) &=&\boldsymbol{\delta }_{x}\left( \mathbf{P}^{l-1}\mathbf{P}%
_{t}\mathbf{P}^{k-1}\mathbf{P}_{u}e\right) \\
&=&\boldsymbol{\delta }_{x}\left( \Pi \mathbf{P}_{t}\mathbf{P}^{k-1}\mathbf{P%
}_{u}e\right) +\boldsymbol{\delta }_{x}\left( Q^{l-1}\mathbf{P}_{t}\mathbf{P}%
^{k-1}\mathbf{P}_{u}e\right) \\
&=&\nu \left( \mathbf{P}_{t}\mathbf{P}^{k-1}\mathbf{P}_{u}e\right) \\
&&+\boldsymbol{\delta }_{x}\left( Q^{l-1}\mathbf{P}_{t}\Pi \mathbf{P}%
_{u}e\right) +\nu \left( Q^{l-1}\mathbf{P}_{t}Q^{k-1}\mathbf{P}_{u}e\right).
\end{eqnarray*}%
Since%
\begin{equation*}
\left\vert \boldsymbol{\delta }_{x}\left( Q^{l-1}\mathbf{P}_{t}\Pi \mathbf{P}%
_{u}e\right) \right\vert =\left\vert \boldsymbol{\delta }_{x}\left( Q^{l-1}%
\mathbf{P}_{t}e\right) \nu \left( \mathbf{P}_{u}e\right) \right\vert \leq
\kappa ^{l-1}C_{Q}C_{\mathbf{P}}^{2}\left\Vert \boldsymbol{\delta }%
_{x}\right\Vert _{\mathcal{B}^{\prime }}\left\Vert \nu \right\Vert _{%
\mathcal{B}^{\prime }}\left\Vert e\right\Vert _{\mathcal{B}}^{2}
\end{equation*}%
and%
\begin{equation*}
\left\vert \boldsymbol{\delta }_{x}\left( Q^{l-1}\mathbf{P}_{t}Q^{k-1}%
\mathbf{P}_{u}e\right) \right\vert \leq \kappa ^{l+k-2}C_{Q}^{2}C_{\mathbf{P}%
}^{2}\left\Vert \boldsymbol{\delta }_{x}\right\Vert _{\mathcal{B}^{\prime
}}\left\Vert e\right\Vert _{\mathcal{B}^{\prime }},
\end{equation*}%
we obtain%
\begin{equation}
\left\vert \widetilde{\phi }\left( t, u\right) -\widetilde{\psi }\left(
t, u;k\right) \right\vert \leq \kappa ^{l-1}C_{\mathbf{P}}^{2}C_{Q}\left(
\left\Vert \nu \right\Vert _{\mathcal{B}^{\prime }}\left\Vert e\right\Vert _{%
\mathcal{B}}+C_{Q}\right) \left\Vert \boldsymbol{\delta }_{x}\right\Vert _{%
\mathcal{B}^{\prime }}\left\Vert e\right\Vert _{\mathcal{B}^{\prime }},
\label{coverg005}
\end{equation}%
where
\begin{equation*}
\widetilde{\psi }\left( t, u;k\right) =\nu \left( \mathbf{P}_{t}\mathbf{P}%
^{k-1}\mathbf{P}_{u}e\right) \mathbb{E}_{x}e^{itV}\mathbb{E}%
_{x}e^{iuV^{\prime }}.
\end{equation*}%
Note that $\widetilde{\psi }\left( t, u;k\right) $ does not depend on the
initial state $x$ since $V$ and $V^{\prime }$ are independent of the Markov
chain. In the same way%
\begin{eqnarray*}
\widetilde{\phi }_{1}\left( t\right) &=&\left( \mathbf{P}^{l-1}\mathbf{P}%
_{t}e\right) \left( x\right) \mathbb{E}_{x}e^{itV},  \\
\widetilde{\phi }_{2}\left( u\right) &=&\left( \mathbf{P}^{l+k-1}\mathbf{P}%
_{u}e\right) \left( x\right) \mathbb{E}_{x}e^{iuV^{\prime }},
\end{eqnarray*}%
where,  for $m\geq 2, $%
\begin{eqnarray*}
\left( \mathbf{P}^{m-1}\mathbf{P}_{t}e\right) \left( x\right) &=&\boldsymbol{%
\delta }_{x}\left( \mathbf{P}^{m-1}\mathbf{P}_{t}e\right) \\
&=&\boldsymbol{\delta }_{x}\left( \Pi \mathbf{P}_{t}e\right) +\boldsymbol{%
\delta }_{x}\left( Q^{m-1}\mathbf{P}_{t}e\right) \\
&=&\nu \left( \mathbf{P}_{t}e\right) +\boldsymbol{\delta }_{x}\left( Q^{m-1}%
\mathbf{P}_{t}e\right).
\end{eqnarray*}%
Since $\left\vert \boldsymbol{\delta }_{x}\left( Q^{m-1}\mathbf{P}%
_{t}e\right) \right\vert \leq \kappa ^{m-1}\left\Vert \boldsymbol{\delta }%
_{x}\right\Vert _{\mathcal{B}^{\prime }}\left\Vert e\right\Vert _{\mathcal{B}%
}C_{\mathbf{P}}C_{Q}, $ we get%
\begin{eqnarray}
\left\vert \widetilde{\phi }_{1}\left( t\right) -\widetilde{\psi }_{1}\left(
t\right) \right\vert &\leq &\kappa ^{l-1}\left\Vert \boldsymbol{\delta }%
_{x}\right\Vert _{\mathcal{B}^{\prime }}\left\Vert e\right\Vert _{\mathcal{B}%
}C_{\mathbf{P}}C_{Q},   \label{coverg005aa} \\
\left\vert \widetilde{\phi }_{2}\left( u\right) -\widetilde{\psi }_{1}\left(
u\right) \right\vert &\leq &\kappa ^{l+k-1}\left\Vert \boldsymbol{\delta }%
_{x}\right\Vert _{\mathcal{B}^{\prime }}\left\Vert e\right\Vert _{\mathcal{B}%
}C_{\mathbf{P}}C_{Q},   \label{coverg005bb}
\end{eqnarray}%
where%
\begin{equation}
\widetilde{\psi }_{1}\left( t\right) =\nu \left( \mathbf{P}_{t}e\right)
\mathbb{E}_{x}e^{itV}=\nu \left( \mathbf{P}_{t}e\right) \mathbb{E}%
_{x}e^{itV^{\prime }}  \label{coverg-psi}
\end{equation}%
does not depend on the initial state $x$ of the Markov chain. Denote $%
\widetilde{\psi }_{0}\left( t, u;k\right) =\widetilde{\psi }\left(
t, u;k\right) -\widetilde{\psi }_{1}\left( t\right) \widetilde{\psi }%
_{1}\left( u\right).$ From (\ref{coverg005aa}) and (\ref{coverg005bb}) it
follows%
\begin{eqnarray}
&&\left\vert \widetilde{\phi }_{0}\left( t, u\right) -\widetilde{\psi }%
_{0}\left( t, u;k\right) \right\vert  \notag \\
&\leq &\left\vert \widetilde{\phi }\left( t, u\right) -\widetilde{\psi }%
\left( t, u;k\right) \right\vert +\left\vert \widetilde{\phi }_{1}\left(
t\right) \widetilde{\phi }_{2}\left( u\right) -\widetilde{\psi }_{1}\left(
t\right) \widetilde{\psi }_{1}\left( u\right) \right\vert  \notag \\
&\leq &\left\vert \widetilde{\phi }\left( t, u\right) -\widetilde{\psi }%
\left( t, u;k\right) \right\vert +\left\vert \widetilde{\phi }_{1}\left(
t\right) -\widetilde{\psi }_{1}\left( t\right) \right\vert +\left\vert \nu
\left( \mathbf{P}_{t}e\right) \right\vert \left\vert \left( \widetilde{\phi }%
_{2}\left( u\right) -\widetilde{\psi }_{1}\left( u\right) \right) \right\vert
\notag \\
&\leq &\kappa ^{l-1}\left\Vert \boldsymbol{\delta }_{x}\right\Vert _{%
\mathcal{B}^{\prime }}\left\Vert e\right\Vert _{\mathcal{B}}C_{\mathbf{P}%
}^{2}C_{Q}\left( \left\Vert \nu \right\Vert _{\mathcal{B}^{\prime
}}\left\Vert e\right\Vert _{\mathcal{B}}+C_{Q}\right) +\kappa
^{l-1}\left\Vert \boldsymbol{\delta }_{x}\right\Vert _{\mathcal{B}^{\prime
}}C_{\mathbf{P}}C_{Q}\left( 1+\left\vert \nu \left( \mathbf{P}_{t}e\right)
\right\vert \right)  \notag \\
&\leq &\kappa ^{l-1}\left\Vert \boldsymbol{\delta }_{x}\right\Vert _{%
\mathcal{B}^{\prime }}\left\Vert e\right\Vert _{\mathcal{B}}\left( C_{%
\mathbf{P}}^{2}C_{Q}\left( \left\Vert \nu \right\Vert _{\mathcal{B}^{\prime
}}\left\Vert e\right\Vert _{\mathcal{B}}+C_{Q}\right) +C_{\mathbf{P}%
}C_{Q}\left( 1+\left\Vert \nu \right\Vert _{\mathcal{B}^{\prime }}C_{\mathbf{%
P}}\right) \right)  \notag \\
&\leq &C\left( x\right) \kappa ^{l-1},   \label{coverg005cc}
\end{eqnarray}%
where $C\left( x\right) =\left\Vert \boldsymbol{\delta }_{x}\right\Vert _{%
\mathcal{B}^{\prime }}\left\Vert e\right\Vert _{\mathcal{B}}\left( C_{%
\mathbf{P}}^{2}C_{Q}\left( \left\Vert \nu \right\Vert _{\mathcal{B}^{\prime
}}\left\Vert e\right\Vert _{\mathcal{B}}+C_{Q}\right) +C_{\mathbf{P}%
}C_{Q}\left( 1+\left\Vert \nu \right\Vert _{\mathcal{B}^{\prime }}C_{\mathbf{%
P}}\right) \right).$ Denote by $s_{k, T}$ the complex number defined by%
\begin{equation*}
s_{k, T}=\frac{1}{\left( 2\pi \right) ^{2}}\int_{-\varepsilon_0}^{\varepsilon_0} \int_{-\varepsilon_0}^{\varepsilon_0} \overline{\widehat{h}%
_{T}\left( t, u\right) }\widetilde{\psi }_{0}\left( t, u;k\right) dtdu.
\end{equation*}%
Note that $s_{k, T}$ does not depend on the initial state $x$ of the Markov
chain since so is $\widetilde{\psi }_{0}\left( t, u;k\right).$ With this
notation we have%
\begin{equation*}
\text{\textrm{Cov}}_{\mathbb{P}_{x}}\left( f\left( X_{l}\right), f\left(
X_{l+k}\right) \right) -s_{k, T}=R^{\prime }+R,
\end{equation*}%
where%
\begin{equation*}
R^{\prime }=\frac{1}{\left( 2\pi \right) ^{2}}
\int_{-\varepsilon_0}^{\varepsilon_0} \int_{-\varepsilon_0}^{\varepsilon_0} \overline{\widehat{h}%
_{T}\left( t, u\right) }\left( \widetilde{\phi }_{0}\left( t, u\right) -%
\widetilde{\psi }_{0}\left( t, u;k\right) \right) dtdu.
\end{equation*}%
Since $\mathbb{E}_{x}e^{itV}\mathbb{E}_{x}e^{iuV^{\prime }}$ has a support
in the square $[-\varepsilon _{0}, \varepsilon _{0}]^{2}, $ using (\ref{cov005}%
) and (\ref{coverg005cc}) it follows that
\begin{equation}
\left\vert R^{\prime }\right\vert \leq \frac{1}{\left( 2\pi \right) ^{2}}%
\left\Vert \widehat{h}_{T}\right\Vert _{L^{2}}\left\Vert \widetilde{\phi }%
_{0}-\widetilde{\psi }_{0}\right\Vert _{L^{2}}\leq \frac{T^{3}}{3\pi ^{2}}%
\varepsilon _{0}^{2}C\left( x\right) \kappa ^{l-1}.  \label{coverg006}
\end{equation}
From (\ref{coverg006}) and (\ref{R bound}),  for any $%
\gamma \in (0, 2\delta ]$ and any $l, k=0, 1, ..., $%
\begin{eqnarray}
\left\vert \text{\textrm{Cov}}_{\mathbb{P}_{x}}\left( f\left( X_{l}\right)
, f\left( X_{l+k}\right) \right) -s_{k, T}\right\vert &\leq &C\left( x\right)
\frac{T^{3}}{3\pi ^{2}}\varepsilon _{0}^{2}\kappa ^{l-1}  \notag \\
&&+ c_\delta T^{-\gamma }A_{0}^{2+\gamma }\left( x\right),
\label{coverg007}
\end{eqnarray}%
From (\ref{coverg007}),  for any $l, l^{\prime }=2, 3, ...$ one obtains%
\begin{eqnarray*}
&&\left\vert \text{\textrm{Cov}}_{\mathbb{P}_{x}}\left( f\left( X_{l}\right)
, f\left( X_{l+k}\right) \right) -\text{\textrm{Cov}}_{\mathbb{P}_{x}}\left(
f\left( X_{l^{\prime }}\right), f\left( X_{l^{\prime }+k}\right) \right)
\right\vert \\
&\leq & c_\delta T^{-\gamma }A_{0}\left( x\right) ^{2+\gamma }+C\left(
x\right) \frac{2T^{3}}{3\pi ^{2}}\varepsilon _{0}^{2}\kappa ^{\min
\{l, l^{\prime }\}-1}.
\end{eqnarray*}%
Taking $T=\kappa ^{-\frac{1}{4}\min \{l, l^{\prime }\}}$ we get,  for any $%
\gamma \leq \min \left\{ 1, 2\delta \right\}, $%
\begin{equation}
\left\vert \text{\textrm{Cov}}_{\mathbb{P}_{x}}(f\left( X_{l}\right)
, f\left( X_{l+k}\right) -\text{\textrm{Cov}}_{\mathbb{P}_{x}}(f\left(
X_{l^{\prime }}\right), f\left( X_{l^{\prime }+k}\right) \right\vert \leq
c_\delta A\left( x\right) \kappa ^{\min \{l, l^{\prime }\}\gamma /4-1},
\label{coverg008}
\end{equation}%
where $A\left( x\right) =  A_{0}^{2+\gamma }\left(x\right) +C\left( x\right) $.
The sequence $\textrm{%
Cov}_{\mathbb{P}_{x}}(f\left( X_{l}\right), f\left( X_{l+k}\right), %
l=1, 2, ...$ is thus Cauchy; denote by $s_{k}\left( x\right) $ its limit as $%
l\rightarrow \infty.$ Taking the limit as $l\rightarrow \infty $ in (\ref%
{coverg007}),  we get that
\begin{equation}
\left\vert s_{k}\left( x\right) -s_{k, T}\right\vert \leq C\left( x\right)
\frac{T^{3}}{3\pi ^{2}}\varepsilon _{0}^{2}\kappa ^{l-1}+ c_\delta T^{-\gamma }A_{0}^{2+\gamma }\left( x\right).  \label{coverg008a}
\end{equation}%
Letting $T=T_{l}=\kappa ^{-\frac{l}{4}}$ this implies that $%
\lim_{l\rightarrow \infty }s_{k, T_{l}}=s_{k}\left( x\right).$
Since $s_{k, T_{l}}$ does not depend on $x, $ we conclude that $s_{k}\left( x\right) $
is also a constant not depending on $x,$ say $s_{k}.$  Taking the limit as $%
l^{\prime }\rightarrow \infty $ in (\ref{coverg008}) we obtain (\ref%
{coverg000}).

The second assertion of the part b) of the lemma follows from (\ref%
{coverg000}) and Lemma \ref{Lemma Cov MC} setting $l=k.$

The third assertion of the part b) follows immediately from the second one.

Let us now prove the part a). From (\ref{cov-expect}),  we have%
\begin{equation*}
\left\vert \mathbb{E}_{x}f\left( X_{l}\right) -m_{T}\right\vert \leq \frac{1%
}{2\pi }\int_{-\varepsilon_0}^{\varepsilon_0} \left\vert \overline{\widehat{g}_{T}\left( t\right) }%
\right\vert \left\vert \widetilde{\phi }_{1}\left( t\right) -\widetilde{\psi
}_{1}\left( t\right) \right\vert dt+\left\vert R_{1}\right\vert,
\end{equation*}%
where%
\begin{equation*}
m_{T}=\frac{1}{2\pi }\int_{-\varepsilon_0}^{\varepsilon_0} \overline{\widehat{g}_{T}\left( t\right) }%
\widetilde{\psi }_{1}\left( t\right) dt,
\end{equation*}%
$R_{1}$ is defined by (\ref{cov-R1 bound}) and $\widetilde{\psi }_{1}$ is
defined by (\ref{coverg-psi}). Note that $m_{T}$ does not depend on $x$
since so is $\widetilde{\psi }\left( t\right).$ Taking into account the
bounds in (\ref{cov003cc}) and (\ref{coverg005aa}),  we get%
\begin{equation*}
\left\vert \mathbb{E}_{x}f\left( X_{l}\right) -m_{T}\right\vert \leq \kappa
^{l-1}\left\Vert \boldsymbol{\delta }_{x}\right\Vert _{\mathcal{B}^{\prime
}}\left\Vert e\right\Vert _{\mathcal{B}}C_{\mathbf{P}}C_{Q}\frac{1}{2\pi }%
\int_{-\varepsilon_0}^{\varepsilon_0} \left\vert \overline{\widehat{g}_{T}\left( t\right) }\right\vert
dt+  T^{-\gamma } c_\delta A_{0}^{1+\gamma }\left( x\right).
\end{equation*}%
Recalling that $g_{T}\left( x\right) =x1\left( \left\vert
x\right\vert \leq T\right), $ to bound $\int \left\vert \overline{\widehat{g}%
_{T}\left( t\right) }\right\vert dt$ we use the usual isometry relation
\begin{equation*}
\left( \int_{-\varepsilon_0}^{\varepsilon_0} \left\vert \overline{\widehat{g}_{T}\left( t\right) }\right\vert
dt\right) ^{2}\leq \int_{-\varepsilon_0}^{\varepsilon_0} \left\vert \overline{\widehat{g}_{T}\left( t\right) }%
\right\vert ^{2}dt=\int g_{T}^{2}\left( x\right) dx=\frac{2}{3}T^{3}.
\end{equation*}%
This implies,  for any $\gamma \leq \min \left\{ 1, 2\delta \right\}, $%
\begin{equation*}
\left\vert \mathbb{E}_{x}f\left( X_{l}\right) -m_{T}\right\vert \leq
c_\delta \left( \left\Vert \boldsymbol{\delta }_{x}\right\Vert _{\mathcal{B}^{\prime
}}\left\Vert e\right\Vert _{\mathcal{B}}C_{\mathbf{P}}C_{Q}T^{3}\kappa
^{l-1}+ T^{-\gamma }A_{0}^{1+\gamma }\left( x\right) \right).
\end{equation*}%
Taking $T=\kappa ^{-\frac{l}{4}}, $ we have%
\begin{equation}
\left\vert \mathbb{E}_{x}f\left( X_{l}\right) -m_{T}\right\vert \leq
c_\delta A_{1}\left( x\right) \kappa ^{l\gamma /4-1},   \label{coverg010}
\end{equation}%
where $A_{1}\left( x\right) =  1 + A_{0}^{1+\gamma }\left( x\right) + \left\Vert \boldsymbol{\delta }_{x}\right\Vert
_{\mathcal{B}^{\prime }}\left\Vert e\right\Vert _{\mathcal{B}}C_{\mathbf{P}}C_{Q}.$
From this inequality it follows that%
\begin{equation}
\left\vert \mathbb{E}_{x}f\left( X_{l}\right) -\mathbb{E}_{x}f\left(
X_{k}\right) \right\vert \leq c_\delta A_{1}\left( x\right) \kappa ^{\min \left\{
l, k\right\} \gamma /4-1},   \label{cauchy001}
\end{equation}%
which proves that the sequence $\left( \mathbb{E}_{x}f\left( X_{l}\right)
\right) _{l\geq 1}$ is Cauchy and therefore has a limit denoted $\mu \left(
x\right).$ Since $m_{T}$ does not depend on $x, $ taking $\lim $ as $%
l\rightarrow \infty $ in (\ref{coverg010}) we conclude that $\mu \left(
x\right) =\mu $ is not depending on $x.$ Taking $\lim_{k\rightarrow \infty }$
in (\ref{cauchy001}),  we get%
\begin{equation*}
\left\vert \mathbb{E}_{x}f\left( X_{l}\right) -\mu \right\vert \leq
c_\delta A_{1}\left( x\right) \kappa ^{l\min \left\{ 1, \gamma \right\} /4-1},
\end{equation*}%
which proves the first assertion of the part a) of the lemma. The second
follows from the first one.
\end{proof}

The bound (\ref{Proposition MC002-001}) of Proposition \ref{Proposition MC002} follows from the part a) of the previous lemma.
It remains to prove the bound (\ref{Proposition MC002-002}) of Proposition \ref{Proposition MC002}.

Let $0<\gamma \leq \min \left\{ 1, 2\delta \right\}.$
First note that,  from Lemma \ref{Lemma Cov erg MC} and Lemma \ref{Lemma Cov
MC} we obtain,  for $k=0, 1, ..., $%
\begin{equation*}
\left\vert \text{\textrm{Cov}}_{\mathbb{P}_{x}}(f\left( X_{l}\right)
, f\left( X_{l+k}\right) -s_{k}\right\vert \leq A_{2}\left( x\right)
c_{\delta, \kappa }\kappa ^{c_{\gamma, \kappa }\max \left\{ l, k\right\} },
\end{equation*}%
where $A_{2}\left( x\right) $ is defined in Proposition \ref{Proposition MC002}.
Then,  for any $k=0, 1, ..., $%
\begin{eqnarray*}
\sum_{l=m}^{m+n-1}\sum_{k=1}^{m+n-l}\left\vert \text{\textrm{Cov}}_{\mathbb{P%
}_{x}}(f\left( X_{l}\right), f\left( X_{l+k}\right) -s_{k}\right\vert &\leq
&A_{2}\left( x\right) c_{\delta, \kappa
}\sum_{l=m}^{m+n-1}\sum_{k=1}^{m+n-l}e^{-c_{\gamma, \kappa }\max \left\{
l, k\right\} } \\
&\leq &A_{2}\left( x\right) c_{\delta, \gamma, \kappa }^{\prime \prime }.
\end{eqnarray*}
Since
\begin{equation*}
\text{\textrm{Var}}_{\mathbb{P}_{x}}\left( \sum_{l=m}^{m+n-1}f\left(
X_{l}\right) \right) =\sum_{l=m}^{m+n-1}\text{\textrm{Var}}_{\mathbb{P}%
_{x}}\left( f\left( X_{l}\right) \right)
+2\sum_{l=m}^{m+n-1}\sum_{k=1}^{m+n-l}\text{\textrm{Cov}}_{\mathbb{P}%
_{x}}(f\left( X_{l}\right), f\left( X_{l+k}\right)
\end{equation*}%
we get%
\begin{equation*}
\left\vert \text{\textrm{Var}}_{\mathbb{P}_{x}}\left(
\sum_{l=m}^{m+n-1}f\left( X_{l}\right) \right) -\left(
ns_{0}+\sum_{l=m}^{m+n-1}\sum_{k=1}^{m+n-l}\left( s_{k}+s_{k}^{\ast }\right)
\right) \right\vert \leq A_{2}\left( x\right) c_{\delta, \gamma, \kappa
}^{\prime \prime }.
\end{equation*}%
Taking into account that,  by Lemma \ref{Lemma Cov erg MC}, $s_k$ are not depending on $x$ and
$\left\vert s_{k}\right\vert \leq A_{2}\left( x\right) \kappa ^{k\gamma /4-1} $ we obtain
\begin{equation}
\left\vert \text{\textrm{Var}}_{\mathbb{P}_{x}}\left(
\sum_{l=m}^{m+n-1}f\left( X_{l}\right) \right) -n\left(
s_{0}+\sum_{k=1}^{\infty }\left( s_{k}+s_{k}^{\ast }\right) \right)
\right\vert \leq A_{2}\left( x\right) c_{\delta, \gamma, \kappa }^{\prime
\prime \prime }.   \label{varlim001}
\end{equation}
Dividing by $n$ and taking the limit as $n\rightarrow \infty $
in (\ref{varlim001}),  we deduce that $s_{0}+\sum_{k=1}^{\infty }\left(
s_{k}+s_{k}^{\ast }\right)$ converges to a non-negative number not depending on $x$, say $\sigma^{2}\geq 0.$
Now (\ref{Proposition MC002-002}) follows from (\ref{varlim001}).

\subsection{Proof of Theorem \protect\ref{Th main res 2}}

First note that conditions \textbf{C1} and \textbf{C3} are satisfied by
Propositions \ref{Proposition MC001} and \ref{Proposition MC002}. Condition
\textbf{C2} is satisfied by hypothesis \textbf{M4}. Let $\mu _{i}\left(
x\right) =\mathbb{E}_{x}f\left( X_{i}\right).$ Let $\alpha <\delta $ and $%
\delta ^{\prime }=\frac{1}{2}\left( \alpha +\delta \right).$ Since $\alpha
<\delta ^{\prime }, $ from \ref{Th main res 1} with $\delta ^{\prime }$
replacing $\delta, $ it follows that for any $x\in \mathbb{X}$ there exists
a probability space $(\Omega, \mathcal{F}, \mathbb{P}_{x}), $ a sequence of
independent standard normal r.v.'s $\left( W_{i}^{\prime }\right) _{i\geq 1}$
and a sequence of r.v.'s $(Y_{i}^{\prime })_{i\geq 1}$ such that $\left(
Y_{i}^{\prime }\right) _{i\geq 1}\overset{d}{=}\left( f\left( X_{i}\right)
\right) _{i\geq 1}$ and for any $0<\rho <\frac{1}{2}\frac{\alpha }{1+2\alpha
}, $%
\begin{equation}
\mathbb{P}_{x}\left( N^{-\frac{1}{2}}\sup_{k\leq N}\left\vert
\sum_{i=1}^{k}\left( Y_{i}^{\prime }-\mu _{i}\left( x\right) -\sigma
W_{i}^{\prime }\right) \right\vert >N^{-\rho }\right) \leq C_{0}\left(
x\right) N^{-\alpha \frac{1+\alpha }{1+2\alpha }+\rho \left( 2+2\alpha
\right) },   \label{th2-002}
\end{equation}%
where $C_{0}\left( x\right) =C_{0}^{\prime }\left( 1+\lambda _{0}\left(
x\right) +\mu _{\delta ^{\prime }}\left( x\right) +\sqrt{\tau \left(
x\right) }\right) ^{2+2\delta ^{\prime }}$
and $\lambda _{0}\left( x\right), \mu _{\delta }\left( x\right), \tau \left(
x\right), \lambda _{1}, \lambda _{2}$ and $\sigma ^{2}$ are defined in
Propositions \ref{Proposition MC001} and \ref{Proposition MC002}. If $%
\overline{\mu }\left( x\right) \leq N^{\frac{1}{2}-\rho }$ (with $\overline{%
\mu }\left( x\right) $ from Proposition \ref{Proposition MC002} using (\ref%
{th2-002}) we have%
\begin{eqnarray}
&&\mathbb{P}_{x}\left( N^{-\frac{1}{2}}\sup_{k\leq N}\left\vert
\sum_{i=1}^{k}\left( Y_{i}^{\prime }-\mu -\sigma W_{i}^{\prime }\right)
\right\vert >2N^{-\rho }\right)  \notag \\
&\leq &\mathbb{P}_{x}\left( \sup_{k\leq N}\left\vert \sum_{i=1}^{k}\left(
Y_{i}^{\prime }-\mu _{i}\left( x\right) -\sigma W_{i}^{\prime }\right)
\right\vert >2N^{\frac{1}{2}-\rho }-\overline{\mu }\left( x\right) \right)
\notag \\
&\leq &C_{0}\left( x\right) N^{-\alpha \frac{1+\alpha }{1+2\alpha }+\rho
\left( 2+2\alpha \right) }.  \label{th2-003}
\end{eqnarray}%
If $\overline{\mu }\left( x\right) >N^{\frac{1}{2}-\rho }, $ it is obvious
that%
\begin{equation}
1\leq \left( \overline{\mu }\left( x\right) N^{-\frac{1}{2}+\rho }\right)
^{2\alpha }\leq \overline{\mu }\left( x\right) ^{2\alpha }N^{-\alpha +2\rho
\alpha }.  \label{th2-004}
\end{equation}%
From (\ref{th2-003}) and (\ref{th2-004}) we get%
\begin{eqnarray*}
&&\mathbb{P}_{x}\left( N^{-\frac{1}{2}}\sup_{k\leq N}\left\vert
\sum_{i=1}^{k}\left( Y_{i}^{\prime }-\mu -\sigma W_{i}^{\prime }\right)
\right\vert >2N^{-\rho }\right) \\
&\leq &\left( C_{0}\left( x\right) +\overline{\mu }\left( x\right) ^{2\alpha
}\right) N^{-\alpha \frac{1+\alpha }{1+2\alpha }+\rho \left( 2+2\alpha
\right) }.
\end{eqnarray*}%
Taking into account the expressions for $\lambda _{0}\left( x\right), \mu
_{\delta }\left( x\right), \tau \left( x\right), \lambda _{1}, \lambda _{2}, $
$\overline{\mu }\left( x\right) $ and choosing $\gamma $ small we obtain%
\begin{equation*}
C_{0}\left( x\right) +\overline{\mu }\left( x\right) ^{2\alpha }\leq C\left(
x\right) =C_{1}\left( 1+\left\Vert \boldsymbol{\delta }_{x}\right\Vert _{%
\mathcal{B}^{\prime }}+\mu _{\delta }\left( x\right) \right) ^{2+2\delta },
\end{equation*}%
where $C_{1}$ is a constant depending only on
$\delta, \alpha, \kappa, C_{\mathbf{P}}, C_{Q}, \left\Vert e\right\Vert _{\mathcal{B}}, \left\Vert \nu \right\Vert _{\mathcal{B}^{\prime }}.$

Generally the measure $\mathbb{P}_{x}$ and the constructed sequence $\left(
Y_{i}^{\prime }\right) _{i\geq 1}$ depend both on the initial state $x.$ It
is easy to reconstruct $\left( Y_{i}^{\prime }\right) _{i\geq 1}$
independently of $x.$ Indeed,  on the canonical space $\widetilde{\Omega }=%
\mathbb{R}^{\infty }\mathbb{\times R}^{\infty }$ there is a probability
measure $\widetilde{\mathbb{P}}_{x}$ which coincides with the joint
distribution of the sequence $\left( Y_{i}^{\prime }, W\right) _{i\geq 1}.$
It is enough to redefine $Y_{i}^{\prime }=\omega _{1, i}$ and $W_{i}=\omega
_{2, i}$ as the coordinate processes,  where $\omega =\left( \omega
_{1}, \omega _{2}\right) \in \widetilde{\Omega }.$ With this construction
only the measure $\widetilde{\mathbb{P}}_{x}$ depends on the initial state $%
x.$ The measurability of the map $x\mathbb{\in X}\rightarrow \widetilde{%
\mathbb{P}}_{x}\left( \cdot \right) $ follows from the construction.

\subsection{Proof of Theorem \protect\ref{Th main res 3}}

In addition to conditions of Theorem \ref{Th main res 2} assume hypothesis \textbf{M5}.
Set for brevity $Y_k=f(X_k).$
First we note that hypothesis \textbf{M5} ensures the existence of the
mean $\nu \left( f\right) =\mathbf{E}_{\nu }Y_{k}=\int \left( \mathbb{E}_{x}Y_{k}\right) \nu \left( dx\right) $ and of the mixed moment $
\mathbf{E}%
_{\nu }\left( Y_{l} Y_{l+k}\right) =\int \mathbb{E}_{x}\left(Y_{l} Y_{l+k}\right) \nu \left( dx\right) $ with respect to the invariant
measure. By Proposition \ref{Proposition MC002},  
we have $\lim_{k\rightarrow \infty }\mathbb{E}_{x}Y_{k}=\mu, $ $\nu $-a.s. on $\mathbb{X} .$ Then by
Lebesgue theorem on dominated convergence%
\begin{equation*}
\nu \left( f\right) =\mathbf{E}_{\nu }Y_{k}=\lim_{k\rightarrow \infty }\int
\left( \mathbb{E}_{x}Y_{k}\right) \nu \left( dx\right) =\int \left(
\lim_{k\rightarrow \infty }\mathbb{E}_{x}Y_{k}\right) \nu \left( dx\right)
=\mu.
\end{equation*}%
Without loss of generality we can assume that $\nu \left( f\right) =0.$
Using hypothesis \textbf{M5} and $\nu \left( f\right) =0, $ we have%
\begin{eqnarray*}
\int \mathrm{Cov}_{\mathbf{P}_{x}}\left( Y_{l}, Y_{l+k}\right) \nu \left(
dx\right) &=&\int \mathbb{E}_{x}\left( Y_{l}Y_{l+k}\right) \nu \left(
dx\right) -\int \mathbb{E}_{x}\left( Y_{l}\right) \mathbb{E}_{x}\left(
Y_{l+k}\right) \nu \left( dx\right) \\
&=&\mathbf{E}_{\nu }\left( Y_{l}Y_{l+k}\right) -\int \mathbb{E}_{x}\left(
Y_{l}\right) \mathbb{E}_{x}\left( Y_{l+k}\right) \nu \left( dx\right) \\
&=&\mathrm{Cov}_{\mathbb{P}_{\nu }}\left( Y_{0}, Y_{k}\right) -\int \mathbb{E}%
_{x}\left( Y_{l}\right) \mathbb{E}_{x}\left( Y_{l+k}\right) \nu \left(dx\right).
\end{eqnarray*}%
By Proposition \ref{Proposition MC002},  we have 
$\lim_{l\rightarrow \infty }\mathrm{Cov}_{\mathbf{P}_{x}}\left(Y_{l}, Y_{l+k}\right) = s_{k}$ 
and 
$\lim_{l\rightarrow \infty }\mathbb{E}_{x} Y_{l} =0$ 
for any $x\in \mathbb{X}$. As before,  integrating with respect to the
stationary measure and using Lebesgue theorem on dominated convergence,  it
follows that 
$s_{k}=\mathrm{Cov}_{\mathbb{P}_{\nu }}\left(Y_{0}, Y_{k}\right).$ 
Thus the conclusions of Theorem \ref{Th main res 2}
hold true with $\mu =\nu \left( f\right) $ and 
$\sigma ^{2}=\sigma _{\nu}^{2}, $ which proves Theorem \ref{Th main res 3}.

\section{Maximal inequalities\label{sec Lp bounds}}

In this section we state two bounds which are used repeatedly in the paper.
The first one gives a control for the $L_{p}$-norm of the maxima of the
partial sums of a sequence of dependent r.v.'s. This proposition is a
consequence of the second one which give a control of the $L_{p}$-norm of
the partial sums of a sequence of dependent r.v.'s.
It is assumed that conditions \textbf{C1} and \textbf{C2} hold true.

\begin{proposition}
\label{Prop max Lp bound} Let $\delta ^{\prime }<\delta $ and $\epsilon >0.$
Then,  there is a constant $c_{\lambda _{1}, \lambda _{2}, \delta, \delta ^{\prime }, \epsilon }$ such that for any $m, n\geq 1$ it holds%
\begin{equation*}
\left\Vert \sup_{1\leq k\leq n}\left\vert \sum_{i=m}^{m+k-1}X_{i}\right\vert
\right\Vert _{L^{2+2\delta ^{\prime }}}\leq c_{\lambda _{1}, \lambda
_{2}, \delta, \delta ^{\prime }, \epsilon }\left( 1+\lambda _{0}+\mu _{\delta
}\right) ^{1+\epsilon }n^{\frac{1}{2}}.
\end{equation*}
\end{proposition}

\begin{proof}
Denote for brevity $S_{m, n}=\sum_{i=m}^{m+n-1}X_{i}.$ Let $\delta ^{\prime
\prime }$ be such that $\delta ^{\prime }<\delta ^{\prime \prime }<\delta.$
By Proposition \ref{Prop Lp bound} (which we assume for the moment),  for any
$m, n\geq 1$ and $\epsilon >0$ we have $\left\Vert S_{m, n}\right\Vert
_{L^{2+2\delta ^{\prime \prime }}}\leq An^{\frac{1}{2}}, $ where $%
A=c_{\lambda _{1}, \lambda _{2}, \delta ^{\prime \prime }, \epsilon }\left(
1+\lambda _{0}+\mu _{\delta }\right) ^{1+\epsilon }.$
Substituting $S_{m, n}^{\prime }=S_{m, n}/A, $
we get $\left\Vert S_{m, n}^{\prime }\right\Vert _{L^{2+2\delta ^{\prime
\prime }}}\leq n^{\frac{1}{2}}, $ for any $m, n\geq 1.$ By Theorem A in
Serfling \cite{Serfl} (see also Billingsley \cite{Blngsl68} p. 102),  it
follows that $\left\Vert \sup_{1\leq k\leq n}S_{m, n}^{\prime }\right\Vert
_{L^{2+2\delta ^{\prime \prime }}}\leq n^{\frac{1}{2}}\log _{2}\left(
4n\right), $ for any $m, n\geq 1.$ Since $\delta ^{\prime }<\delta ^{\prime
\prime }, $ it follows that
\begin{equation*}
\left\Vert \sup_{1\leq k\leq n}S_{m, n}^{\prime }\right\Vert _{L^{2+2\delta
^{\prime }}}\leq \left\Vert \sup_{1\leq k\leq n}S_{m, n}^{\prime }\right\Vert
_{L^{2+2\delta ^{\prime \prime }}}\leq \left( n^{\frac{1}{2}}\log _{2}\left(
4n\right) \right) ^{\frac{2+2\delta ^{\prime \prime }}{2+2\delta ^{\prime }}%
}\leq c_{\delta, \delta ^{\prime }}n^{\frac{1}{2}},
\end{equation*}%
from which we deduce $\left\Vert \sup_{1\leq k\leq n}S_{m, n}\right\Vert
_{L^{2+2\delta ^{\prime }}}\leq Ac_{\delta, \delta ^{\prime }}n^{\frac{1}{2}%
}.$
\end{proof}

The following assertion is an adaptation of Proposition 4.1 in Gou\"{e}zel
\cite{Gouez}. In order to derive an explicit
dependence of the constant involved in the bound on some of the constants in
conditions \textbf{C1} and \textbf{C2} we give an independent proof.
Tracking this explicit dependence
plays a crucial role in the proof of Theorem \ref{Th main res 2} to work out
the dependence of the bound on the initial state of the Markov chain $%
X_{0}=x.$

\begin{proposition}
\label{Prop Lp bound} Let $0<\delta ^{\prime }<\delta $ and $\epsilon >0.$ Then,
$c_{\lambda _{1}, \lambda _{2}, \delta ^{\prime }, \epsilon }$ such that,  for any $m, n\geq 1, $%
\begin{equation*}
\left\Vert \sum_{i=m}^{m+n-1}X_{i}\right\Vert _{L^{2+2\delta ^{\prime
}}}\leq c_{\lambda _{1}, \lambda _{2}, \delta ^{\prime }, \epsilon }\left(
1+\lambda _{0}+\mu _{\delta }\right) ^{1+\epsilon }n^{\frac{1}{2}}.
\end{equation*}
\end{proposition}

The proof of this proposition is given below. First we state several
auxiliary assertions.

\subsection{Auxiliary assertions}

\begin{proposition}
\label{Proposition boundL2} There is a constant $c_{\lambda_{1}, \lambda _{2}, \epsilon }$ such that, for any $\epsilon >0, $%
\begin{equation}
\left\Vert \sum_{i=m}^{m+n-1}X_{i}\right\Vert _{L^{2}}\leq c_{\lambda_{1}, \lambda _{2}, \epsilon }\left( 1+\lambda _{0}+\mu _{\delta }\right)
^{1+\epsilon }n^{\frac{1}{2}}.   \label{bound of  L2 norm}
\end{equation}
\end{proposition}

The proof is based on the following two lemmas.
\begin{lemma}
\label{Lemma semiadd1}Let $u_{n}=\max_{m\geq 1}\left\Vert
\sum_{i=m}^{m+n-1}X_{i}\right\Vert _{L^{2}}^{2}, $ $n\geq 1.$  Then,  for any natural numbers $a, $ $%
b\geq 1$ and any $\alpha \in (0, \frac{1}{2}), $ $\gamma \in \left( 0, \delta
\right), $ it holds
\begin{equation*}
u_{a+b}\leq A+u_{a}+u_{b}+\left( c\mu _{\delta }\right) ^{2}\left(
a^{2\alpha }+b^{2\alpha }\right) +c\mu _{\delta }\left( a^{\alpha
}+b^{\alpha }\right) \left( A+u_{a}+u_{b}\right)
^{1/2}+cu_{a}^{1/2}+cu_{b}^{1/2},
\end{equation*}%
where $c>1$ and $A=c_{\lambda _{1}, \lambda _{2}, \gamma
, \alpha }\left( 1+\lambda _{0}+\mu _{\delta }\right) ^{2+\gamma }.$
\end{lemma}

\begin{proof}
Let $m\in \mathbb{N}.$ Assume that $a\leq b$ (the case $a>b$ is treated in
the same manner). Denote $Y_{1}=\sum_{i=m}^{m+a-1}X_{i}, $ $Y_{2}=\sum_{i=m+a+%
\left[ b^{\alpha }\right] }^{m+a+b-1}X_{i}$ and $Y_{0}=%
\sum_{i=m}^{m+a+b-1}X_{i}, $ where $\alpha \in (0, \frac{1}{2}).$ Note that $%
Y_{0}=Y_{1}+Y_{2}+Y_{gap}, $ where $Y_{gap}=\sum_{i=m+a}^{m+a+\left[
b^{\alpha }\right] -1}X_{i}.$ Therefore%
\begin{equation}
\left\Vert Y_{0}\right\Vert _{L^{2}}^{2}\leq \left\Vert
Y_{1}+Y_{2}\right\Vert _{L^{2}}^{2}+\left\Vert Y_{gap}\right\Vert
_{L^{2}}^{2}+2\left\Vert Y_{1}+Y_{2}\right\Vert _{L^{2}}\left\Vert
Y_{gap}\right\Vert _{L^{2}}.  \label{mm001}
\end{equation}%
In the sequel we shall bound each of the terms in the right-hand side of (%
\ref{mm001}).

Let $V_{1}$ and $V_{2}$ be two independent identically distributed r.v.'s of
mean $0, $ independent of $Y_{1}$ and $Y_{2}$ with a common characteristic
function supported in the interval $[-\varepsilon _{0}, \varepsilon _{0}], $
for some $\varepsilon _{0}\in \left( 0, 1\right) $ and such that $\left\Vert
V_{i}\right\Vert _{L^{2+2\delta }}\leq c. $ Denote $\widetilde{Y}_{1}=Y_{1}+V_{1}$ and $\widetilde{Y}%
_{2}=Y_{2}+V_{2}.$ Let $Z_{1}$ and $Z_{2}$ be independent copies of $%
\widetilde{Y}_{1}$ and $\widetilde{Y}_{2}.$ Since $\mathbf{E}e^{itV_{1}}$ is
supported in the interval $[-\varepsilon _{0}, \varepsilon _{0}], $ by Lemma %
\ref{Lemma smoo},  for any $T>0, $
\begin{eqnarray*}
\pi \left( \left( \widetilde{Y}_{1}, \widetilde{Y}_{2}\right), \left(
Z_{1}, Z_{2}\right) \right) &\leq &\frac{T}{\pi }\left( \int_{\left[
-\varepsilon _{0}, \varepsilon _{0}\right] ^{2}}\left\vert  \phi \left(
t, u\right) -\psi _{1}\left( t\right) \psi _{2}\left( u\right)
\right\vert ^{2}dtdu\right) ^{1/2} \\
&&+\mathbb{P}\left( \max \left\{ \left\vert \widetilde{Y}_{1}\right\vert
, \left\vert \widetilde{Y}_{2}\right\vert \right\} >T\right),
\end{eqnarray*}%
where $\phi $ is the characteristic function of the vector $\left(
Y_{1}, Y_{2}\right) $ and $\psi _{1}$ and $\psi _{2}$ are the characteristic
functions of the r.v.'s $Y_{1}$ and $Y_{2}.$ Condition \textbf{C1} implies
that%
\begin{eqnarray*}
\left\vert \left( \phi \left( t, u\right) -\psi _{1}\left( t\right) \psi
_{2}\left( u\right) \right) \right\vert &\leq &\lambda _{0}\left( 1+b\right)
^{2\lambda _{2}}\exp \left( -\lambda _{1}\left[ b^{\alpha }\right] \right) \\
&\leq &\lambda _{0}c_{\lambda _{1}}\left( 1+b\right) ^{2\lambda _{2}}\exp
\left( -\lambda _{1}b^{\alpha }\right).
\end{eqnarray*}%
Let $T=e^{\lambda _{1}b^{\alpha }/2}.$ Taking into account that%
\begin{eqnarray*}
\mathbb{P}\left( \max \left\{ \left\vert \widetilde{Y}_{1}\right\vert
, \left\vert \widetilde{Y}_{2}\right\vert \right\} >T\right) &\leq &T^{-1}%
\mathbb{E}\max \left\{ \left\vert \widetilde{Y}_{1}\right\vert, \left\vert
\widetilde{Y}_{2}\right\vert \right\} \\
&\leq &T^{-1}\left( \left\Vert Y_{1}+V_{1}\right\Vert _{L^{2+2\delta
}}+\left\Vert Y_{2}+V_{2}\right\Vert _{L^{2+2\delta }}\right) \\
&\leq &e^{-\frac{\lambda _{1}}{2}b^{\alpha }}\left( c+\left( a+b\right)
\max_{l\geq 0}\left\Vert X_{l}\right\Vert _{L^{2+2\delta }}\right) \\
&\leq &c_{\delta }e^{-\frac{\lambda _{1}}{2}b^{\alpha }}b\left( 1+\mu
_{\delta }\right)
\end{eqnarray*}%
we obtain%
\begin{eqnarray*}
\pi \left( \left( \widetilde{Y}_{1}, \widetilde{Y}_{2}\right), \left(
Z_{1}, Z_{2}\right) \right) &\leq &\frac{1}{\pi }\lambda _{0}\left(
1+b\right) ^{2\lambda _{2}}e^{-\frac{\lambda _{1}}{2}b^{\alpha }}+c_{\delta
}be^{-\frac{\lambda _{1}}{2}b^{\alpha }}\left( 1+\mu _{\delta }\right) \\
&\leq &\Delta =c_{\delta, \alpha }\left( 1+b\right) ^{2\lambda _{2}}e^{-\frac{%
\lambda _{1}}{2}b^{\alpha }}\left( 1+\lambda _{0}+\mu _{\delta }\right).
\end{eqnarray*}
By Lemma \ref{lemma strassen-dudley}
there is a coupling of $\left(
\widetilde{Y}_{1}, \widetilde{Y}_{2}\right) $ and $\left( Z_{1}, Z_{2}\right) $
such that%
\begin{equation*}
P\left( \left\Vert \left( \widetilde{Y}_{1}, \widetilde{Y}_{2}\right) -\left(
Z_{1}, Z_{2}\right) \right\Vert _{\infty }\geq \Delta \right) \leq \Delta.
\end{equation*}%
Let $S=\widetilde{Y}_{1}+\widetilde{Y}_{2}-\left( Z_{1}+Z_{2}\right).$
Taking into account that $\left\Vert V_{i}\right\Vert _{L^{2+2\delta }}\leq
c, $ we have
\begin{eqnarray}
\left\Vert S\right\Vert _{L^{2+2\delta }} &=&\left\Vert \widetilde{Y}_{1}+%
\widetilde{Y}_{2}-\left( Z_{1}+Z_{2}\right) \right\Vert _{L^{2+2\delta }}
\notag \\
&\leq &2\left\Vert \widetilde{Y}_{1}+\widetilde{Y}_{2}\right\Vert
_{L^{2+2\delta }}  \notag \\
&\leq &c\left( a+b\right) \left( 1+\max_{l\geq 0}\left\Vert X_{l}\right\Vert
_{L^{2+2\delta }}\right)  \notag \\
&\leq &cb\left( 1+\mu _{\delta }\right).  \label{mm003}
\end{eqnarray}%
Then,  for any $\gamma \in (0, \delta ), $%
\begin{eqnarray}
\left\Vert S\right\Vert _{L^{2}}^{2} &\leq &4\Delta ^{2}+\mathbb{E}%
\left\vert S\right\vert ^{2}1\left( \left\vert S\right\vert \geq 2\Delta
\right)  \notag \\
&\leq &4\Delta ^{2}+\left\Vert S\right\Vert _{L^{2+2\gamma }}^{2}\mathbb{P}%
\left( \left\vert S\right\vert \geq 2\Delta \right) ^{\frac{\gamma }{%
1+\gamma }}  \notag \\
&\leq &4\Delta ^{2}+cb^{2}\left( 1+\mu _{\delta }\right) ^{2}\Delta ^{\frac{%
\gamma }{1+\gamma }}  \notag \\
&\leq &4c^{2}\left( 1+b\right) ^{4\lambda _{2}}\left( 1+\lambda _{0}+\mu
_{\delta }\right) ^{2}e^{-\lambda _{1}b^{\alpha }}  \notag \\
&&+4b^{2}\left( 1+\mu _{\delta }\right) ^{2}c_{\delta, \alpha }^{\frac{\gamma
}{1+\gamma }}\left( 1+b\right) ^{4\lambda _{2}\frac{\gamma }{1+\gamma }}e^{-%
\frac{\lambda _{1}}{2}\frac{\gamma }{1+\gamma }b^{\alpha }}\left( 1+\lambda
_{0}+\mu _{\delta }\right) ^{\frac{\gamma }{1+\gamma }}  \notag \\
&\leq &A^{\prime },   \label{mm004}
\end{eqnarray}%
where $A^{\prime }=c_{\lambda _{1}, \lambda _{2}, \gamma, \alpha }^{\prime
}\left( 1+\lambda _{0}+\mu _{\delta }\right) ^{2+\gamma }.$
From (\ref{mm003}) and (\ref{mm004}),  it follows that
\begin{eqnarray}
\left\Vert \widetilde{Y}_{1}+\widetilde{Y}_{2}\right\Vert _{L^{2}}^{2} &\leq
&\left\Vert S\right\Vert _{L^{2}}^{2}+\left\Vert Z_{1}\right\Vert
_{L^{2}}^{2}+\left\Vert Z_{2}\right\Vert _{L^{2}}^{2}  \notag \\
&\leq &A^{\prime }+\left\Vert \widetilde{Y}_{1}\right\Vert
_{L^{2}}^{2}+\left\Vert \widetilde{Y}_{2}\right\Vert _{L^{2}}^{2}.
\label{mm005a}
\end{eqnarray}%
Since $\left\Vert V_{i}\right\Vert _{L^{2}}\leq c, $ we have%
\begin{equation}
\left\Vert Y_{1}+Y_{2}\right\Vert _{L^{2}}\leq \left\Vert \widetilde{Y}_{1}+%
\widetilde{Y}_{2}\right\Vert _{L^{2}}+2c.  \label{mm005b}
\end{equation}%
Taking into account (\ref{mm005a}) and (\ref{mm005b}),  one gets%
\begin{eqnarray}
&&\left\Vert Y_{1}+Y_{2}\right\Vert _{L^{2}}^{2}  \label{mm005d} \\
&\leq &A^{\prime }+\left\Vert \widetilde{Y}_{1}\right\Vert
_{L^{2}}^{2}+\left\Vert \widetilde{Y}_{2}\right\Vert
_{L^{2}}^{2}+4c\left\Vert \widetilde{Y}_{1}\right\Vert _{L^{2}}+4c\left\Vert
\widetilde{Y}_{2}\right\Vert _{L^{2}}+4c^{2}.  \notag
\end{eqnarray}
Since $\left\Vert \widetilde{Y}_{k}\right\Vert _{L^{2}}\leq \left\Vert
Y_{k}\right\Vert _{L^{2}}+c, $ it holds%
\begin{eqnarray}
\left\Vert Y_{1}+Y_{2}\right\Vert _{L^{2}}^{2} &\leq &A^{\prime }+\left(
\left\Vert Y_{1}\right\Vert _{L^{2}}+c\right) ^{2}+\left( \left\Vert
Y_{2}\right\Vert _{L^{2}}+c\right) ^{2}  \notag \\
&&+4c\left( \left\Vert Y_{1}\right\Vert _{L^{2}}+c\right) +4c\left(
\left\Vert Y_{2}\right\Vert _{L^{2}}+c\right) +4c^{2}  \notag  \label{mm006a}
\\
&\leq &A^{\prime }+\left\Vert Y_{1}\right\Vert _{L^{2}}^{2}+\left\Vert
Y_{2}\right\Vert _{L^{2}}^{2}+6c\left( \left\Vert Y_{1}\right\Vert
_{L^{2}}+\left\Vert Y_{2}\right\Vert _{L^{2}}\right) +14c^{2}.  \label{mm006b}
\end{eqnarray}%
Since the gap is of size $\left[ b^{\alpha }\right], $%
\begin{equation}
\sup_{m\geq 1}\left\Vert Y_{gap}\right\Vert _{L^{2}}\leq \left[ b^{\alpha }%
\right] \max_{i\geq 1}\left\Vert X_{i}\right\Vert _{L^{2+2\delta }}\leq
b^{\alpha }\mu _{\delta }.  \label{mm006}
\end{equation}%
From (\ref{mm001}),  (\ref{mm006b}) and (\ref{mm006}) we obtain
\begin{eqnarray}
\left\Vert Y_{0}\right\Vert _{L^{2}}^{2} &\leq &A^{\prime }+\left\Vert
Y_{1}\right\Vert _{L^{2}}^{2}+\left\Vert Y_{2}\right\Vert
_{L^{2}}^{2}+b^{2\alpha }\mu _{\delta }^{2}  \notag \\
&&+2b^{\alpha }\mu _{\delta }\left( \left\Vert Y_{1}\right\Vert
_{L^{2}}+\left\Vert Y_{2}\right\Vert _{L^{2}}\right)  \notag \\
&&+6c\left( \left\Vert Y_{1}\right\Vert _{L^{2}}+\left\Vert Y_{2}\right\Vert
_{L^{2}}\right) +14c^{2}.  \label{mm009}
\end{eqnarray}%
Recall that $u_{a}=\sup_{m\geq 1}\left\Vert Y_{1}\right\Vert _{L^{2}}^{2}, $ $%
u_{b}=\sup_{m\geq 1}\left\Vert Y_{2}\right\Vert _{L^{2}}^{2}$ and $%
u_{a+b}=\sup_{m\geq 1}\left\Vert Y_{0}\right\Vert _{L^{2}}^{2}.$ Using (\ref%
{mm006}) we have $\left\Vert Y_{2}\right\Vert _{L^{2}}\leq
u_{b}^{1/2}+\left\Vert Y_{gap}\right\Vert _{L^{2}}\leq u_{b}^{1/2}+b^{\alpha
}\mu _{\delta }.$ From this and (\ref{mm009}) we deduce that%
\begin{eqnarray*}
u_{a+b} &\leq &A^{\prime }+u_{a}+u_{b}+b^{2\alpha }\mu _{\delta
}^{2}+2b^{\alpha }\mu _{\delta }u_{b}^{1/2} \\
&&+2b^{\alpha }\mu _{\delta }\left( u_{a}^{1/2}+u_{b}^{1/2}+b^{\alpha }\mu
_{\delta }\right) \\
&&+6c\left( u_{a}^{1/2}+u_{b}^{1/2}+b^{\alpha }\mu _{\delta }\right)over
+14c^{2}.
\end{eqnarray*}%
Rearranging the terms and taking into account that $A^{\prime }>1, $ we obtain%
\begin{eqnarray*}
u_{a+b} &\leq &A^{\prime }+14c^{2}+u_{a}+u_{b}+3b^{2\alpha }\mu _{\delta
}^{2} \\
&&+b^{\alpha }\mu _{\delta }\left( 6c+u_{a}^{1/2}+u_{b}^{1/2}\right)
+6cu_{a}^{1/2}+6cu_{b}^{1/2},
\end{eqnarray*}%
which proves the lemma.
\end{proof}

\begin{lemma}
\label{Lemma semiadd2}Assume that the sequence $\left( u_{n}\right) _{n\geq
1}$ is such that $u_{n}>0$ and
\begin{equation*}
u_{a+b}\leq \left( u_{a}+u_{b}+A\right) +\left( a^{2\alpha }+b^{2\alpha
}\right) B^{2}+\left( a^{\alpha }+b^{\alpha }\right) B\left(
u_{a}+u_{b}+A\right) ^{1/2}+cu_{a}^{1/2}+cu_{b}^{1/2},
\end{equation*}%
for all $a, $ $b\geq 1$ and some $A, B>0, $ $\alpha \in (0, \frac{1}{2}).$ Then
\begin{equation*}
u_{n}\leq c_{\alpha }\left( 1+u_{1}+A+B^{2}\right) n.
\end{equation*}
\end{lemma}

\begin{proof}
Note that $xy\leq \frac{1}{2}\left( \varepsilon x^{2}+\varepsilon
^{-1}y^{2}\right), $ for any $x, $ $y, $ $\varepsilon >0.$ Using the assumption
of the lemma,  we have%
\begin{eqnarray*}
u_{a+b} &\leq &\left( u_{a}+u_{b}+A\right) +\left( 1+\varepsilon
^{-1}\right) \left( a^{2\alpha }+b^{2\alpha }\right) B^{2} \\
&&+\frac{\varepsilon }{2}\left( u_{a}+u_{b}+A\right) +\varepsilon ^{-1}c^{2}+%
\frac{\varepsilon }{2}u_{\alpha }+\frac{\varepsilon }{2}u_{b} \\
&\leq &\left( 1+\varepsilon \right) \left( u_{a}+u_{b}+A\right) +\varepsilon
^{-1}c^{2}+\left( 1+\varepsilon ^{-1}\right) \left( a^{2\alpha }+b^{2\alpha
}\right) B^{2}.
\end{eqnarray*}%
Denote $v_{k}=\max_{1\leq n\leq 2^{k}}u_{n}, $ $k\geq 0.$ From the previous
equation it follows that%
\begin{equation*}
v_{k+1}\leq \left( 1+\varepsilon \right) \left( 2v_{k}+A\right) +\varepsilon
^{-1}c^{2}+\left( 1+\varepsilon ^{-1}\right) 2^{2\alpha k+1}B^{2}.
\end{equation*}%
Dividing by $\left( 2+2\varepsilon \right) ^{k+1}$ we get%
\begin{eqnarray*}
\frac{v_{k+1}}{\left( 2+2\varepsilon \right) ^{k+1}} &\leq &\frac{2v_{k}+A}{%
2\left( 2+2\varepsilon \right) ^{k}}+\left( 1+\varepsilon ^{-1}\right) \frac{%
2^{2\alpha k+1}}{\left( 2+2\varepsilon \right) ^{k+1}}\left(
B^{2}+c^{2}\right) \\
&\leq &\frac{v_{k}}{\left( 2+2\varepsilon \right) ^{k}}+\frac{A}{2\left(
2+2\varepsilon \right) ^{k}}+\left( 1+\varepsilon ^{-1}\right) \frac{%
2^{2\alpha k+1}}{\left( 2+2\varepsilon \right) ^{k+1}}\left(
B^{2}+c^{2}\right).
\end{eqnarray*}
Taking into account that $\alpha <\frac{1}{2}, $ by induction,  we obtain%
\begin{eqnarray*}
\frac{v_{k}}{\left( 2+2\varepsilon \right) ^{k}} &\leq &v_{0}+\frac{A}{2}%
\sum_{i=0}^{\infty }\frac{1}{\left( 2+2\varepsilon \right) ^{i}}+2\left(
1+\varepsilon ^{-1}\right) \left( B^{2}+c^{2}\right) \sum_{i=1}^{\infty }%
\frac{2^{2\alpha i}}{\left( 2+2\varepsilon \right) ^{i+1}} \\
&\leq &v_{0}+\frac{A}{2}\frac{2+2\varepsilon }{1+2\varepsilon }+2\left(
1+\varepsilon ^{-1}\right) \left( B^{2}+c^{2}\right) \frac{1+\varepsilon }{%
\varepsilon } \\
&\leq &v_{0}+c_{\varepsilon }^{\prime }\left( A+B^{2}+c^{2}\right),
\end{eqnarray*}%
where $c_{\varepsilon }^{\prime }$ depends only on $\varepsilon.$ This
implies that
\begin{equation*}
v_{k}=\max_{1\leq n\leq 2^{k}}u_{n}\leq C_{0}\left( 2+2\varepsilon \right)
^{k},
\end{equation*}%
where $C_{0}=\left( v_{0}+c_{\varepsilon }^{\prime }\left(
A+B^{2}+c^{2}\right) \right).$ Once again using the assumption of the lemma
it follows that%
\begin{eqnarray*}
& & v_{k+1} \\
&\leq &\left( 2v_{k}+A\right) +2^{2\alpha k+1}B^{2}+2^{\alpha
k+1}B\left( 2C_{0}\left( 2+2\varepsilon \right) ^{k}+A\right)
^{1/2}+2cC_{0}^{1/2}\left( 2+2\varepsilon \right) ^{k/2} \\
&\leq &\left( 2v_{k}+A\right) +2^{2\alpha k+1}B^{2}+2^{\alpha k+1}B\left(
2C_{0}^{1/2}\left( 2+2\varepsilon \right) ^{k/2}+A^{1/2}\right)
+2cC_{0}^{1/2}\left( 2+2\varepsilon \right) ^{k/2}.
\end{eqnarray*}%
Dividing by $2^{k+1}$ and choosing $\varepsilon =\varepsilon \left( \alpha
\right) $ so small that $2+2\varepsilon \leq 2^{1+\left( \frac{1}{2}-\alpha
\right) }, $ one gets%
\begin{eqnarray*}
\frac{v_{k+1}}{2^{k+1}} &\leq &\frac{v_{k}}{2^{k}}+\frac{A}{2^{k+1}}%
+2^{\left( 2\alpha -1\right) k+1}B^{2}+2^{\left( \alpha -1\right)
k+1}B\left( 2C_{0}^{1/2}2^{\frac{k}{2}+\frac{k}{2}\left( \frac{1}{2}-\alpha
\right) }+A^{1/2}\right) \\
&&+2^{-k}cC_{0}^{1/2}2^{\frac{k}{2}+\frac{k}{2}\left( \frac{1}{2}-\alpha
\right) } \\
&\leq &\frac{v_{k}}{2^{k}}+\frac{A}{2^{k+1}}+2^{\left( 2\alpha -1\right)
k+1}B^{2}+4BC_{0}^{1/2}2^{\left( \alpha -\frac{1}{2}\right) \frac{k}{2}%
}+2^{\left( \alpha -1\right) k+1}BA^{1/2} \\
&&+cC_{0}^{1/2}2^{-\frac{k}{2}\left( \frac{1}{2}+\alpha \right) }.
\end{eqnarray*}%
Using induction,  this implies%
\begin{equation*}
\frac{v_{k}}{2^{k}}\leq c_{\alpha }\left( A+C_{0}+B^{2}\right),
\end{equation*}%
since $\varepsilon $ depends only on $\alpha .$ From this we get $u_{2^{k}}\leq D2^{k}, $ for any $%
k\geq 1, $ where $D=c_{\alpha }^{\prime }\left( 1+v_{0}+A+B^{2}\right).$
Therefore,  for any $2^{k-1}\leq n\leq 2^{k}$ it holds $u_{n}\leq D2^{k}\leq
2D2^{k-1}\leq 2Dn.$ Lemma is proved.
\end{proof}

Let $\alpha =\frac{1}{4}.$ In the notations of Lemma \ref{Lemma semiadd1} $%
u_{1}\leq \max_{m\geq 1}\left\Vert X_{m}\right\Vert _{L^{2+2\delta
}}^{2}\leq \mu _{\delta }^{2}.$ From Lemmas \ref{Lemma semiadd1} and \ref%
{Lemma semiadd2} with $B=c\mu _{\delta }$ it follows,  for any $\epsilon \in
\left( 0, \delta \right), $
\begin{eqnarray*}
\max_{m\geq 1}\left\Vert \sum_{i=m}^{m+n-1}X_{i}\right\Vert _{L^{2}}^{2}
&=&u_{n}\leq c_{\alpha }\left( u_{1}+A+c^{2}\mu _{\delta }^{2}\right) n \\
&\leq &c\left( c_{\lambda _{1}, \lambda _{2}, \epsilon }\left( 1+\lambda
_{0}+\mu _{\delta }\right) ^{2+\frac{\epsilon }{1+\epsilon }}+2\mu _{\delta
}^{2}\right) n \\
&\leq &c_{\lambda _{1}, \lambda _{2}, \epsilon }^{\prime }\left( 1+\lambda
_{0}+\mu _{\delta }\right) ^{2+\epsilon }n,
\end{eqnarray*}%
which proves Proposition \ref{Proposition boundL2}.

\subsection{Proof of Proposition \protect\ref{Prop Lp bound}}

Let $m, n\in \mathbb{N}$ and $a=\left[ n^{1-\alpha }\right] $ and $b=\left[
n^{\alpha +\rho }\right], $ where $\alpha >0, $ and $\rho >0$ are such that $%
2\alpha +\rho <1.$ Note that $a>b$ and $ba\leq n^{1-\rho }.$ Consider the
intervals $I_{k}=[m+\left( k-1\right) a, m+ka-b), $ $J_{k}=[m+ka-b, m+ka), $ for
$k=1, ..., \left[ n^{\alpha }\right] $ and $I_{\text{fin}}=[m+ba, m+n), $ such
that $[m, m+n)=\cup _{k=1}^{\left[ n^{\alpha }\right] }\left( I_{k}\cup
J_{k}\right) \cup I_{\text{fin}}.$ Here $a-b>0$ and $b>0$ are interpreted as
the length of an island $I_{k}$ and the length of a gap $J_{k}$ respectively.

Denote $Y_{k}=\sum_{i\in I_{k}}X_{i}, $ $k=1, ..., \left[ n^{\alpha }\right].$
Let $V_{1}, ..., V_{\left[ n^{\alpha }\right] }$ be independent identically
distributed r.v.'s of mean $0, $ independent of $Y_{1}, ..., Y_{\left[
n^{\alpha }\right] }$ with a common characteristic function supported in the
interval $[-\varepsilon _{0}, \varepsilon _{0}], $ for some $\varepsilon
_{0}>0 $ and such that $\left\Vert V_{k}\right\Vert _{L^{2+2\delta }}\leq c, $
$k=1, ..., \left[ n^{\alpha }\right]. $
Denote $\widetilde{Y}_{k}=Y_{k}+V_{k}.$ Let $Z_{1}, ..., Z_{\left[ n^{\alpha }%
\right] }$ be independent copies of $\widetilde{Y}_{1}, ..., \widetilde{Y}_{%
\left[ n^{\alpha }\right] }.$ By Lemma \ref{properties 003}
\begin{equation}
\pi \left( \left( \widetilde{Y}_{1}, ..., \widetilde{Y}_{\left[ n^{\alpha }%
\right] }\right), \left( Z_{1}, ..., Z_{\left[ n^{\alpha }\right] }\right)
\right) \leq \sum_{k=1}^{\left[ n^{\alpha }\right] }\pi \left( \left(
\widetilde{Y}_{1}, ..., \widetilde{Y}_{k-1}, \widetilde{Y}_{k}\right), \left(
\widetilde{Y}_{1}, ..., \widetilde{Y}_{k-1}, , Z_{k}\right) \right).
\label{LpProof001}
\end{equation}%
Since $\mathbf{E}e^{itV_{k}}$ is supported in the interval $[-\varepsilon
_{0}, \varepsilon _{0}], $ by Lemma \ref{Lemma smoo},  for any $T>0$ and $k\leq %
\left[ n^{\alpha }\right], $%
\begin{eqnarray}
&&\pi \left( \left( \widetilde{Y}_{1}, ..., \widetilde{Y}_{k-1}, \widetilde{Y}%
_{k}\right), \left( \widetilde{Y}_{1}, ..., \widetilde{Y}_{k-1}, , Z_{k}\right)
\right)  \notag \\
&\leq &\frac{T}{\pi }\left( \int_{\left[ -\varepsilon _{0}, \varepsilon _{0}%
\right] ^{k}}\left\vert \phi \left( t, u\right) -\psi _{1}\left( t\right)
\psi _{2}\left( u\right) \right\vert ^{2}dtdu\right) ^{1/2}  \notag \\
&&+\mathbb{P}\left( \left\Vert \left( \widetilde{Y}_{1}, ..., \widetilde{Y}%
_{k-1}, \widetilde{Y}_{k}\right) \right\Vert _{\infty }>T\right),
\label{lp001}
\end{eqnarray}%
where $\phi $ is the characteristic function of $\left( \widetilde{Y}%
_{1}, ..., \widetilde{Y}_{k-1}, \widetilde{Y}_{k}\right) $ and $\psi _{1}$ and $%
\psi _{2}$ are the characteristic functions of the r.v.'s $\left( \widetilde{%
Y}_{1}, ..., \widetilde{Y}_{k-1}\right) $ and $\widetilde{Y}_{k}.$
Condition \textbf{C1} implies that%
\begin{eqnarray}
\left\vert \phi \left( t, u\right) -\psi _{1}\left( t\right) \psi _{2}\left(u\right) \right\vert
&\leq& \lambda _{0}\left( 1+a\right) ^{k}\exp \left(-\lambda _{1}b\right) \nonumber \\
&\leq& c_{\lambda _{1}}\lambda _{0}\left( 1+n^{1-\alpha}\right) ^{n^{\alpha }}\exp \left( -\lambda _{1}n^{\alpha +\rho }\right).
\label{lp002}
\end{eqnarray}
Let $T=e^{\frac{\lambda _{1}}{2}n^{\alpha +\rho }}.$ By Chebychev's
inequality,  taking into account that for $k\leq \left[ n^{\alpha }\right], $
we have%
\begin{eqnarray}
\mathbb{P}\left( \left\Vert \left( \widetilde{Y}_{1}, ..., \widetilde{Y}_{k-1}, %
\widetilde{Y}_{k}\right) \right\Vert _{\infty }>T\right) &\leq
&T^{-1}\sum_{i=1}^{k}\left\Vert Y_{i}+V_{i}\right\Vert _{L^{1}}  \notag \\
&\leq &T^{-1}\sum_{i=1}^{k}\left( \left\Vert Y_{i}\right\Vert _{L^{2+2\delta
}}+c\right)  \notag \\
&\leq &e^{-\frac{\lambda _{1}}{2}n^{\alpha +\rho }}\left[ n^{\alpha }\right]
\left( \left[ n^{1-\alpha }\right] \mu _{\delta }+c\right)  \notag \\
&\leq &cne^{-\frac{\lambda _{1}}{2}n^{\alpha +\rho }}\left( 1+\mu _{\delta
}\right).  \label{lp003}
\end{eqnarray}%
From (\ref{lp001}),  (\ref{lp002}) and (\ref{lp003}) we obtain%
\begin{eqnarray*}
&&\pi \left( \left( \widetilde{Y}_{1}, ..., \widetilde{Y}_{k-1}, \widetilde{Y}%
_{k}\right), \left( \widetilde{Y}_{1}, ..., \widetilde{Y}_{k-1}, , Z_{k}\right)
\right) \\
&\leq &\frac{T}{\pi }\lambda _{0}\varepsilon _{0}^{n}\left( 1+n^{1-\alpha
}\right) ^{n^{\alpha }}e^{-\lambda _{1}n^{\alpha +\rho }}+c_{\alpha }ne^{-%
\frac{\lambda _{1}}{2}n^{\alpha +\rho }}\left( 1+\mu _{\delta }\right) \\
&\leq &c\varepsilon _{0}^{n}n\left( 1+n^{1-\alpha }\right) ^{n^{\alpha }}e^{-%
\frac{\lambda _{1}}{2}n^{\alpha +\rho }}\left( 1+\lambda _{0}+\mu _{\delta
}\right).
\end{eqnarray*}%
Implementing this bound in (\ref%
{LpProof001}) we get%
\begin{eqnarray*}
\pi \left( \left( \widetilde{Y}_{1}, ..., \widetilde{Y}_{\left[ n^{\alpha }%
\right] }\right), \left( Z_{1}, ..., Z_{\left[ n^{\alpha }\right] }\right)
\right) &\leq &c\varepsilon _{0}^{n}n^{1+\alpha }\left( 1+n^{1-\alpha
}\right) ^{n^{\alpha }}e^{-\frac{\lambda _{1}}{2}n^{\alpha +\rho }}\left(
1+\lambda _{0}+\mu _{\delta }\right) \\
&\leq &\Delta =c_{\alpha, \lambda _{1}}e^{-\frac{\lambda _{1}}{4}n^{\alpha
+\rho }}\left( 1+\lambda _{0}+\mu _{\delta }\right).
\end{eqnarray*}%
According to Strassen-Dudley's theorem (see Lemma \ref{lemma strassen-dudley}) there is a coupling of $\left(
\widetilde{Y}_{1}, ..., \widetilde{Y}_{b}\right) $ and $\left(
Z_{1}, ..., Z_{b}\right) $ such that%
\begin{equation*}
P\left( \left\Vert \left( \widetilde{Y}_{1}, ..., \widetilde{Y}_{\left[
n^{\alpha }\right] }\right) -\left( Z_{1}, ..., Z_{\left[ n^{\alpha }\right]
}\right) \right\Vert _{\infty }\geq \Delta \right) \leq \Delta.
\end{equation*}%
Let $S=\widetilde{Y}_{1}+...+\widetilde{Y}_{\left[ n^{\alpha }\right]
}-\left( Z_{1}+...+Z_{\left[ n^{\alpha }\right] }\right).$ Taking into
account that $\left\Vert V_{i}\right\Vert _{L^{2+2\delta }}\leq c, $ we have
\begin{eqnarray}
\left\Vert S\right\Vert _{L^{2+2\delta }} &=&\left\Vert \widetilde{Y}%
_{1}+...+\widetilde{Y}_{\left[ n^{\alpha }\right] }-\left( Z_{1}+...+Z_{%
\left[ n^{\alpha }\right] }\right) \right\Vert _{L^{2+2\delta }}  \notag \\
&\leq &cn^{\alpha }a\left( 1+\max_{l\geq 1}\left\Vert X_{l}\right\Vert
_{L^{2+2\delta }}\right) \leq c^{\prime }n\left( 1+\mu _{\delta }\right).
\label{LpProof002a}
\end{eqnarray}%
Let $\eta \in \left( 0, \delta -\delta ^{\prime }\right), $ $p=2+2\delta
^{\prime }, $ $p^{\prime }=p+2\eta \leq 2+2\delta $ and $\gamma =\gamma
\left( \eta \right) =\frac{2\eta }{p\left( p+2\eta \right) }.$ By H\"{o}%
lder's inequality, %
\begin{equation*}
\left\Vert S^{2+2\delta ^{\prime }}1\left( \left\vert S\right\vert \geq
n^{\alpha }\Delta \right) \right\Vert _{L^{2+2\delta ^{\prime }}}\leq
\left\Vert S\right\Vert _{L^{p^{\prime }}}\left( \mathbb{P}\left( \left\vert
S\right\vert \geq n^{\alpha }\Delta \right) \right) ^{\gamma }\leq
\left\Vert S\right\Vert _{L^{2+2\delta }}\left( \mathbb{P}\left( \left\vert
S\right\vert \geq n^{\alpha }\Delta \right) \right) ^{\gamma }.
\end{equation*}%
Using the bound $\left\vert S\right\vert \leq n^{\alpha }\left\Vert \left(
\widetilde{Y}_{1}, ..., \widetilde{Y}_{\left[ n^{\alpha }\right] }\right)
-\left( Z_{1}, ..., Z_{\left[ n^{\alpha }\right] }\right) \right\Vert _{\infty
}$ and we have%
\begin{eqnarray*}
\left\Vert S\right\Vert _{L^{2+2\delta ^{\prime }}} &\leq &n^{\alpha }\Delta
+\left\Vert S^{2+2\delta ^{\prime }}1\left( \left\vert S\right\vert \geq
n^{\alpha }\Delta \right) \right\Vert _{L^{2+2\delta ^{\prime }}} \\
&\leq &n^{\alpha }\Delta +\left\Vert S\right\Vert _{L^{2+2\delta }}\left(
\mathbb{P}\left( \left\vert S\right\vert \geq n^{\alpha }\Delta \right)
\right) ^{\gamma } \\
&\leq &n^{\alpha }\Delta +\left\Vert S\right\Vert _{L^{2+2\delta }}\left(
\mathbb{P}\left( \left\Vert \left( \widetilde{Y}_{1}, ..., \widetilde{Y}_{%
\left[ n^{\alpha }\right] }\right) -\left( Z_{1}, ..., Z_{\left[ n^{\alpha }%
\right] }\right) \right\Vert _{\infty }\geq \Delta \right) \right) ^{\gamma }
\\
&\leq &n^{\alpha }\Delta +cn\left( 1+\mu _{\delta }\right) \Delta ^{\gamma }.
\end{eqnarray*}%
Taking into account the definition of $\Delta, $%
\begin{eqnarray}
\left\Vert S\right\Vert _{L^{2+2\delta ^{\prime }}} &\leq &n^{\alpha
}c_{\alpha, \lambda _{1}}e^{-\frac{\lambda _{1}}{4}n^{\alpha +\rho }}\left(
1+\lambda _{0}+\mu _{\delta }\right)  \notag \\
&&+cn\left( 1+\mu _{\delta }\right) \left( c_{\alpha, \lambda _{1}}e^{-\frac{%
\lambda _{1}}{4}n^{\alpha +\rho }}\left( 1+\lambda _{0}+\mu _{\delta
}\right) \right) ^{\gamma }  \notag \\
&\leq &A^{\prime }=c_{\lambda _{1}, \lambda _{2}, \gamma, \alpha, \rho
}^{\prime }\left( 1+\lambda _{0}+\mu _{\delta }\right) ^{1+\gamma }.
\label{LpProof003}
\end{eqnarray}
From (\ref{LpProof002a}) and (\ref{LpProof003}),  it follows that%
\begin{eqnarray}
\left\Vert \widetilde{Y}_{1}+...+\widetilde{Y}_{\left[ n^{\alpha }\right]
}\right\Vert _{L^{2+2\delta ^{\prime }}} &\leq &\left\Vert S\right\Vert
_{L^{2+2\delta ^{\prime }}}+\left\Vert Z_{1}+...+Z_{\left[ n^{\alpha }\right]
}\right\Vert _{L^{2+2\delta ^{\prime }}}  \notag \\
&\leq &A^{\prime }+\left\Vert Z_{1}+...+Z_{\left[ n^{\alpha }\right]
}\right\Vert _{L^{2+2\delta ^{\prime }}}.  \label{LpProof004}
\end{eqnarray}%
Since the r.v.'s $Z_{1}, ..., Z_{\left[ n^{\alpha }\right] }$ are independent,
by Rosenthal's inequality (Theorem 3 in \cite{Rosenth}),  there exists some constant $%
c_{\delta ^{\prime }} $ such that%
\begin{equation}
\left\Vert Z_{1}+...+Z_{\left[ n^{\alpha }\right] }\right\Vert
_{L^{2+2\delta ^{\prime }}}\leq c_{\delta ^{\prime }}\left( \sum_{i=1}^{%
\left[ n^{\alpha }\right] }\mathbb{E}Z_{i}^{2}\right) ^{\frac{1}{2}%
}+c_{\delta ^{\prime }}\left( \sum_{i=1}^{\left[ n^{\alpha }\right] }\mathbb{%
E}\left\vert Z_{i}\right\vert ^{2+2\delta ^{\prime }}\right) ^{\frac{1}{%
2+2\delta ^{\prime }}}.  \label{LpProof005a}
\end{equation}%
Taking into account that $Y_{i}=\sum_{j\in I_{i}}X_{j}$ and that $\left\vert
I_{i}\right\vert \leq a-b\leq n^{1-\alpha }, $ by Proposition \ref%
{Proposition boundL2} we have%
\begin{equation}
\mathbf{E}Z_{i}^{2}=\left\Vert \widetilde{Y}_{i}\right\Vert _{L^{2}}^{2}\leq
\left( c+\left\Vert Y_{i}\right\Vert _{L^{2}}\right) ^{2}\leq c_{\lambda
_{1}, \lambda _{2}, \gamma }^{\prime }\left( 1+\lambda _{0}+\mu _{\delta
}\right) ^{2+\gamma }n^{1-\alpha }.  \label{LpProof005b}
\end{equation}%
Note also that $\left\Vert Z_{i}\right\Vert _{L^{2+2\delta ^{\prime }}}\leq
v_{a-b}+c, $ where $v_{n}=\sup_{m\geq 1}\left\Vert
\sum_{i=m}^{m+n-1}X_{i}\right\Vert _{L^{2+2\delta ^{\prime }}}.$  Therefore,  from (\ref{LpProof005a}) and (\ref%
{LpProof005b}),  it follows that%
\begin{eqnarray*}
&&\left\Vert Z_{1}+...+Z_{b}\right\Vert _{L^{2+2\delta ^{\prime }}} \\
&\leq &A^{\prime }+c_{\lambda _{1}, \lambda _{2}, \gamma, \alpha, \rho, \delta
^{\prime }}\left( 1+\lambda _{0}+\mu _{\delta }\right) ^{1+\frac{\gamma }{2}%
}n^{\frac{1}{2}}+c_{\delta ^{\prime }}\left( \sum_{i=1}^{\left[ n^{\alpha }%
\right] }\left( v_{a-b}+c\right) ^{2+2\delta ^{\prime }}\right) ^{\frac{1}{%
2+2\delta ^{\prime }}} \\
&\leq &c_{\lambda _{1}, \lambda _{2}, \gamma, \alpha, \rho, \delta ^{\prime
}}\left( 1+\lambda _{0}+\mu _{\delta }\right) ^{1+\gamma }n^{\frac{1}{2}%
}+c_{\delta ^{\prime }}v_{a-b}n^{\frac{\alpha }{2+2\delta ^{\prime }}}.
\end{eqnarray*}%
Using (\ref{LpProof004}),  we get%
\begin{eqnarray*}
\left\Vert \widetilde{Y}_{1}+...+\widetilde{Y}_{\left[ n^{\alpha }\right]
}\right\Vert _{L^{2+2\delta ^{\prime }}} &\leq &A^{\prime }+\left\Vert
Z_{1}+...+Z_{\left[ n^{\alpha }\right] }\right\Vert _{L^{2+2\delta ^{\prime
}}} \\
&\leq &c_{\lambda _{1}, \lambda _{2}, \gamma, \alpha, \rho, \delta ^{\prime
}}^{\prime }\left( 1+\lambda _{0}+\mu _{\delta }\right) ^{1+\gamma }n^{\frac{%
1}{2}}+c_{\delta ^{\prime }}v_{a-b}n^{\frac{\alpha }{2+2\delta ^{\prime }}}.
\end{eqnarray*}%
Since $\widetilde{Y}_{k}=Y_{k}+V_{k}$ and $\left\Vert V_{k}\right\Vert
_{L^{2+2\delta }}\leq c, $ it holds%
\begin{eqnarray*}
\left\Vert Y_{1}+...+Y_{\left[ n^{\alpha }\right] }\right\Vert
_{L^{2+2\delta ^{\prime }}} &\leq &c\left[ n^{\alpha }\right] +\left\Vert
\widetilde{Y}_{1}+...+\widetilde{Y}_{\left[ n^{\alpha }\right] }\right\Vert
_{L^{2+2\delta ^{\prime }}} \\
&\leq &cn^{\alpha }+c_{\lambda _{1}, \lambda _{2}, \gamma, \alpha, \rho
, \delta ^{\prime }}\left( 1+\lambda _{0}+\mu _{\delta }\right) ^{1+\gamma
}n^{\frac{1}{2}}+c_{\delta ^{\prime }}v_{a-b}n^{\frac{\alpha }{2+2\delta
^{\prime }}} \\
&\leq &c_{\lambda _{1}, \lambda _{2}, \gamma, \alpha, \rho, \delta ^{\prime
}}\left( 1+\lambda _{0}+\mu _{\delta }\right) ^{1+\gamma }n^{\frac{1}{2}%
}+c_{\delta ^{\prime }}v_{a-b}n^{\frac{\alpha }{2+2\delta ^{\prime }}},
\end{eqnarray*}%
where for the last line we use the fact that $\alpha <\frac{1-\rho }{2}<%
\frac{1}{2}.$ Filling up the gaps in the final interval $I_{\text{fin}}$,  we
get%
\begin{eqnarray*}
\left\Vert \sum_{i=1}^{m+n-1}X_{i}\right\Vert _{L^{2+2\delta ^{\prime }}} &\leq &\left\Vert Y_{1}+...+Y_{\left[ n^{\alpha }\right] }\right\Vert_{L^{2+2\delta ^{\prime }}}\\
&&+\sum_{k=1}^{\left[ n^{\alpha }\right]}\sum_{i\in J_{k}}\left\Vert X_{i}\right\Vert _{L^{2+2\delta ^{\prime}}}+\left\Vert \sum_{i=m+a\left[ n^{\alpha }\right] }^{m+n}X_{i}\right\Vert
_{L^{2+2\delta ^{\prime }}} \\
&\leq &c_{\lambda _{1}, \lambda _{2}, \gamma, \alpha, \rho, \delta ^{\prime
}}\left( 1+\lambda _{0}+\mu _{\delta }\right) ^{1+\gamma }n^{\frac{1}{2}%
}+c_{\delta ^{\prime }}v_{a-b}n^{\frac{\alpha }{2+2\delta ^{\prime }}} \\
&&+n^{2\alpha +\rho }\mu _{\delta }+v_{n-\left[ n^{1-\alpha }\right] \left[
n^{\alpha }\right] }.
\end{eqnarray*}%
From this,  we deduce the inequality%
\begin{equation}
v_{n}\leq c_{\lambda _{1}, \lambda _{2}, \gamma, \alpha, \rho, \delta ^{\prime
}}\left( 1+\lambda _{0}+\mu _{\delta }\right) ^{1+\gamma }n^{\frac{1}{2}%
}+n^{2\alpha +\rho }\mu _{\delta }+c_{\delta ^{\prime }}v_{\left[
n^{1-\alpha }\right] -\left[ n^{\alpha +\rho }\right] }n^{\frac{\alpha }{%
2+2\delta ^{\prime }}}+v_{n-\left[ n^{1-\alpha }\right] \left[ n^{\alpha }%
\right] }.  \label{LpProof006a}
\end{equation}%
Denote $\overline{v}_{n}=\frac{v_{n}}{\left( 1+\lambda _{0}+\mu _{\delta
}\right) ^{1+\gamma }}.$ Then from (\ref{LpProof006a}),  it follows that%
\begin{equation*}
\overline{v}_{n}\leq c_{\lambda _{1}, \lambda _{2}, \gamma, \alpha, \rho
, \delta ^{\prime }}n^{\frac{1}{2}}+n^{2\alpha +\rho }+c_{\delta ^{\prime }}%
\overline{v}_{\left[ n^{1-\alpha }\right] -\left[ n^{\alpha +\rho }\right]
}n^{\frac{\alpha }{2+2\delta ^{\prime }}}+\overline{v}_{n-\left[ n^{1-\alpha
}\right] \left[ n^{\alpha }\right] }.
\end{equation*}%
Fixing $\alpha =\frac{1}{6}$ and $\rho =\frac{1}{6}, $ we get
\begin{equation}
\overline{v}_{n}\leq c_{\lambda _{1}, \lambda _{2}, \gamma, \delta ^{\prime
}}n^{\frac{1}{2}}+c_{\delta ^{\prime }}\overline{v}_{\left[ n^{\frac{5}{6}}%
\right] -\left[ n^{\frac{1}{3}}\right] }n^{\frac{1}{6}\frac{1}{2+2\delta
^{\prime }}}+\overline{v}_{n-\left[ n^{\frac{5}{6}}\right] \left[ n^{\frac{1%
}{6}}\right] }.  \label{LpProof006}
\end{equation}%
We start with the inequality $\overline{v}_{n}\leq n^{q_{0}}, $ where $%
q_{0}=1.$ Since $n-\left[ n^{1-\alpha }\right] \left[ n^{\alpha }\right]
\leq cn^{1-\alpha }, $ we have $v_{n-\left[ n^{1-\alpha }\right] \left[
n^{\alpha }\right] }\leq cn^{\frac{5}{6}q_{0}}$ and $v_{\left[ n^{1-\alpha }%
\right] -\left[ n^{\alpha +\rho }\right] }\leq cn^{\frac{5}{6}q_{0}}.$
Implementing this in (\ref{LpProof006}) gives, %
\begin{equation*}
\overline{v}_{n}\leq c_{\lambda _{1}, \lambda _{2}, \gamma, \delta ^{\prime
}}n^{\frac{1}{2}}+c_{\delta ^{\prime }}n^{\frac{5}{6}q_{0}+\frac{1}{6}\frac{1%
}{2+2\delta ^{\prime }}}\leq c_{1r_{1}, \gamma, \delta ^{\prime }}n^{\max
\left\{ \frac{1}{2}, q_{1}\right\} },
\end{equation*}%
where $q_{1}=\frac{5}{6}q_{0}+\frac{1}{6}\frac{1}{2+2\delta ^{\prime }}.$
Continuing in the same way,  at the iteration $k+1, $ we obtain
\begin{equation*}
\overline{v}_{n}\leq c_{kr_{1}, \gamma, \delta ^{\prime }}n^{\max \left\{
\frac{1}{2}, q_{k+1}\right\} },
\end{equation*}%
where $q_{k+1}=\frac{5}{6}q_{k}+\frac{1}{6}\frac{1}{2+2\delta ^{\prime }}.$
Since $\lim_{k\rightarrow \infty }q_{k}=\frac{1}{2+2\delta ^{\prime }}, $
there exists a constant $k_{0}<\infty, $ such that $q_{k_{0}+1}\leq \frac{1}{%
2}.$ With this $k_{0}, $ we get%
\begin{equation*}
\overline{v}_{n}\leq c_{k_{0}\lambda _{1}, \lambda _{2}, \gamma, \delta
^{\prime }}n^{\frac{1}{2}}.
\end{equation*}%
Since $\gamma =\gamma \left( \eta \right) =\frac{2\eta }{p\left( p+2\eta
\right) }\leq \frac{2\eta }{p^{2}}, $ we have,  for any $m\geq 1, $%
\begin{equation*}
\left\Vert \sum_{i=m}^{m+n}X_{i}\right\Vert _{L^{2+2\delta ^{\prime }}}\leq
c_{\lambda _{1}, \lambda _{2}, \eta, \delta ^{\prime }}\left( 1+\lambda
_{0}+\mu _{\delta }\right) ^{1+\gamma }n^{\frac{1}{2}}\leq c_{\lambda
_{1}, \lambda _{2}, \delta ^{\prime }, \eta }\left( 1+\lambda _{0}+\mu _{\delta
}\right) ^{1+\frac{2\eta }{p^{2}}}n^{\frac{1}{2}}.
\end{equation*}%
Since $\eta $ is arbitrary we obtain the assertion of Proposition \ref{Prop
Lp bound}.

\section{Appendix\label{sec Appendix}}

\subsection{Some general bounds for the Prokhorov distance\label{sec
Prokhorov dist}}

Let $\left( E, d\right) $ be a metric space endowed with the metric $d$ and $%
\mathcal{E}$ be the Borel $\sigma $-algebra on $E.$ For any $B\in \mathcal{E}
$ denote by $B^{\varepsilon }$ its $\varepsilon $-extension: $B^{\varepsilon
}=\left\{ x\in E:d\left( x, E\right) \leq \varepsilon \right\}.$ Let $\pi
\left( \mathbf{P}, \mathbf{Q}\right) $ be the Prokhorov distance between two
probability measures $\mathbf{P}$ and $\mathbf{Q}$ defined by%
\begin{equation*}
\pi \left( \mathbf{P}, \mathbf{Q}\right) =\inf \left\{ \varepsilon
:\sup_{B\in \mathcal{E}}\left\vert \mathbf{P}\left( B\right) -\mathbf{Q}%
\left( B^{\varepsilon }\right) \right\vert \leq \varepsilon \right\}.
\end{equation*}%
The following assertion is known as Strassen-Dudley theorem and is a
consequence of the results in Strassen \cite{strassen1965} (see also Dudley
\cite{Dudley1968}). Let $\mathcal{P}_{E}\left( \mathbf{P}, \mathbf{Q}\right) $
be the set of probability measures on $E\times E$ with given marginals $%
\mathbf{P}$ and $\mathbf{Q}.$ Denote by $\mathcal{D}_{E, d}\left( \varepsilon
\right) $ the $\varepsilon $-extension of the diagonal in $E\times E, $ i.e. $%
\mathcal{D}_{E, d}\left( \varepsilon \right) =\left\{ \left( s, s^{\prime
}\right) \in E\times E:d\left( s, s^{\prime }\right) \leq \varepsilon
\right\} $ and by $\overline{\mathcal{D}}_{E, d}\left( \varepsilon \right) $
its complement.

\begin{lemma}
\label{lemma strassen-dudley}(Strassen-Dudley) If $\left( E, d\right) $ is a
complete separable metric space,  then%
\begin{equation*}
\pi \left( \mathbf{P}, \mathbf{Q}\right) =\min \left\{ \varepsilon :\exists
\text{ }\mathbb{P}\in \mathcal{P}_{E}\left( \mathbf{P}, \mathbf{Q}\right)
\text{ such that\ }\mathbb{P}\left( \overline{\mathcal{D}}_{E, d}\left(
\varepsilon \right) \right) \leq \varepsilon \text{ }\right\}.
\end{equation*}
\end{lemma}

Let $\left(E_{1}, d_{1}\right) $ and $\left(E_{2}, d_{2}\right) $ be two
complete separable metric spaces endowed with Borel $\sigma $-algebras $%
\mathcal{E}_{1}$ and $\mathcal{E}_{2}$ respectively. Consider the product
space $E=E_{1}\times E_{2}$ endowed with the metric $d\left( x, y\right)
=\max \left\{ d_{1}\left( x_{1}, y_{1}\right), d_{2}\left( x_{2}, y_{2}\right)
\right\}, $ where $x=\left( x_{1}, x_{2}\right), y=\left( y_{1}, y_{2}\right)
\in E.$ Denote by $\mathcal{E}$ the Borel $\sigma $-algebra on $E.$

\begin{lemma}
\label{proj lemma 002}Consider the r.v.'s $X, Y\in E_{1}$ and $Z\in E_{2}$
such that $Z$ and $\left( X, Y\right) $ are independent. Then%
\begin{equation*}
\pi \left( \mathcal{L}_{X, Z}, \mathcal{L}_{Y, Z}\right) =\pi \left( \mathcal{L}%
_{X}, \mathcal{L}_{Y}\right).
\end{equation*}
\end{lemma}

\begin{proof}
Let $\mathbb{P}_{1}\in \mathcal{P}_{E_{1}}\left( \mathcal{L}_{X}, \mathcal{L}%
_{Y}\right) $ and $\mathbb{P}_{2}\in \mathcal{P}_{E_{2}}\left( \mathcal{L}%
_{Z}, \mathcal{L}_{Z}\right).$ If $\mathbb{P}_{2}$ is concentrated on the
diagonal of $E_{2}\times E_{2}, $ then with $\mathbb{P}=\mathbb{P}_{1}\times
\mathbb{P}_{2}$ we have $\mathbb{P}\left( \overline{\mathcal{D}}_{E, d}\left(
\varepsilon \right) \right) =\mathbb{P}_{1}\left( \overline{\mathcal{D}}%
_{E_{1}, d_{1}}\left( \varepsilon \right) \right).$ This means that%
\begin{eqnarray*}
&&A=\left\{ \varepsilon :\exists \text{ }\mathbb{P}_{1}\in \mathcal{P}%
_{E}\left( \mathcal{L}_{X}, \mathcal{L}_{Y}\right) \text{ such that }\mathbb{P%
}_{1}\left( \overline{\mathcal{D}}_{E_{1}, d_{1}}\left( \varepsilon \right)
\right) \leq \varepsilon \right\} \\
&=&\left\{ \varepsilon :\exists \text{ }\mathbb{P}\in \mathcal{P}_{E}\left(
\mathcal{L}_{X, Z}, \mathcal{L}_{Y, Z}\right) \text{ such that }\mathbb{P}%
\left( \overline{\mathcal{D}}_{E, d}\left( \varepsilon \right) \right) \leq
\varepsilon \right\} =B.
\end{eqnarray*}%
By Lemma \ref{lemma strassen-dudley},  $\pi \left( \mathcal{L}_{X}, \mathcal{L}%
_{Y}\right) =\inf A=\inf B=\pi \left( \mathcal{L}_{X, Z}, \mathcal{L}%
_{Y, Z}\right).$
\end{proof}

Let $\left( E_{1}, d_{1}\right), ..., \left( E_{n}, d_{n}\right) $ be complete
separable metric spaces. On the product space $E=E_{1}\times...\times E_{n}$
consider the metric $d\left( x, y\right) =\max \left\{ d_{1}\left(
x_{1}, y_{1}\right), ..., d_{n}\left( x_{n}, y_{n}\right) \right\}, $ where $%
x=\left( x_{1}, ..., x_{n}\right), y=\left( y_{1}, ..., y_{n}\right) \in E.$

\begin{lemma}
\label{properties 003}Consider the r.v.'s $X=\left( X_{1}, ..., X_{n}\right)
\in E$ and $Y=\left( Y_{1}, ..., Y_{n}\right) \in E.$ If $X$ and $Y$ are
independent and $Y_{1}, ..., Y_{n}$ are independent,  then%
\begin{equation*}
\pi \left( \mathcal{L}_{X_{1}, ..., X_{n}}, \mathcal{L}_{Y_{1}, ..., Y_{n}}%
\right) \leq \sum_{k=1}^{n}\pi \left( \mathcal{L}_{X_{1}, ..., X_{k-1}, X_{k}}, %
\mathcal{L}_{X_{1}, ..., X_{k-1}, Y_{k}}\right).
\end{equation*}
\end{lemma}

\begin{proof}
The assertion of the lemma is obtained using the telescope rule and Lemma %
\ref{proj lemma 002}.
\end{proof}

\begin{lemma}
\label{properties 004}Consider the r.v.'s $X=\left( X_{1}, ..., X_{n}\right)
\in E$ and $Y=\left( Y_{1}, ..., Y_{n}\right) \in E.$ If $\left(
X_{1}, Y_{1}\right)..., \left( X_{n}, Y_{n}\right) $ are independent,  then%
\begin{equation*}
\pi \left( \mathcal{L}_{X_{1}, ..., X_{n}}, \mathcal{L}_{Y_{1}, ..., Y_{n}}%
\right) \leq \sum_{k=1}^{n}\pi \left( \mathcal{L}_{X_{k}}, \mathcal{L}%
_{Y_{k}}\right).
\end{equation*}
\end{lemma}

\begin{proof}
The assertion of the lemma is obtained using Lemma \ref{properties 003} and
Lemma \ref{proj lemma 002}.
\end{proof}

The following lemma is taken from \cite{Gouez}.

\begin{lemma}
\label{Lemma smoo}Let $\mathbf{P}$ and $\mathbf{Q}$ be two probability
measures on $\left( \mathbb{R}^{N}, \mathcal{B}^{N}\right).$
Assume that the characteristic functions
$\widehat{p}\left( t\right) $ and $\widehat{q}\left( t\right) $
pertaining to $\mathbf{P}$ and $\mathbf{Q}$
are square integrable with respect to the Lebesgue measure in $\mathbb R ^N$. Then
\begin{equation}
\pi \left( \mathbf{P}, \mathbf{Q}\right) \leq \left( T/\pi \right)
^{N/2}\left( \int_{\mathbb{R}^{N}}\left\vert \widehat{p}\left( t\right) -%
\widehat{q}\left( t\right) \right\vert ^{2}dt\right) ^{1/2}+\mathbf{P}\left(
\left\Vert x\right\Vert _{\infty }>T\right).  \label{smoo 001}
\end{equation}
\end{lemma}

\begin{proof}
Assume first that $\mathbf{P}, $ $\mathbf{Q}$ have square integrable
densities $p$ and $q$ respectively. Denote $C_{T}=\left\{ x
\in \mathbb{R}^{N}:\left\Vert x\right\Vert _{\infty }\leq T\right\} $ and $%
B_{T}=\mathbb{R}^{N}\smallsetminus C_{T}.$ Assume that $A\in \mathcal{B}^{N}$
and let $\varepsilon >0.$ Then
\begin{eqnarray*}
&&\left\vert \mathbf{P}\left( A\right) -\mathbf{Q}\left( A^{\varepsilon
}\right) \right\vert \\
&=&\left\vert \mathbf{P}\left( A^{\varepsilon }\cap C_{T}\right) +\mathbf{P}%
\left( A^{\varepsilon }\cap B_{T}\right) -\mathbf{Q}\left( A^{\varepsilon
}\right) \right\vert \\
&\leq &\left\vert \mathbf{P}\left( A^{\varepsilon }\cap C_{T}\right) -%
\mathbf{Q}\left( A^{\varepsilon }\cap C_{T}\right) \right\vert +\mathbf{P}%
\left( B_{T}\right) \\
&=&\left\vert \int_{A^{\varepsilon }\cap C_{T}}\left( p\left( x\right)
-q\left( x\right) \right) dx\right\vert +\mathbf{P}\left( B_{T}\right) \\
&\leq &\int_{\mathbb{R}^{N}}\left\vert p\left( x\right) -q\left( x\right)
\right\vert 1_{C_{T}}\left( x\right) dx+\mathbf{P}\left( \left\Vert
x\right\Vert _{\infty }>T\right).
\end{eqnarray*}%
Using H\"{o}lder's inequality,  we get%
\begin{eqnarray*}
\pi \left( \mathbf{P}, \mathbf{Q}\right) &\leq &\left\vert \mathbf{P}\left(
A\right) -\mathbf{Q}\left( A^{\varepsilon }\right) \right\vert \\
&\leq &\left( \int_{\mathbb{R}^{N}}\left\vert p\left( x\right) -q\left(
x\right) \right\vert ^{2}dx\right) ^{1/2}\left( \int_{\mathbb{R}%
^{N}}1_{C_{T}}\left( x\right) dx\right) ^{1/2}+\mathbf{P}\left( \left\Vert
x\right\Vert _{\infty }>T\right).
\end{eqnarray*}%
Since,  by Plancherel's identity%
\begin{equation*}
\int_{\mathbb{R}^{N}}\left\vert p\left( x\right) -q\left( x\right)
\right\vert ^{2}dx=\left( 2\pi \right) ^{-N}\int_{\mathbb{R}^{N}}\left\vert
\widehat{p}\left( t\right) -\widehat{q}\left( t\right) \right\vert ^{2}dt,
\end{equation*}%
we obtain (\ref{smoo 001}) for $\mathbf{P}$ and $\mathbf{Q}$ having square
integrable densities.

If $\mathbf{P}$ and $\mathbf{Q}$ do not have densities,  denote by $\mathbf{P}%
_{v}=\mathbf{P}\ast \mathbf{G}_{v}$ and $\mathbf{Q}_{v}=\mathbf{Q}\ast
\mathbf{G}_{v}$ the smoothed versions of $\mathbf{P}$ and $\mathbf{Q}, $
where $\mathbf{G}_{v}$ is the normal distribution of mean $0$ and variance $%
v^{2}.$ Using (\ref{smoo 001}) and the obvious inequality $\left\vert
\widehat{p}_{v}\left( t\right) -\widehat{q}_{v}\left( t\right) \right\vert
\leq \left\vert \widehat{p}\left( t\right) -%
\widehat{q}\left( t\right) \right\vert, $ we obtain%
\begin{equation*}
\pi \left( \mathbf{P}_{v}, \mathbf{Q}_{v}\right) \leq \left( 2\pi \right)
^{-N/2}\left( 2T\right) ^{N/2}\left( \int_{\mathbb{R}^{N}}\left\vert
\widehat{p}\left( t\right) -\widehat{q}\left( t\right) \right\vert
^{2}dt\right) ^{1/2}+\mathbf{P}_{v}\left( \left\Vert x\right\Vert _{\infty
}>T\right).
\end{equation*}
Since $\pi \left( \mathbf{P}_{v}, \mathbf{P}\right) \rightarrow 0\;$and $\pi
\left( \mathbf{Q}_{v}, \mathbf{Q}\right) \rightarrow 0$ it follows that $\pi
\left( \mathbf{P}_{v}, \mathbf{Q}_{v}\right) \rightarrow \pi \left( \mathbf{P}%
, \mathbf{Q}\right) $ as $v\rightarrow 0.$ Note also that $\lim
\sup_{v\rightarrow 0}\mathbf{P}_{v}\left( \left\Vert x\right\Vert _{\infty
}>T\right) \leq \mathbf{P}\left( \left\Vert x\right\Vert _{\infty }\geq
T\right).$ The inequality (\ref{smoo 001}) follows for arbitrary $\mathbf{P}%
, \mathbf{Q.}$
\end{proof}

\subsection{Coupling independent and Gaussian r.v.'s\label{subsectionSakh}}

The following result is proved in Theorem 5 of Sakhanenko \cite{Sakh85} (see
also \cite{Sakh83},  \cite{Sakh89},  \cite{Sakh00} for related results). Let $%
X_{1}, ..., X_{n}$ be a sequence of independent r.v.'s satisfying $\mathbb{E}%
X_{i}=0$ and $\mathbb{E}\left\vert X_{i}\right\vert ^{p}<\infty $ for some $%
p\geq 2$ and all $1\leq i\leq n.$

\begin{theorem}
\label{Th-KMT001}On some probability space $\left( \Omega ^{\prime }, %
\mathcal{F}^{\prime }, \mathbb{P}^{\prime }\right) $ there is a sequence of
independent normal r.v.'s $Y_{1}, ..., Y_{n}$ satisfying $\mathbb{E}^{\prime
}Y_{i}=0$ and $\mathbb{E}^{\prime }Y_{i}^{2}=\mathbb{E}X_{i}^{2}, $ $1\leq
i\leq n, $ and a sequence of independent r.v.'s $X_{1}^{\prime
}, ..., X_{n}^{\prime }$ satisfying $X_{i}^{\prime }\overset{d}{=}X_{i}, $ $%
1\leq i\leq n, $ such that%
\begin{equation*}
\mathbb{E}^{\prime }\left( \max_{1\leq k\leq n}\left\vert
\sum_{i=1}^{k}X_{i}^{\prime }-\sum_{i=1}^{k}Y_{i}\right\vert \right)
^{p}\leq c_{p}\sum_{i=1}^{n}\mathbb{E}\left\vert X_{i}\right\vert ^{p}.
\end{equation*}%
\end{theorem}

In particular,  by Chebyshev's inequality,  for the same construction and any $%
a>0$ it holds%
\begin{equation}
\mathbb{P}^{\prime }\left( \max_{1\leq k\leq n}\left\vert
\sum_{i=1}^{k}X_{i}^{\prime }-\sum_{i=1}^{k}Y_{i}\right\vert >a\right) \leq
\frac{c_{p}}{a^{p}}\sum_{i=1}^{n}\mathbb{E}\left\vert X_{i}\right\vert ^{p}.
\label{SakhLp}
\end{equation}

\vskip10mm

\end{document}